%% file: LAST_VERSION.tex
\documentclass{article}
\RequirePackage{etex} 
\usepackage[T1]{fontenc}
\usepackage[utf8]{inputenc}
\usepackage{mymacros}
\usepackage{authblk}
\usepackage[a4paper,total={6.5in, 10in}]{geometry}
\usepackage{pstricks}
\usepackage{graphicx} 
\usepackage{ragged2e}
\usepackage{pifont}
\usepackage{lipsum}
\usepackage{sidecap}
\usepackage{autonum}
\usepackage{subcaption}
\usepackage[margin=2cm]{caption}
\usepackage{stackengine}

\usepackage{tgheros}
\input{dibujos.tex}

\theoremstyle{plain}
\newtheorem{theorem}{Theorem}[section]
\newtheorem{lemma}[theorem]{Lemma}
\newtheorem{proposition}[theorem]{Proposition}
\newtheorem{corollary}[theorem]{Corollary}

\newtheorem{conjecture}[theorem]{Conjecture}
\newtheorem{remark}[theorem]{Remark}

\numberwithin{equation}{section}

\theoremstyle{definition}
\newtheorem{definition}[theorem]{Definition}

\theoremstyle{remark}

\newcommand{\retainlabel}[1]{\label{#1}\sbox0{\ref{#1}}}

\renewcommand{\tilde}{\widetilde}
\renewcommand{\bar}{\overline}

\newcommand{\bbN}{\mathbb{N}}

\newcommand{\bbZ}{\mathbb{Z}}

\newcommand{\mcN}{\mathcal{N}}
\newcommand{\mcS}{\mathcal{S}}
\newcommand{\mcE}{\mathcal{E}}
\newcommand{\mcW}{\mathcal{W}}

\newcommand{\undH}{{\underline{\mathbf{H}}}}

\newcommand{\Hhat}{ \hat{\undH} }

\newcommand{\bfH}{\mathbf{H}}
\newcommand{\bfD}{\mathbf{D}}

\newcommand{\bfU}{\mathbf{U}}

\newcommand{\bfN}{\mathbf{N}}

\newcommand{\geqh}{\overset{\bfH}{\geq}}

\newcommand{\starcup}{$\sqcup$\kern-0.58em{$\star$}}

\newcommand{\front}{\partial}

\DeclareMathOperator{\Supp}{Supp}

\definecolor{Mred}{RGB}{236,56,36}
\definecolor{Mgreen}{RGB}{72,220,38}
\definecolor{Mblue}{RGB}{45,30,200}
\definecolor{Darkgrey}{RGB}{74,74,74}
\definecolor{Mediumgrey}{RGB}{128,128,128}
\definecolor{Lightgrey}{RGB}{155,155,155}
\definecolor{Mbrown}{RGB}{139,87,42}
\definecolor{Mpurple}{RGB}{189,15,224}

\definecolor{mygray1}{gray}{0.1}
\definecolor{mygray2}{gray}{0.2}
\definecolor{mygray3}{gray}{0.3}
\definecolor{mygray4}{gray}{0.4}
\definecolor{mygray5}{gray}{0.5}
\definecolor{mygray6}{gray}{0.6}
\definecolor{mygray7}{gray}{0.7}
\definecolor{mygray8}{gray}{0.8}
\definecolor{mygray9}{gray}{0.9}

\definecolor{cb-black}      {RGB}{  0,   0,   0}
\definecolor{cb-blue-green} {RGB}{  0,  073,  073}
\definecolor{cb-green-sea}  {RGB}{  0, 146, 146}
\definecolor{cb-rose}       {RGB}{255, 109, 182}
\definecolor{cb-salmon-pink}{RGB}{255, 182, 119}
\definecolor{cb-purple}     {RGB}{ 73,   0, 146}
\definecolor{cb-blue}       {RGB}{ 0, 109, 219}
\definecolor{cb-lilac}      {RGB}{182, 109, 255}
\definecolor{cb-blue-sky}   {RGB}{109, 182, 255}
\definecolor{cb-blue-light} {RGB}{182, 219, 255}
\definecolor{cb-burgundy}   {RGB}{146,   0,   0}
\definecolor{cb-brown}      {RGB}{146,  73,   0}
\definecolor{cb-clay}       {RGB}{219, 209,   0}
\definecolor{cb-green-lime} {RGB}{ 36, 255,  36}
\definecolor{cb-yellow}     {RGB}{255, 255, 109}
\definecolor{cb-dark-blue}  {RGB}{28, 51, 167}

\author{Karina Batistelli{\thanks{Universidad de Chile }},  Aram Bingham{\thanks{Tulane University of Louisiana}} ,  and  David Plaza{\thanks{Universidad de Talca}}   } 

\title{Kazhdan-Lusztig polynomials for $\tilde{B}_2$}
\date{\vspace{-5ex}}

\begin{document}
\maketitle

\begin{abstract}
\justify
In their seminal paper \cite{kazhdan1979representations} Kazhdan and Lusztig define, for an arbitrary Coxeter system $(W,S)$, a family of polynomials indexed by pairs of elements of $W$.  Despite their relevance and  elementary definition, the explicit computation of these polynomials is still one of the hardest open problems in algebraic combinatorics. In this paper we  explicitly compute Kazhdan-Lusztig polynomials for a Coxeter system of type $\tilde{B}_2$. 
\end{abstract}

\section{Introduction}
\justify
Kazhdan-Lusztig polynomials lie at the intersection of representation theory, geometry  and algebraic combinatorics. Their relevance derives from the fact that they elegantly express the answers to many difficult problems in the aforementioned areas. Despite their straightforward definition (through a recursive algorithm involving only elementary operations) and enormous efforts made by several authors over the past forty years, in general the explicit computation of Kazhdan-Lusztig polynomials remains elusive.

In this paper we compute Kazhdan-Lusztig polynomials in type $\tilde{B}_2$; for relevant background and related theory we refer the reader to \cite{elias2020introduction}. In order to present our results we introduce some notation. The main protagonist in this paper is the affine Weyl group of type $\tilde{B}_2$, which we denote by $W$.   It is a Coxeter group generated by involutions $S=\{s_1,s_2, s_0\}$ with relations $(s_1s_2)^4=(s_2s_0)^4=(s_1s_0)^2=1$.  As usual, we denote by $\ell (\cdot)$ and $\leq$ the length and Bruhat order on $W$, respectively.  It is convenient to recall the realization of $W$ as the group of isometric transformations of the plane generated by reflections in the lines that support the sides of an isosceles right triangle as is illustrated in Figure \ref{fig: geometric realization of W}. In this figure the triangle marked with a dot represents the identity of $W$. The colors on the edges of the triangles represent the following simple reflections: red is $s_1$, blue is $s_2$ and green is $s_0$. Given an element  $w \in W$ and an expression $s_{i_1}s_{i_2} \ldots s_{i_k}$ (not necessarily reduced) of $w$ we define a sequence of triangles $(\triangle_0,\triangle_1, \ldots \triangle_k)$ as follows: $\triangle_0$ is the identity triangle. Then $\triangle_{j+1} $ is obtained from $\triangle_j$ by reflecting it through the side colored with $s_{i_j}$. We identify $w$ with  $\triangle_k$. Of course, this identification does not depend on the choice of an expression for $w$. Henceforth, we do not distinguish between elements of $W$ and triangles.

We split $W-\{1\}$ into three regions: The \emph{big region} (the region colored using the lightest gray), the \emph{thick region} (the region colored by the darkest gray) and the \emph{thin region} (the region colored with the intermediate gray).  The reader familiar with  cells in Kazhdan-Lusztig theory might have noticed the small discrepancy between Figure \ref{fig: geometric realization of W} and \cite[Figure 2]{lusztig1985cells}. The big/thick/thin regions coincide very nearly with two-sided cells, though we will prefer our description in order to keep this paper self-contained. 

\begin{figure}
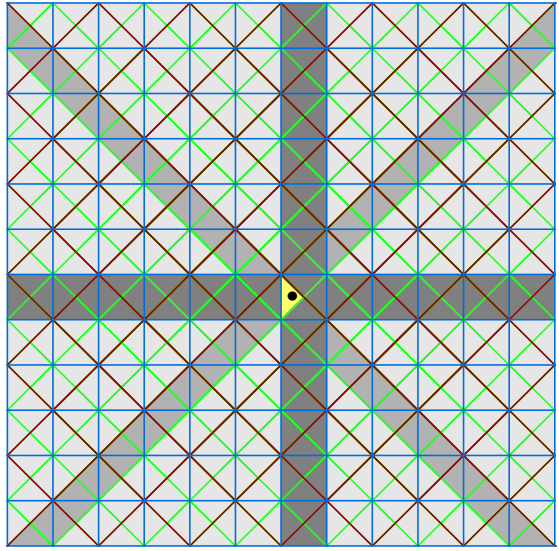

    \centering
   \scalebox{.6}{ \IntTess }
    \caption{Geometric realization of $W$.}
    \label{fig: geometric realization of W}
\end{figure}

Let $\mathcal{H}$ be the Hecke algebra of $W$ with standard basis $\{ \bfH_{w}  \}_{w\in W}$ and  Kazhdan-Lusztig basis (canonical basis)  $\{ \undH_{w}  \}_{w\in W}$. Kazhdan-Lusztig polynomials $\{h_{x,w}(v) \mid x,w\in W  \}$ are defined by the equality
\begin{equation}
    \undH_{w} = \sum_{x\in W}  h_{x,w} (v) \bfH_{x}.
\end{equation}

There is a group automorphism $\varphi:W\rightarrow W$ which interchanges $s_0$ and $s_1$ and fixes $s_2$. This extends to a Hecke algebra automorphism which we also denote by $\varphi$.  We have $\undH_{\varphi (w)}= \varphi(\undH_{w}) $ for all $w\in W$. To condense notation we often write $w':=\varphi (w)$, for $w\in W$. 

In this paper we prove an explicit formula for $\undH_{w}$ for all $w\in W$ located in either the big region or the thick region, and we conjecture explicit formulas for the thin region. The remainder of this introduction is devoted to explaining these formulas.

We begin by considering  elements in the big region. Let $a=s_1s_2s_1$, $b=s_0s_2s_0$, $c=s_1s_2s_0s_2$, and $d=s_0s_2s_1s_2$. Let $\mathbb{N}$ be the set of non-negative integers.  For $(m,n)\in \mathbb{N}^2$ we define
\begin{equation} \label{eq defi thetas}
\theta(m,n)= 
\begin{cases}
{(ab)^{k} \, a} \, s_2 \, d^n, & \text{ if $m=2k$; }\\
(ab)^{k+1} \, s_2 \, c^n,  & \text{ if $m= 2k+1$. }
\end{cases}
\end{equation}
We define $t_{m} = s_0$ for $m $ even and $t_m=s_1$ for $m$ odd.   Then (modulo $\varphi$)  all the elements in the big region are of the form  $x\theta(m,n)y$, where $x\in \{ 1, s_0,s_2s_0,s_1s_2s_0  \}$ and $y \in \{ 1, t_m, t_ms_2,t_ms_2t_{m}'   \}$. 

For $(m,n)\in \bbN^2$ we define 
\begin{equation}
    \Supp (m,n) = \left\{ \begin{array}{ll}
      \{  (m-2i,n-j) \in \bbN^2 \, |\,  i,j\in \bbN \},     & \mbox{if } m \mbox{ is odd;}  \\
       \{  (m-2i,n-j) \in \bbN^2 \, |\,  i,j\in \bbN \} -  \{ (0,b) \, \ \, | \,  b\not \equiv n \bmod 2  \},  & \mbox{if } m \mbox{ is even. } 
    \end{array}   \right.
\end{equation}
For $w\in W$ we define
\begin{equation}
    \bfN_w=\sum_{x\leq w} v^{l(w)-l(x)}\bfH_x.
\end{equation}
\begin{theorem}  \label{theorem intro big}
Let  $(m,n)\in \bbN^2$. Then, we have 
\begin{equation}
\undH_{\theta (m,n)} = \sum_{(a,b)\in \Supp (m,n)} v^{(m-a)+2(n-b)} \bfN_{\theta (a,b)}.
\end{equation}
Furthermore, if $x\in \{ 1, s_0,s_2s_0,s_1s_2s_0  \}$ and $y \in \{ 1, t_m, t_ms_2,t_ms_2t_{m}'   \}$ then $\undH_{x\theta(m,n)y} = X\undH_{\theta(m,n)} Y$, where 
\[
	X= \left\{  \begin{array}{ll}
	        1,             	& \mbox{if } x=1; \\
	\undH_{s_0} ,       & \mbox{if } x=s_0;\\
	\undH_{s_2}\undH_{s_0}-1	       ,             	& \mbox{if } x=s_2s_0; \\
\undH_{s_1}	\undH_{s_2}\undH_{s_0}-\undH_{s_1} - \undH_{s_0}	 ,       & \mbox{if } x=s_1s_2s_0;\\
	\end{array}       \right.
   Y = \left\{  
\begin{array}{ll}
1,	&  \mbox{if } y=1 ;\\
\undH_{t_m}  &  \mbox{if } y=t_m;\\
 \undH_{t_m} \undH_{s_2}  -1 & \mbox{if } y=t_ms_2;\\
\undH_{t_m} \undH_{s_2} \undH_{t_m'}  -\undH_{t_m} -\undH_{t_m'}   &  \mbox{if } y=t_ms_2t_m' .
\end{array}	
	 \right. 
\]
\end{theorem}

We now move to the thick region.  We denote by $\mathcal{N}$ (resp. $\mathcal S$, $\mathcal E$ and $\mathcal W$) the sub-region of the thick region formed by triangles located to the north (resp. south, east and west) of the identity triangle.
Consider the infinite sequence 
\begin{equation} \label{eq:xnseq}
 \{ a_n\}_{n=1}^\infty =(s_1,s_2,s_1,s_0,s_2,s_0,s_1,s_2,s_1,s_0,s_2,s_0,\ldots).   
\end{equation}
We define $x_n=a_1\cdots a_n$ and $\bar{x}_n=s_1s_2s_0x_{n-3}$. Then,  $ \mcN = \{ x_n \mid n\geq 1 \} \cup  \{ \bar{x}_{3n} \mid n\geq 1 \}$. In what follows we refer to $\mcN$ as the north wall.  We notice that $\mathcal{S}=\varphi(\mcN)$. We define $e_n=s_1x_{n}'$. Then, 
$$\mcE =\{e_{n} \mid n\geq 1   \} \cup \{ e_{3k}' \mid k\geq 1 \}.$$
Finally, we define $w_n=s_2e_n$. Then, we have 
$$\mcW =\{ w_n \mid n\geq 1  \}  \cup \{ w_{3k}' \mid k\geq 1  \} \cup \{ s_2,s_2s_1,s_2s_0 \}.$$
This completes the description of the elements in the thick region.

The following theorem provides formulas for Kazhdan-Lusztig basis elements indexed by elements located in $\mathcal{N}$. As we already pointed out $\mathcal{S}=\varphi(\mcN)$. Therefore,  the formulas for elements located in $\mathcal{S}$ are obtained by applying $\varphi$ to the formulas in the theorem. For the sake of brevity, in this introduction we omit the formulas for the elements located in $\mathcal{E}$  and $\mathcal{W}$. These formulas are presented in \S\ref{section east and west}.

\begin{theorem}
For all $k\geq 2$ we have
  \begin{equation}  \label{closed formula for x 3k +1 intro}
     \undH_{x_{f(k)}}= \bfN_{x_{f(k)}}+v\bfN_{x_{f(k-1)}} + \lt(\sum_{j=2}^{k-1} v^{j-1}(\bfN_{e_{f(k-j)}}+\bfN_{u_{f(k-j)}})\rt)
 +v^{k-1}\bfN_{s_1s_0},
 \end{equation}
 where $f(k):=3k+1$ and $u_n=s_2x_n$. Furthermore, using the convention of Remark~\ref{definition negative index is set equal to zero} we have
\begin{align}
 \label{recurrence thick wall north A intro}     \undH_{x_{3k+1}}\undH_{s_2}  & = \undH_{x_{3k+2}} , \\ 
 \label{recurrence thick wall north B intro}     \undH_{x_{3k+2}} \undH_{s_0} & = \left\{  \begin{array}{ll}
     \undH_{\bar{x}_{3k+3} }+ \undH_{x_{3k+1}}+\undH_{\theta(k-1,0)} +\undH_{s_1\theta'(k-2,0)}+\undH_{\theta(k-3,0)}    , &  \mbox{if } k \mbox{ is even;}  \\
     \undH_{x_{3k+3}} + \undH_{x_{3k+1}} +\undH_{s_1s_2s_0\theta(k-2,0) }    , &  \mbox{if } k \mbox{ is odd.}
     \end{array}  \right.  \\
\label{recurrence thick wall north C intro}    \undH_{x_{3k+2}} \undH_{s_1}   & = \left\{  \begin{array}{ll}
 \undH_{x_{3k+3}} + \undH_{x_{3k+1}} + \undH_{s_1s_2s_0\theta(k-2,0)} ,  &  \mbox{if } k \mbox{ is even;}  \\
      \undH_{\bar{x}_{3k+3} }+\undH_{x_{3k+1}}+\undH_{\theta(k-1,0)} +\undH_{s_1\theta'(k-2,0)}+\undH_{\theta(k-3,0)}        , &  \mbox{if } k \mbox{ is odd.}
     \end{array}  \right.   
\end{align}
\end{theorem}
Conjectural formulas for $\undH_{w}$ for $w$ located in the thin region are presented in \S\ref{section thin region}.  It is likely that they could be proved by the methods used to obtain formulas in the other regions, though the accounting becomes more difficult due to the number and type of terms appearing in the explicit formulas. We leave this task for a future investigation.

The alert reader might have noticed that the results presented in this introduction provide formulas for Kazhdan-Lusztig basis elements rather than Kazhdan-Lusztig polynomials. By the nature of our formulas, in order to compute Kazhdan-Lusztig polynomials we need to understand the elements $\bfN_w$, or, equivalently, to understand the sets $\lessdot w := \{ x\leq  w \mid  x\in W \}$. This is exactly the content of \S\ref{section lower intervals complete}.  We will finish this introduction with an example that illustrates how our formulas can be used in order to efficiently compute Kazhdan-Lusztig polynomials. 

Suppose we want to compute $h_{\bar x_{3n}, x_{3m}}(v)$ for integers $m>n \geq 2$, $m$ odd and $n$ even.  We notice that taking the coefficient of $ \bfH_{\bar x_{3n}} $ on both sides of \eqref{recurrence thick wall north C intro} and using \eqref{equation y greater than ys} we obtain
\begin{equation}  \label{example intro A}
    v^{-1}h_{\bar x_{3n},x_{3m-1}} (v)+ h_{x_{3n-1},x_{3m-1}}  (v) = h_{\bar x_{3n},x_{3m}}(v) + h_{\bar x_{3n}, x_{3m-2}}(v) +h_{\bar x_{3n}, s_1s_2s_0\theta( m-3,0) }(v).
\end{equation}
On the other hand, \eqref{recurrence thick wall north A intro} and \eqref{equation y greater than ys}  imply 
\begin{align}
\label{example intro B}  h_{\bar x_{3n},x_{3m-1}} (v) & = v h_{\bar x_{3n},x_{3m-2}} (v) +   h_{s_1s_2s_0\theta (n-2,0),x_{3m-2}} (v),\\ 
\label{example intro C}  h_{x_{3n-1},x_{3m-1}}  (v)   & = v^{-1} h_{x_{3n-1},x_{3m-2}}  (v) + h_{x_{3n-2},x_{3m-2}}  (v) .
\end{align}
For an integer $l$ we define    $F_{l}(v)=\displaystyle\sum_{i=0}^{l-1} v^{2i}$.

We can use \eqref{closed formula for x 3k +1 intro} and the description of lower intervals in \S\ref{section lower intervals complete} to obtain
\begin{align}
      h_{\bar x_{3n},x_{3m-2}} (v) &  =  v^{m-n}   \left(  F_{m-n}(v) +F_{m-n-2}(v)  \right),\\  
      h_{s_1s_2s_0\theta (n-2,0),x_{3m-2}} (v)    &  = v^{m-n+1} \left(  F_{m-n-1}(v) +F_{m-n-3}(v)  \right),\\
      h_{x_{3n-1},x_{3m-2}} (v)    &  =  v^{m-n+1} \left(  F_{m-n}(v) +F_{m-n-2}(v)  \right),\\
      h_{x_{3n-2},x_{3m-2}}  (v)   &  =   v^{m-n}  \left(  F_{m-n+1}(v) +F_{m-n-1}(v)  \right).
\end{align}
Similarly, using Theorem \ref{theorem intro big} (or Lemma \ref{lemma explicit for 120 theta })  we obtain
\begin{equation}
 h_{\bar x_{3n}, s_1s_2s_0\theta( m-3,0) }(v) = v^{m-n}  \left(  F_{m-n}(v) +F_{m-n-2}(v)  \right).  
\end{equation}
Putting all these together we obtain
\begin{equation}\label{closed formula intro}
    h_{\bar x_{3n}, x_{3m}}(v) =  v^{m-n} (F_{m-n+1}(v) + 2F_{m-n-1}(v)+F_{m-n-3}(v) ).
\end{equation}
Closed formulas for $h_{x,w}(v)$, such as \eqref{closed formula intro}, can be obtained using the results in this paper for all $x\in W$ and for all $w\in W$ located either in the big region or the thick region. We have chosen this specific example in order to point out a certain discrepancy between an identity appearing in \cite[\textbf{1.}(a)]{lusztig1997nonlocal} and our results. Concretely, in that paper it is claimed that\footnote{We have slightly modified the original statement in order to match the conventions in this paper.} 
\begin{equation}\label{eq to be corrected}
    \mu (\bar x_{3n}, x_{3m}) = 1,
\end{equation}
for all $m > n > 0$, $m$ odd, $n$ even, where $\mu(x,w)$ denotes the coefficient of $v$ in the Kazhdan-Lusztig polynomial $h_{x,w}(v)$\footnote{We remark that throughout this paper we follow Soergel's normalization in \cite{soergel1997kazhdan}.}.   However, it is clear from \eqref{closed formula intro} that   $\mu (\bar x_{3n}, x_{3m} ) =0 $ if $m-n>1$. Equation \eqref{eq to be corrected} allows the author to conclude that the $W$-graph of an affine Weyl group of type $\tilde{B}_2$ is not locally finite. This conclusion was later shown to be true by Wang \cite{wang2011kazhdan} through a different set of examples which are consistent with our results.

This paper is organized as follows. In Section 2 we describe geometrically the lower intervals in the Bruhat order of the elements of the big region and the thick region. We specify the region corresponding to each lower interval inductively and obtain closed formulas for their sizes. 
In Section 3 we introduce the Hecke algebra of $W$ and prove several technical lemmas containing identities which will be necessary for the computations that follow. Our main results lie in Sections 4 and 5, which provide explicit formulas for the Kazhdan-Lusztig basis elements of elements in the big region and the thick region respectively. Finally, in Section 6 we conjecture the explicit formulas of the Kazhdan-Lusztig basis elements corresponding to the elements of the thin region.

\subsection*{Acknowledgements}
The first author is supported by Fondecyt project 3190144. The third author was partially supported by Fondecyt project 1200341. The authors would like to thank Nicolas Libedinsky for useful discussions and all the participants of the "Soergel bimodules learning seminar", which took place at the Universidad de Chile during 2020, and from which this project and collaboration was born. We also wish to acknowledge the contributors and developers of SageMath \cite{sagemath} of which we have made indispensable use.

\section{Lower intervals in the Bruhat order}\label{section lower intervals complete}

\justify

Throughout this paper $W$ denotes the affine Weyl group of type $\tilde{B}_2$ with  generators $S=\{s_0,s_1,s_2\}$. In this section we provide a geometric description for lower Bruhat intervals and obtain formulas for their size. Given $w \in W$ we define $   \leq w := \{ x\in W \mid x\leq w  \}$ and  $\ab{w} := \ab{\leq w}$. We also define $D_R(w):=\{s\in S \mid ws <w \}$ and $D_L(w) = D_{R}(w^{-1})$. 

\subsection{Lower intervals  for the big region} \label{section lower intervals}

We recall from \eqref{eq defi thetas} the definition of the elements $\theta(m,n)$. We notice that $D_L(\theta(m,n))=\{s_1,s_2\}$ and $D_R(\theta(m,n))= \{s_2,t_m\}$. In particular, $D_R(\theta(m,n))$ only depends on the parity of $m$. Left and right descent sets characterize $\theta$-elements in the following sense: If $x\in W$ satisfies $D_L(x)=\{s_1,s_2\}$ and $D_R(x)=\{s_2, t \}$ for some $t\in \{s_0,s_1\}$, then there exists $(m,n)\in \mathbb{N}^2$  such that $x=\theta (m,n)$. 

All these $\theta$-elements belong to the same connected component of the big region. We denote this connected component by  $\mathfrak{C}$. All of the elements in $ \mathfrak{C}$ are of the form $\theta(m,n)x $ for some $x \in \{1,t_m,t_ms_2,t_ms_2t_m' \}$. 
\begin{lemma} \label{lemma less than theta m0}
The set $\leq \theta(m,0)$ is the square $S(m,0)$ defined as the smallest square satisfying the following properties: 
\begin{itemize}
    \item $S(m,0)$ contains  $\theta(m,0)$.
    \item The center of $S(m,0)$, say  $O$, is the upper vertex of the identity triangle. 
    \item The two diagonals of $S(m,0)$ are the horizontal and vertical lines passing through $O$.
\end{itemize}
\end{lemma}

\begin{proof}
We proceed by induction on $m$. For  $m=0$ the claim is clear.  We now assume the lemma holds for some $m \geq 0$.  Note that $\theta(m+1,0) = \theta(m,0) t_ms_2t_m$. Let $w \leq \theta(m+1,0) $. Then $w=xy$ for some $x\leq \theta(m,0) $ and $y\leq  t_ms_2t_m
$. By our inductive hypothesis we know that $x\in S(m,0)$. It follows that $w\in S(m+1,0)$. Therefore, $\leq \theta(m+1,0) \subseteq S(m+1,0) $. Conversely, let $w \in S(m+1,0)$. If $w \in S(m,0)$ we are done, so we can assume that $w\not \in S(m,0)$. It is easy to see that $w$ can be obtained from some $x \in S(m,0)$ by reflecting it along one of the following sequences: $t_m$, $t_ms_{2}$ or $t_ms_{2}t_m$. Once again, our inductive hypothesis guarantees that $w \leq \theta (m+1,0)$. Thus $S(m+1,0) \subseteq (\leq \theta (m+1,0))$.
\end{proof}

Any square  obtained from $S(0,0)$  by a translation will be called  a $B_2$-square.  Let $m,n\in \mathbb N$. We recursively define sets $S(m,n)$ as follows: If  $n=0$ then $S(m,0)$ is defined as in Lemma \ref{lemma less than theta m0}. Assume that $S(m,n)$ has been defined. Then $S(m,n+1)$ is defined as the set obtained from $S(m,n)$ by surrounding its boundary with $B_2$-squares. This is illustrated in Figure \ref{fig:some lower intervals thetas}. Starting from $S(4,0)$ (the black square), we obtain $S(4,1)$ by  adding $20$ $B_2$-squares (in dark yellow). Then, we obtain $S(4,2)$ by adding $24 $  $B_2$-squares (in light yellow).   

\begin{figure}
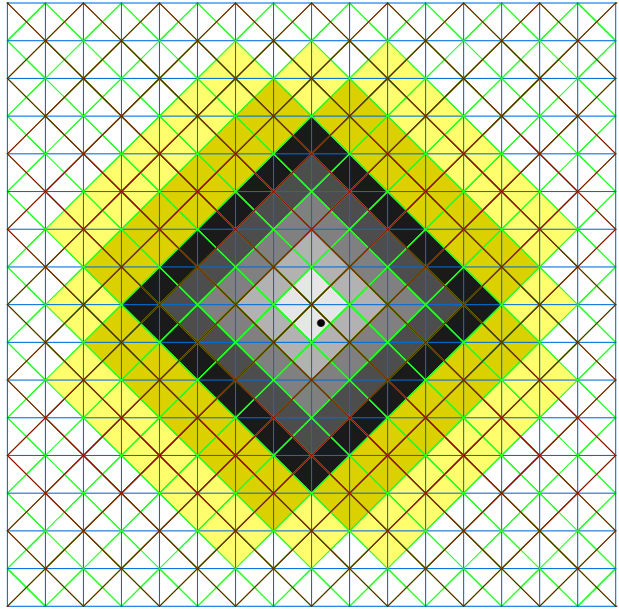

    \centering
    \scalebox{.5}{  \dibujothetas }
    \caption{Lower intervals: From lightest to darkest gray $\theta(0,0)$--$\theta(4,0)$. Darkest yellow $\theta(4,1)$ and lightest yellow $\theta(4,2)$.  }
    \label{fig:some lower intervals thetas}
\end{figure}

\begin{definition}
Let $x\in W$. We define $\front (x)$ to be the set formed by all the elements $ y \leq x$ such that the triangle associated to $y$ has a side belonging to the boundary of $\leq x$. We partition $\front (x)$ into three (possibly empty) sets according to the color of the relevant side. More precisely, we define $\front_{s_{k}} (x)$ to be the set of all $y\in \front (x)  $ such that the side of $y$ belonging to the boundary of $\leq x$ is colored by $s_k$. 
\end{definition}

\begin{definition} \label{def:XY}
For any $X,Y \subseteq W$ we define $XY:=\{ xy\mid x\in X, \, y \in Y \}.$ In case $X=\{x\}$, we may also write $x(Y)$ to mean the same.
\end{definition}

\begin{lemma}  \label{lemma less than theta mn}
For all $m,n\in \mathbb N$ we have $\leq \theta (m,n)  = S(m,n)$. 
\end{lemma}
\begin{proof}
We proceed by induction on $n$. The case $n=0$ is covered by Lemma \ref{lemma less than theta m0}. Assume the lemma holds for some fixed $n$.  We have 
$   \theta(m,n+1) = \theta (m,n) t_ms_{2}t_m's_{2}  $
and this implies 
\begin{equation} \label{decomposition theta m.n+1}
\leq \theta(m,n+1)= (\leq \theta(m,n)) (\leq t_ms_2t_m's_2).    
\end{equation}
From this it is easy to see, using our inductive hypothesis, that $ (\leq \theta(m,n+1))   \subseteq S(m,n+1) $.\\ Conversely, we have 
\begin{equation}
    S(m,n+1)=S(m,n)\cup \bigcup_{x \in X}   \front (\theta(m,n) )x, 
\end{equation}
where $X=\{t_m, t_ms_{2} ,t_ms_{2}t_m', t_ms_{2}t_m's_2  \}$. It follows from our inductive hypothesis and \eqref{decomposition theta m.n+1} that $S(m,n+1)\subseteq (\leq \theta(m,n+1) )$.
We conclude that $S(m,n+1)= (\leq \theta(m,n+1) )$ as desired.
\end{proof}

\begin{corollary}\label{lem:abx}
For all $m,n\in \mathbb N$ we have $\ab{\theta(m,n)}=8(m^2+2m+4mn+2n^2+2n+1)$. In particular, we have 
$\ab{\theta(m,0)}=8(m+1)^2$.
\end{corollary}

\begin{proof}
The equality follows easily from Lemma \ref{lemma less than theta m0} and Lemma \ref{lemma less than theta mn} by a counting argument.  Indeed, $\leq \theta(m,0)$ is made of $(m+1)^2$ (non-intersecting) $B_2$-squares, which leads to $\ab{\leq \theta(m,0)} = 8 (m+1)^2$. Furthermore,  the number of $B_2$-squares that one needs to add to $\leq \theta(m,0)$ in order to obtain $\leq \theta(m,n)$ is $4(m+1)n +2n(n-1)$.
\end{proof}

\begin{figure}[ht]
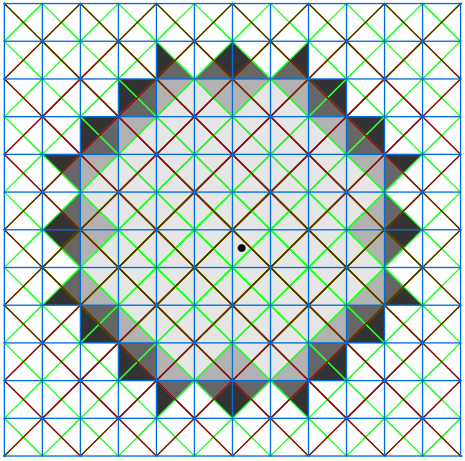
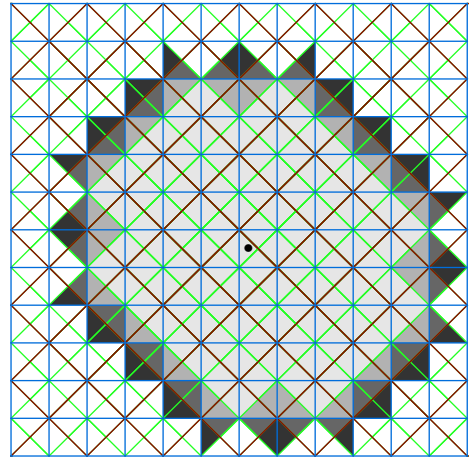
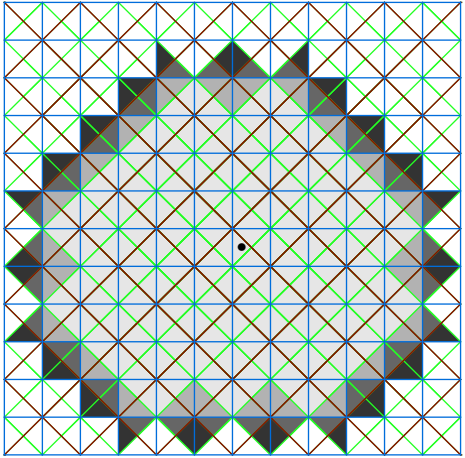
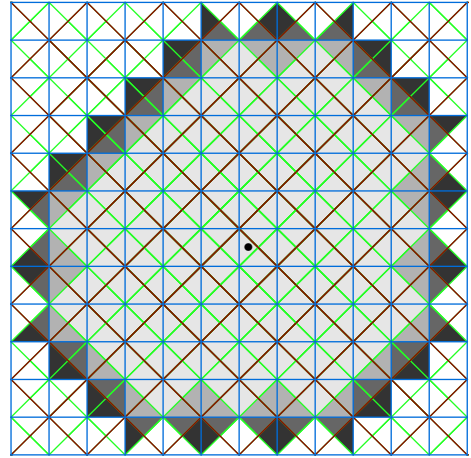
 
  \begin{subfigure}[b]{0.5\linewidth}
    \centering
    \scalebox{.5}{\dibujothetaone}
    \caption{ From lightest to darkest grey:  Lower intervals of  $\theta(2,1)   $, $\theta(2,1)s_0$, $\theta(2,1)s_0s_2$ and $\theta(2,1)s_0s_2s_1$.}
    \label{fig: dibujothetaone} 
    \vspace{4ex}
  \end{subfigure}\hfill%
  \begin{subfigure}[b]{0.5\linewidth}
    \centering
    \scalebox{.5}{\dibujothetatwo} 
    \caption{From lightest to darkest grey:  Lower intervals of  $s_0\theta(2,1)$, $s_0\theta(2,1)s_0$, $s_0\theta(2,1)s_0s_2$ and $s_0\theta(2,1)s_0s_2s_1$.} 
    \label{fig: dibujothetatwo} 
    \vspace{4ex}
  \end{subfigure}\hfill%
  \begin{subfigure}[b]{0.5\linewidth}
    \centering
    \scalebox{.5}{\dibujothetathree} 
    \caption{From lightest to darkest grey:  Lower intervals of  $s_2s_0\theta(2,1)$, $s_2s_0\theta(2,1)s_0$, $s_2s_0\theta(2,1)s_0s_2$ and $s_2s_0\theta(2,1)s_0s_2s_1$. }  
    \label{fig: dibujothetathree} 
  \end{subfigure}\hfill%
  \begin{subfigure}[b]{0.5\linewidth}
    \centering
    \scalebox{.5}{\dibujothetafour}
    \caption{From lightest to darkest grey:  Lower intervals of  $s_1s_2s_0\theta(2,1)$, $s_1s_2s_0\theta(2,1)s_0$, $s_1s_2s_0\theta(2,1)s_0s_2$ and $s_1s_2s_0\theta(2,1)s_0s_2s_1$.} 
    \label{fig: dibujothetafour} 
  \end{subfigure}\hfill%
  \caption{Lower intervals in the big region.}
  \label{fig: all theta's friend} 
\end{figure}

 We now  continue with the description of the sets $\leq \theta(m,n)t_m$, $\leq \theta(m,n)t_ms_2$ and $\leq \theta(m,n)t_ms_2t_m'$.  In order to obtain $\leq \theta(m,n)t_m$, we take as a starting point $\leq \theta(m,n) =S(m,n)$ and we add to it the triangles that have an adjoining side matching the color of $t_m$. Similarly, we can obtain $\leq \theta(m,n)t_ms_2$ from $\leq \theta(m,n)t_m$ and $\leq \theta(m,n)t_ms_2t'_m$ from $\leq \theta(m,n)t_ms_2$. In formulas we have 
\begin{align} \label{description first friend}
    \leq \theta(m,n)t_m =& \leq \theta(m,n)\,  \sqcup   \, \front_{t_m} (\theta(m,n)) t_m.\\
\label{description second friend}    \leq \theta(m,n)t_ms_2 = &  \leq \theta(m,n)t_m \,  \sqcup  \, \front_{s_{2}} (\theta(m,n)t_m ) s_2;\\
\label{description third friend}    \leq \theta(m,n)t_ms_2t_m' = &  \leq \theta(m,n)t_ms_2 \,  \sqcup  \, \front_{t_m'} (\theta(m,n)t_ms_2) t_m'.
\end{align}

This construction is illustrated in Figure \ref{fig: dibujothetaone} for $\theta(2,1)$. 
With this description in hand we are in position to obtain the sizes of these sets.

\begin{lemma} \label{lemma size closer friends}
Let $m,n \in \mathbb N $. Then we have
\begin{align}
\label{mejor amigo theta}    \ab{\theta(m,n)t_m}           &  = 8(m^2+3m+4mn+2n^2+4n+2),   \\
\label{amigo theta}          \ab{\theta(m,n)t_ms_2}        &  = 8(m^2+4m+4mn+2n^2+5n+3),  \\
\label{amigo theta lejano}   \ab{\theta(m,n)t_ms_2t_m'}    &  =  8(m^2+5m+4mn+2n^2+6n+4)  . 
\end{align}
\end{lemma}
\begin{proof}
It is a straightforward counting exercise to show that $\ab{ \front_{t_m} (\theta(m,n))   } = 8(m+2n+1) $. Thus \eqref{mejor amigo theta} follows from Corollary \ref{lem:abx} and  \eqref{description first friend}. Similarly, we have that  $\ab{\front_{s_{2}} (\theta(m,n)t_m )  } = 8(2m+3n+2) $. This gives us \eqref{amigo theta} by combining \eqref{description second friend} and the already proved identity \eqref{mejor amigo theta}. Finally, we obtain \eqref{amigo theta lejano} by combining  \eqref{description third friend} and the already proved identity \eqref{amigo theta}, together with the equality $\ab{\front_{t_m'} (\theta(m,n)t_ms_2)} =  8(3m+4n+1) $. 
\end{proof}

 We now move on to the description of the remaining components of the big region. We begin with the region $s_0\mf{C}$. Let $w=s_0x$ for some $x \in \mf{C}$. Then $\leq w = \leq x \cup s_0( \leq x )$. Geometrically, this implies that $\leq w$ can be obtained as the union of $ \leq x$  with the image of $ \leq x$ under the reflection through the green line that supports a side of the identity triangle. Since we have already obtained a geometric  description of all the sets $ \leq x$ for $x\in \mf{C}$, we now have a geometric description of all the sets $\leq w$ for $w\in s_0\mf{C}$. Examples are given in  Figure \ref{fig: dibujothetatwo}, obtained using those from Figure \ref{fig: dibujothetaone}.  

We now describe lower intervals for elements in $s_2s_0\mf{C}$. Let $w\in s_2s_0\mf{C} $. Then $w=s_2x$ for some  $x \in s_0\mf{C}$. Arguing as in the previous paragraph, we get that $\leq w$ is the union of $\leq x$ with the image of $\leq x$ under the reflection through the blue line that supports a side of the identity triangle. This is illustrated in Figure \ref{fig: dibujothetathree} using the examples from Figure \ref{fig: dibujothetatwo}. 

Finally, we describe the lower intervals for elements in $s_1s_2s_0\mf{C}$. Let $w=s_1x$ for some $x \in s_2s_0\mf{C}$. This time $\leq w$ corresponds to the union of $\leq x$ with the image of $\leq x$ under the reflection through the red line that supports a side of the identity triangle. We have illustrated this case in Figure \ref{fig: dibujothetafour} starting from the corresponding pictures in Figure \ref{fig: dibujothetathree}.


Having described the lower intervals geometrically it is now an easy (though tedious) task to determine the size of these sets. For space reasons, we omit the proof and leave the reader with the resulting formulas.

\begin{lemma}   \label{lemma cardinality big region not fundamental}
Let $m,n \in \mathbb N$. Then,
\begin{align}
   \ab{s_0\theta(m,n)} & =  8(m^2+ 3m+ 4mn+ 2n^2+ 4n+ 2) \\
    \ab{s_2s_0\theta(m,n)} & =  8(m^2+ 4m+ 4mn+ 2n^2+ 6n+ 2) \\
     \ab{s_1s_2s_0\theta(m,n)} & = 8(m^2+ 5m+ 4mn+ 2n^2+ 8n+ 2)\\
   \ab{s_0\theta(m,n)t_m} & =  8(m^2+ 4m+ 4mn+ 2n^2+ 6n+3)+ 4 \\
   \ab{s_2s_0\theta(m,n)t_m} & =   8(m^2+ 5m+ 4mn+ 2n^2+ 8n+ 5) \\
    \ab{s_1s_2s_0\theta(m,n)t_m} &  = 8(m^2+ 6m+ 4mn+ 2n^2+ 10n+6)+ 4\\
   \ab{s_0\theta(m,n)t_ms_2} & =  8(m^2+ 5m+ 4mn+ 2n^2+ 8n+ 4) \\
   \ab{s_2s_0\theta(m,n)t_ms_2} & =   8(m^2+ 6m+ 4mn+ 2n^2+ 9n+ 5)+6\\
   \ab{s_1s_2s_0\theta(m,n)t_ms_2} & =  8(m^2+ 7m+ 4mn+ 2n^2+ 10n+ 7)+4\\
    \ab{s_0\theta(m,n)t_ms_2t_m'} & =  8(m^2+ 6m+ 4mn+ 2n^2+ 8n+ 6)+4\\
     \ab{s_2s_0\theta(m,n)t_ms_2t_m'} & =     8(m^2+ 7m+ 4mn+ 2n^2+ 9n+ 8)+4\\
   \ab{s_1s_2s_0\theta(m,n)t_ms_2t_m'} & =   8(m^2+ 8m+ 4mn+ 2n^2+ 10n+ 10)+4.
\end{align}
\end{lemma}

\begin{remark}\rm
The big region is made of eight sub-regions, namely, its connected components. In this section we have described just four of them. The remaining four  regions are obtained by applying the automorphism $\varphi$, and therefore, their description follows from the description of the regions already considered. In particular, the size formulas in Corollary \ref{lem:abx}, Lemma \ref{lemma size closer friends} and Lemma \ref{lemma cardinality big region not fundamental} are the same for an element  and its $\varphi$-counterpart. In formulas $\ab{x} = \ab{x'}$.
\end{remark}



\subsection{Lower intervals for the thick region}\label{lower intervals thick wall}

In this section we describe some lower intervals for elements located in the thick region. More precisely, we provide a geometric description of the sets $\leq x_{n}$ and $\leq e_{n}$. We omit description of the lower intervals of the remaining elements in the thick region as this is not needed in the sequel. 

We begin by considering the elements $x_{n}$. In this case we can see there is a pattern that appears modulo three: $\leq x_{3k}$, $\leq x_{3k+1}$ and $\leq x_{3k+2}$ all behave differently. Figure \ref{fig: all x_n} shows some examples. 

\begin{lemma}
    Let $k\geq 1$. Then $\leq x_{3k} = S(k-1,0)\setminus Z $, where $Z$ is the subset of $\front{(S(k-1,0))}$ formed by the triangles with a side on the line that connects the north and west vertices of $S(k-1,0)$. 
\end{lemma}

\begin{proof}
The result follows by induction on $k$ and the identity $\leq x_{3k+3}  = (\leq x_{3k} ) ( \leq t_{k}'s_2t_k' )   $.  
\end{proof}

Having described the set $\leq x_{3k}$, it is now easy to obtain a geometric description of set $\leq x_{3k+1}$. Indeed, $\leq x_{3k+1}  = \leq x_{3k} \sqcup   (\front_{t_k'}(x_{3k})) t_k' $. Similarly, we have $ \leq x_{3k+2}  = \leq x_{3k+1} \sqcup  (\front_{s_2}(x_{3k+1})) s_2  $. Using these descriptions and a straightforward counting argument we obtain the following result.

\begin{lemma}
For all $k\geq 1$ we have
\begin{equation}
     \ab{x_{3k}}  = 8k^2- 2k, \qquad
    \ab{x_{3k+1}}  = 8k^2+4k, \qquad  \text{and} \qquad 
    \ab{x_{3k+2}} = 8k^2+12k.   
\end{equation}
\end{lemma}

\begin{figure}[ht]
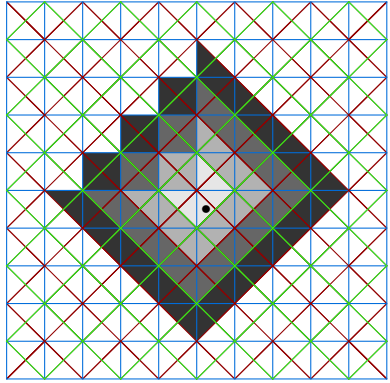
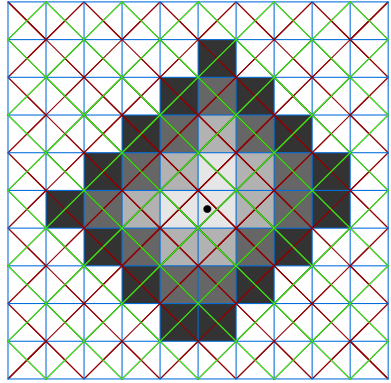
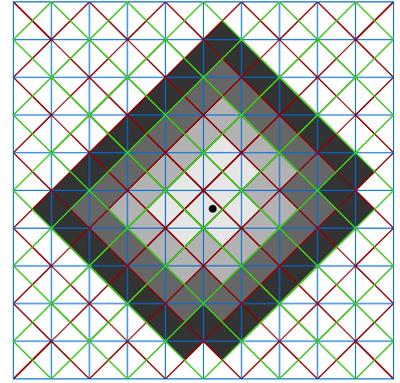
 
  \begin{subfigure}[b]{0.3\linewidth}
    \centering
    \scalebox{.5}{\dibujoxtresk}
    \caption{ From lightest to darkest grey:  Lower intervals of  $ x_{3k}$ for $k=1,2,3$ and $4$. }
    \label{fig: dibujoxtresk}
  \end{subfigure}\hfill%
  \begin{subfigure}[b]{0.3\linewidth}
    \centering
     \scalebox{.5}{\dibujoxtreskmasuno} 
     \caption{ From lightest to darkest grey:  Lower intervals of  $ x_{3k+1}$ for $k=1,2,3$ and $4$. }
    \label{fig: dibujoxtreskmasuno} 
  \end{subfigure}\hfill%
  \begin{subfigure}[b]{0.3\linewidth}
    \centering
   \scalebox{.5}{\dibujoxtreskmasdos}
    \caption{ From lightest to darkest grey:  Lower intervals of  $ x_{3k+2}$ for $k=1,2,3$ and $4$. }
    \label{fig: dibujoxtreskmasdos} 
  \end{subfigure}
  \caption{Lower intervals for elements $x_n$.}
  \label{fig: all x_n} 
\end{figure}

We now focus on the elements $e_n$. Just as the $x_n$, they obey a pattern that appears modulo $3$. Figure 
\ref{fig: all e_n} shows the different patterns of the lower intervals for the elements $e_n$. In this case it is easier to first consider the lower intervals of the elements $e_{3k+2}$. 

\begin{lemma} \label{lemma e 3k+2 }
    Let $k\geq 0$. Then the set $\leq e_{3k+2} =S(k)$ where $S(k)$ is the smallest square containing $e_{3k+2}$  and $x_{3k+2}$.
\end{lemma}

\begin{proof}
The result follows by induction on $k$ and the identity $\leq e_{3k+5}  = (\leq e_{3k+2} ) ( \leq t_{k}t_k's_2 )   $.  
\end{proof}

Using Lemma \ref{lemma e 3k+2 } as starting point, we can describe the sets $\leq e_{3k+3}$ and $ \leq e_{3k+4}$. These are 
\begin{equation}
    \leq e_{3k+3} = (\leq e_{3k+2}) \sqcup (\front_{t_k}(\leq e_{3k+2} ))t_k \qquad \mbox{and} \qquad   \leq e_{3k+4} = (\leq e_{3k+3}) \sqcup (\front_{t_k'}(\leq e_{3k+3} ))t_k'. 
\end{equation}
An easy counting argument using these descriptions shows the following. 
 \begin{lemma}\label{lem:abeu}
  For all $k\geq 0$, we have  \[\ab{e_{3k}}=8k^2+4k,\qquad \ab{e_{3k+1}}=8k^2+8k+4,\qquad \ab{e_{3k+2}}=8(k+1)^2.\]
\end{lemma}

\begin{figure}[ht] 
  \begin{subfigure}[b]{0.3\linewidth}
    \centering
    \scalebox{.5}{\dibujoetreskmasdos}
    \caption{ From lightest to darkest grey:  Lower intervals of $ e_{3k+2}$ for $k=0,1,2$ and $3$. }
    \label{fig: dibujoetreskmasdos}
  \end{subfigure}\hfill%
  \begin{subfigure}[b]{0.3\linewidth}
    \centering
     \scalebox{.5}{\dibujoetresk}
    \caption{ From lightest to darkest grey:  Lower intervals of $ e_{3k+3}$ for $k=0,1,2$ and $3$. }
    \label{fig: dibujoetresk}
  \end{subfigure}\hfill%
  \begin{subfigure}[b]{0.3\linewidth}
    \centering
   \scalebox{.5}{\dibujoetreskmasuno}
    \caption{ From lightest to darkest grey:  Lower intervals of $ e_{3k+4}$ for $k=0,1,2$ and $3$. }
    \label{fig: dibujoetreskmasuno}
  \end{subfigure}
  \caption{Lower intervals for elements $e_n$.}
  \label{fig: all e_n} 
\end{figure}

\subsection{Coatoms for lower intervals.}  \label{section coatoms }

We write $y\lessdot w$ to mean that $y\leq w$ and $\ell(y)=\ell(w)-1$, and $\lessdot w$ to denote the set of all such $y$. If $\ul{w}=(s_1,\dots,s_n)$ is a reduced expression for $w$, then declare $\ul{w}(\ol{i}):=s_1\cdots \hat{s}_i \cdots s_n=s_1 \cdots s_{i-1}s_{i+1}\cdots s_n$. Further, let $\ul{x}_n$ be the preferred reduced expression $(a_1,\cdots,a_n)$ obtained from the sequence (\ref{eq:xnseq}) and $\ul{e}_n=(s_1,a_1,a_2,\dots,a_n)$. We also consider coatoms for elements ${d}_n$ defined as the the product of the first $n$ symbols of the infinite sequence $\{b_i\}_{i=1}^\infty= (s_2,s_1,s_2,s_0,s_2,s_1,s_2,s_0,\dots)$. These elements admit a unique reduced expression, and they comprise the ``northwest wall'' of the thin region; we will return to them in Section~\ref{section thin region}.

\begin{lemma} \label{lem:lessx}
For $k\geq 2$, the sets $\{y \mid y\lessdot x_n\}$ have the following description.
\begin{align}
    \lessdot x_{3k} &= \{ \ul{x}_{3k}(\ol{1}),\  \ul{x}_{3k}(\ol{3}),\ \ul{x}_{3k}(\ol{3k-2}),\ \ul{x}_{3k}(\ol{3k}) \} \\
    &= \{w_{3k-3},\ s_1\theta'(k-2,0),\ \theta(k-2,0)t_{k-2},\ x_{3k-1}\} \\
    \lessdot x_{3k+1} &= \{ \ul{x}_{3k+1}(\ol{1}),\  \ul{x}_{3k+1}(\ol{3}),\ \ul{x}_{3k+1}(\ol{3k}),\ \ul{x}_{3k+1}(\ol{3k+1})\}\\
    &=\{w_{3k-2},\ s_1\theta'(k-2,0)t'_{k-2},\ \bar{x}_{3k},\ x_{3k}\} \\
    \lessdot x_{3k+2} &= \{ \ul{x}_{3k+2}(\ol{1}),\   \ul{x}_{3k+2}(\ol{3}),\ \ul{x}_{3k+2}(\ol{3k}),\ \ul{x}_{3k+2}(\ol{3k+1}),\ \ul{x}_{3k+2}(\ol{3k+2}) \}\\
    &=\{w_{3k-1},\ s_1\theta(k-2,0)t'_{k-2}s_2,\ s_1s_2s_0\theta(k-2,0),\ \theta(k-1,0),\ x_{3k+1}\} 
\end{align}
Furthermore, for $k\geq 3$,
\begin{align}
    \lessdot \ol{x}_{3k}=\{w_{3k-3},\ s_1\theta'(k-3,0)t'_{k-3}s_2t_{k-3},\  s_1s_2s_0\theta(k-3,0)t_{k-3},\ x_{3k-1} \}.
\end{align}
\end{lemma}
\begin{proof}
Recall that all elements $y\leq w$ can be obtained as subexpressions of a fixed reduced expression for $w$, so all elements of $\lessdot {x}_n$ are obtained by removing a single symbol from $\ul{x}_n$. We will argue for $x_{3k}$, the other cases being similar. We examine which individual symbols can be removed from $\ul{x}_{3k}$ to obtain a reduced expression. It is clear that $\ul{x}_{3k}(\ol{1})$ and $\ul{x}_{3k}(\ol{3k})$ are reduced, and that none of the $s_2$'s in $\ul{x}_{3k}$ can be removed to obtain a reduced expression. If we remove any $a_i=s_1$ where $4\leq i\leq 3k-3$, then there is guaranteed to be a substring $s_0s_2s_0$ immediately to the right or to the left of $a_i$. Assume it is to the right; upon removal we obtain a substring of the form $s_0s_1s_2s_0s_2s_0=s_0s_1s_0s_2s_0s_2=s_1s_2s_0s_2$ which is not reduced. An identical argument applies if the substring is to left, or if $a_i=s_0$. Thus, the only other possibilities are $\ul{x}_{3k}(\ol{3})$ and $\ul{x}_{3k}(\ol{3k-2})$, and these are indeed reduced.

The latter description of each set $\lessdot x_n$ can be easily observed in the geometric realization of $W$, and the set $\lessdot \ol{x}_{3k}$ is obtained by similar arguments. 
\end{proof}
We will also need the following in Section~\ref{section east and west}.
\begin{lemma}\label{lemma coatomos e}
    For $k\geq 2$, the sets $\lessdot e_n$ have the following description.
    \begin{align}
        \lessdot e_{3k}&=\{ x'_{3k},\ \ol{x}_{3k},\ s_1\theta'(k-2,0)t'_{k-2},\ e_{3k-1} \}\\
         \lessdot e_{3k+1}&=\{ x'_{3k+1},\ {x}_{3k+1},\ e'_{3k},\ e_{3k} \}\\
          \lessdot e_{3k+2}&=\{ x'_{3k+2},\ {x}_{3k+2},\ s_0\theta(k-1,0),\ s_1\theta'(k-1,0),\ e_{3k+1} \}
    \end{align}
\end{lemma}
\begin{proof}
Multiplying ${x}_n$ by $s_1$ on the left does not introduce new possibilities for removal of the symbols from $\ul{x}_n$, though it can rule them out. We see that removing $a_3$ from $\ul{e}_n$ gives us a non-reduced substring $s_1s_0s_2s_1s_2s_1$, but now we can also remove the first symbol $s_1$ from $\ul{e}_n$ and get $x_n$ back, so the size of the coatom set remains the same. The descriptions above are evident from the geometric realization of $W$.
\end{proof}
\begin{lemma}\label{lem:lessw}
For $n\geq 7$, the sets $\{y \mid y\lessdot d_n\}$ have the following description.
\[ \lessdot d_n = \begin{cases}
\{ \ul{d}_n(\ol{1}), \ul{d}_n(\ol{3}), \ul{d}_n(\ol{n-3}), \ul{d}_n(\ol{n-1}), \ul{d}_n(\ol{n}) \} &\text{ if } n \text{ is even},\\
\{ \ul{d}_n(\ol{1}), \ul{d}_n(\ol{3}), \ul{d}_n(\ol{n-2}), \ul{d}_n(\ol{n}) \}& \text{ if } n \text{ is odd},
\end{cases}
\]
\end{lemma}
\begin{proof}
No symbol $s_1$ or $s_0$ can be removed from $\ul{d}_n$ and leave a reduced expression unless it is $b_n$. If $b_i=s_2$ for $5\leq i \leq n-4$, then removal introduces a substring (possibly swapping $s_1$ with $s_0$) of the form
\[s_2s_1s_2s_0s_1s_2s_0s_2= s_2s_1s_2s_1s_0s_2s_0s_2=s_1s_2s_1s_0s_2s_0\]
so $\ul{d}_n(\ol{i})$ is not reduced. The remaining options figure in the lists above.
\end{proof}


It is clear from Figure~\ref{fig: geometric realization of W} that each element in $\mf{C}$ admits a reduced decomposition of the form $x_{6k}d_n$ or $x_{6k+3}d'_n$. 

\begin{lemma} \label{lem: less theta mejores amigos}
For any $w\in \mf{C}$, the set $\lessdot w$ consists of at most five elements. In particular for $m,n>0$ we have
\begin{align}
    \lessdot \theta(m,n)    &= \{s_2s_1\theta'(m-1,n), \, \theta(m-1,n)t_{m-1}s_2, \, s_1s_2s_0\theta(m,n-1), \, \theta(m,n-1)t_ms_2t_m' \}  \\
    \lessdot \theta(m,n)t_m &=\{s_2s_1\theta'(m-1,n)t'_{m-1}, \,  \theta(m-1,n)t_{m-1}s_2t'_{m-1},\\  
                            & \qquad\qquad \qquad   \qquad \qquad  s_1s_2s_0\theta(m,n-1)t_m,   \,  \theta(m+1,n-1)t_{m+1},  \,  \theta(m,n) \} \\
\lessdot \theta(m,n)t_ms_2  & =\{s_2s_1\theta'(m-1,n)t'_{m-1}s_2, \,  \theta(m-1,n+1),  \\ 
                            & \qquad\qquad \qquad  \qquad \qquad  s_1s_2s_0\theta(m,n-1)t_ms_2, \,  \theta(m+1,n-1)t_{m+1}s_2, \, \theta(m,n)t_m \} \\
\lessdot \theta(m,n)t_ms_2t'_m &= \{s_2s_1\theta'(m-1,n)t'_{m-1}s_2t_{m-1}, \, \theta(m-1,n+1)t_{m-1}, \\ 
                           & \qquad \qquad \qquad  \qquad \qquad  s_1s_2s_0\theta(m,n-1)t_ms_2t'_m,  \, \theta(m+2,n-1),  \, \theta(m,n)t_ms_2 \}
\end{align}
\end{lemma}
\begin{proof}
We assume $w=x_m d_n$ for $m=6k$; the other case is identical. From Lemma~\ref{lem:lessx}, $x_m$ has four removable symbols, but when $d_n$ is appended, the last $a_n$ can no longer be removed to obtain a reduced expression. If $a_3=s_1$ is removed, then braid relations can be applied to give an expression for $\ul{x}_{m}(\ol{3})$ ending in $s_2$ which disqualifies this possibility as well. Thus there are only two possibilities for removal from $x_m$: $a_1$ and $a_{m-2}$.

By Lemma~\ref{lem:lessw}, the symbols of $d_n$ give four or five more possibilities depending on the parity of $n$. However, removal of $b_3$, leads to a substring of the form $s_0 s_2 s_0 s_2 s_1 s_0 = s_2s_0s_2s_1 $. Further, if $n=4k$ (respectively, $n=4k+2$), removal of $b_{n-1}$ (respectively, $b_{n-3}$) plus the application of braid relations leads to an expression for $\ul{d}_n(\ol{n-1})$ (respectively, $\ul{d}_n(\ol{n-3})$) that begins with $s_2 s_0$. As $x_m$ ends with $s_0s_2s_0$, the result is not a reduced expression. Thus, $d_m$ adds three more possibilities for a maximum total of five elements.

The elements $\theta(m,n)$ in particular can be written as $\theta(m,n)=x_{3(m+1)}d_{4n+1}$ for $m$ odd and $\theta(m,n)=x_{3(m+1)}d'_{4n+1}$ for $m$ even. An argument similar to the one of the previous paragraph shows that the removal of $b_{4n-1}$ (or $b'_{4n-1}$, as the case may be) from this reduced expression yields something not reduced, so $\lessdot \theta(m,n)$ consists of four elements when $m,n>0$. These are (when $m$ is odd)
\begin{equation}\label{eq:lessdottheta} \lessdot \theta(m,n) = \{ x_{3(m+1)}(\ol{1})d_{4n+1},\ x_{3(m+1)}(\ol{3m+1})d_{4n+1}, \ x_{3(m+1)}d_{4m+1}(\ol{1}), \ x_{3(m+1)}d_{4n+1}(\ol{4n+1}) \}. 
\end{equation}
The description of this set given in the statement of the lemma can be observed in the geometric realization of $W$. The other cases can be reasoned similarly, with all five possibilities yielding reduced expressions.
\end{proof}

\begin{remark}\rm 
If $m=0$, the first two elements of (\ref{eq:lessdottheta}) are the same, and if $n=0$ the last two are the same, so the set has three elements in case one of $m$ or $n$ is 0. A similar situation occurs for the elements of $\lessdot\theta(m,n)y$ ($y\in\{t_m, t_ms_2, t_ms_2t_m'\})$ in these cases as well. We leave the precise description to the reader.
\end{remark}
The following descriptions will be necessary for the arguments of Section~\ref{section big region}. They can be deduced by analysis  similar to that of Lemma~\ref{lem: less theta mejores amigos}.
\begin{lemma} \label{lemma amigos de theta no cercanos}
   For all $m,n\geq 0$ and $y\in\{t_m, t_ms_2, t_ms_2t_m'\}$,
   \begin{align}
    \lessdot s_0\theta(m,n)y & = \{s_0\}(\lessdot\theta(m,n)y) \cup\{\theta(m,n)y\} \\
     \lessdot s_2s_0\theta(m,n)y & =  \{s_2s_0\}(\lessdot\theta(m,n)y)\cup\{s_0\theta(m,n)y\} \\
       \lessdot s_1s_2s_0\theta(m,n)y & =  \{s_1s_2s_0\}(\lessdot\theta(m,n)y)\cup\{s_2s_0\theta(m,n)y\}
   \end{align}
\end{lemma}

\begin{remark}\rm
Observe that while the description of the set $\lessdot \theta(m,n)y$ is affected when $m$ or $n$ is zero, the relationship between this set and $\lessdot x \theta(m,n) y$ is the same in all cases for $x\in \{s_0, s_2s_0,s_1s_2s_0\}$. 
\end{remark}

\section{Multiplicative formulas in the Hecke algebra}

In this section we introduce the Hecke algebra of the affine Weyl group of type $\tilde{B}_2$. We also collect several multiplicative identities that will be key in order to obtain formulas for Kazhdan-Lusztig basis elements in the forthcoming sections. 

\subsection{The Hecke algebra of type   \texorpdfstring{$\tilde{B}_2$}{} }
Let $\mathcal{H}$ be the Hecke algebra of $W$. It is the ${\cal{A}}:=\bbZ[v,v^{-1}]$-algebra with generators $\bfH_{s_0}$, $\bfH_{s_1}$ and $\bfH_{s_2}$ and relations $\bfH_{s_{i}}^2 = (v^{-1}-v)\bfH_{s_i} +1$,
\begin{equation}
\bfH_{s_{0}}\bfH_{s_{1}} =  \bfH_{s_{1}}\bfH_{s_{0}}, \qquad  \bfH_{s_{0}}\bfH_{s_{2}}\bfH_{s_{0}}\bfH_{s_{2}}= \bfH_{s_{2}}\bfH_{s_{0}}\bfH_{s_{2}}\bfH_{s_{0}} \, \mbox{ and } \,    \bfH_{s_{1}}\bfH_{s_{2}}\bfH_{s_{1}}\bfH_{s_{2}}= \bfH_{s_{2}}\bfH_{s_{1}}\bfH_{s_{2}}\bfH_{s_{1}}. 
\end{equation}

Given a reduced expression $s_{i_1}s_{i_2}\ldots s_{i_k}$  of an element $w\in W$ we define $  \bfH_{w}:=\bfH_{s_{i_1}}\bfH_{s_{i_2}} \ldots \bfH_{s_{i_k}}$. It is well-known that $\bfH_{w}$ does not depend on the choice of a reduced expression. The set $\{ \bfH_w  \}_{w\in W}$ forms an $\cal{A}$-basis of $\mathcal{H}$, which is called the standard basis. It is easy to see that each generator of $\mathcal{H}$ is invertible and therefore all the elements of the standard basis are invertible. There is a $\bbZ$-linear involution  $d:\mathcal{H} \rightarrow \mathcal{H} $  which is determined by $d(v)=v^{-1}$ and $d(\bfH_{w})= \bfH_{w^{-1}}^{-1}$. An element invariant under $d$ is called \emph{self-dual}.\\
There is another basis $\{\undH_{w} \}_{w\in W}$ called the Kazhdan-Lusztig basis whose elements are uniquely determined by two conditions: They are self-dual and 
\begin{equation}
    \undH_w  = \bfH_{w} + \sum_{x < w} h_{x,w}(v) \bfH_x, 
\end{equation}
for some polynomials $h_{x,w}(v) \in v\bbZ [v]$. These polynomials are the Kazhdan-Lusztig polynomials.

There is a recursive algorithm to compute this basis. Indeed, if $\mu(x,w)$ denotes the coefficient of $v$ in $h_{x,w}(v)$ then 
\begin{equation}  \label{eq mult recurrence}
    \undH_w\undH_{s} = \undH_{ws} + \sum_{ xs<x<w  } \mu(x,w) \undH_{x},
\end{equation}
for all pairs $(w,s)\in W\times S$  such that $w<ws$. Also, we have 
\begin{equation}   \label{eq v+v^-1}
      \undH_w\undH_{s} = (v+v^{-1})\undH_w
\end{equation}
for all pairs $(w,s)\in W\times S$  such that $w>ws$.

We recall the elements defined in the introduction for any $w\in W$,
\begin{equation}
    \bfN_w=\sum_{x\leq w} v^{l(w)-l(x)}\bfH_x.
\end{equation}

\begin{definition}
Given $w\in W$ and $X \in \mathcal{H}$ we denote by $G_{w}(X)$ the coefficient of $\bfH_{w}$ in $X$ when it is written in terms of the standard basis. That is,
\begin{equation}
    X= \sum_{w\in W} G_{w}(X) \bfH_w. 
\end{equation}
We also define the \emph{content} of $X$ by 
\begin{equation}
    c(X):=\sum_{w\in W} G_{w}(X)(1)\in \bbZ.
\end{equation}
\end{definition}
Note that this implies that $c(\bfN_w)= \ab{w}$ and also that $c(X\undH_s)=2c(X)$ for any $X\in \mc{H}$ and any $s\in S$, since
\begin{equation}  \label{mult H by Hunderline_s}
\bfH_w\undH_s = \begin{cases}
\bfH_{ws} +v\bfH_w, & \text{if } w<ws;\\
\bfH_{ws}+v^{-1}\bfH_w, & \text{if } w>ws.
\end{cases}
\end{equation}

\begin{definition}
	Let $w\in W$. An element $H\in \mathcal{H}$ is called \emph{triangular of height} $w$ if $G_{w}(H) = 1$ and $G_x(H)=0$ for $x \not \leq w$.  Furthermore,  a triangular element $H$ of height $w$ is called \emph{monotonic} if $G_{x}(H)\in \bbN[v,v^{-1}]$ for all $x\in W$ and 
	\begin{equation} \label{eq defin monotonicity}
		G_y(H)-v^{l(x)-l(y)}G_{x}(H) \in \mathbb{N}[v,v^{-1}],
	\end{equation}
	for all $y\leq x \leq w$.
\end{definition}

A trivial example of a monotonic element is any $\bfN_{w}$. We also know that $\undH_w$ is always monotonic (see  \cite{braden2001moment,plaza2017graded}). 

\begin{lemma} \label{lemma preserving monotonicity}
	Let $H \in \mathcal{H}$ be a monotonic element of height $w$. Suppose that $ws>w$. Then, $H\undH_s$ is monotonic of height $ws$.  
\end{lemma}

\begin{proof}
It is clear that $H\undH_s$ is triangular of height $ws$. On the other hand, \eqref{mult H by Hunderline_s} shows that 
 \begin{equation}  \label{equation y greater than ys}
	G_x(H\undH_s)=\left\{  \begin{array}{rl}
	  	vG_x(H) + G_{xs}(H), & \mbox{if } xs>x;\\
	  	v^{-1}G_x(H) + G_{xs}(H), & \mbox{if } xs <x.
	\end{array}   \right.
\end{equation}
Using this, a simple case analysis shows that the monotonicity of $H$ implies the monotonicity of $H\undH_s $. \end{proof}

\begin{lemma} \label{lemma N por Hs cuando baja}
Let $(W,S)$ be an arbitrary Coxeter system. Let $w\in W$ and $s\in S$ such that $s\in D_{R}(w)$. Then, $\bfN_{w} \undH_s = (v+v^{-1})\bfN_{w}$. 
\end{lemma}
\begin{proof}
Since $ws<w$ we can use the Lifting Property \cite[Proposition 2.2.7]{bjorner2006combinatorics}	to conclude that multiplication on the right by $s$ induces  a permutation of the lower interval $[e,w]$. Therefore the result follows by \eqref{equation y greater than ys}.   
\end{proof}

\begin{definition} \label{defi X geqh Y}
Let $p$ and $q$ in $\mathcal{A}$.  We write $p\geq q$ if $p-q \in \mathbb{N}[v,v^{-1}] $. Then for $X,Y \in \mathcal{H} $, we write $X\geqh Y$ if $G_w(X) \geq  G_w(Y) $ for all $w\in W$. 
\end{definition}

Notice that $X\geqh Y$ and $c(X)=c(Y)$ implies $X=Y$. We claim no originality in this observation and refer to \cite{libedinsky2020affine} for its first application in the context of computation of Kazhdan-Lusztig polynomials. In general, it is not an easy task to prove an inequality of the form $X\geqh Y$. However, there are certain situations where we can make some simplifications. For instance, let us suppose  that 
\begin{equation}\label{eq Y suma de Ns}
    Y= \sum_{w\in Z} p_w(v)\bfN_{w}
\end{equation}
where $Z$ is a finite subset of $W$ and $p_{z}(v)=\sum_{i\in \mathbb{Z}} p_{z}^i  v^i \in \mathbb{N}[v,v^{-1}]   $. For each $i\in \mathbb{Z}$ we define
\begin{equation}
    Y_i= \sum_{w\in Z} p_w^{i-\ell(w)}\bfN_{w}.
\end{equation}
Since the elements $Y_i$ lie in different degrees (in the sense that $G_u(Y_i) $ and $G_{u}(Y_j)$ are monomials of different degree if $i\neq j$) in order to prove $X\geqh Y$ it is enough to show  $X \geqh Y_i$ for all $i\in \Z$. We stress that $Y_i=0$ for all but finitely many integers. Henceforth, we refer to this simplification as ``degree reasons.''

In order to prove inequalities of the form $X\geqh Y_i$ we can make some extra reductions if $X$ is monotonic.  For example, if $X$ is monotonic and $Y_i$ is made of just one $\bfN$-element, i.e. $Y_i=cv^k\bfN_{w}$, then in order to prove $X\geqh Y_i$ it is enough to check $G_{w}(X)\geq cv^k$. In contrast, if $Y_{i}$ is made of two or more $\bfN$-elements then additional analysis of the  relationship between the elements involved in $Y_i$ in the Bruhat order is required.

We conclude this section with a useful result that allows us to rule out the occurrence of terms in the sum in \eqref{eq mult recurrence}. A proof can be found in \cite[(2.3.f)]{kazhdan1979representations} or \cite[Proposition 5.1.9]{bjorner2006combinatorics}.

\begin{lemma} \label{useful lemma}
 Let $x,w\in W$ such that $\mu (x,w)\neq 0$. Suppose that $l(w)-l(x) > 1$. Then 
 $$   D_{R}(w)\subseteq D_R(x)\qquad \mbox{ and } \qquad D_L(w) \subseteq D_L(x).  $$
  In particular, if $s\not \in D_R(w)  $ and $\undH_{x}$ appears in the expansion of $\undH_w\undH_s$ when written in terms of the Kazhdan-Lusztig basis, then $D_R(w)\cup \{s \} \subseteq D_R(x)$. Similarly, if $s\not \in D_L(w)  $ and $\undH_{x}$ appears in the expansion of $\undH_s\undH_w$ when written in terms of the Kazhdan-Lusztig basis, then $D_L(w)\cup \{s \} \subseteq D_L(x)$.
\end{lemma}


\subsection{Multiplication formulas for \texorpdfstring{$\bfN_{w}$}{}}

In this section we collect several multiplicative formulas involving elements $\bfN_w$.

\begin{lemma}  \label{lemma N mult by friend}
Let $m$ and $n$ be positive integers.  Then,
	\begin{equation} \label{mult N by one}
	\bfN_{\theta (m,n)}\undH_{t_m}= \bfN_{\theta (m,n)t_m} +  v\bfN_{\theta (m-1,n)t_m'}.
	\end{equation}	 
\end{lemma}

\begin{proof}
First, let us prove that the contents on both sides of \eqref{mult N by one} coincide.
On the left side, Corollary \ref{lem:abx} shows that
$$c(\bfN_{\theta (m,n)}\undH_{t_m})=2c(\bfN_{\theta (m,n)})=\ab{\theta(m,n)}=16(m^2+2m+4mn+2n^2+2n+1).$$
On the right side, Lemma \ref{lemma size closer friends} implies that
\begin{equation}
 c(\bfN_{\theta (m,n)t_m} +  v\bfN_{\theta (m-1,n)t_m'})  =   \ab{\theta (m,n)t_m} +  \ab{\theta (m-1,n)t_m'} = 16(m^2+2m+4mn+2n^2+2n+1).
\end{equation}
Thus we only need to show that  
\begin{equation} \label{eq inequality H}
\bfN_{\theta (m,n)}\undH_{t_m}  \geqh  \bfN_{\theta (m,n)t_m} +  v\bfN_{\theta (m-1,n)t_m'}.	
\end{equation}
By degree reasons it is enough to check 
\begin{align}
  \label{lemma N two terms A}   \bfN_{\theta (m,n)}\undH_{t_m} & \geqh \bfN_{\theta (m,n)t_m}, \\
   \label{lemma N two terms B}  \bfN_{\theta (m,n)}\undH_{t_m} &  \geqh  v\bfN_{\theta (m-1,n)t_m'}.
\end{align}
Furthermore, by the monotonicity of $\bfN_{\theta (m,n)}\undH_{t_m} $ ensured by Lemma \ref{lemma preserving monotonicity}, inequalities \eqref{lemma N two terms A} and \eqref{lemma N two terms B}  will follow from the inequalities
\begin{align}
  \label{lemma N two terms C}    G_{\theta (m,n)t_m} (\bfN_{\theta (m,n)}\undH_{t_m}) &  \geq   1 , \\
  \label{lemma N two terms D}     G_{\theta (m-1,n)t_m'} (\bfN_{\theta (m,n)}\undH_{t_m}) &  \geq  v ,
\end{align}
respectively. Inequality  \eqref{lemma N two terms C} is clear. On the other hand, a direct computation shows that
\begin{equation}
    G_{\theta (m-1,n)t_m'} (v^2\bfH_{\theta (m-1,n)t_m'}\undH_{t_m}) =v,
\end{equation}
which proves \eqref{lemma N two terms D}. The lemma is proved. 
\end{proof}

\begin{remark} \rm \label{definition negative index is set equal to zero}
Henceforth, we adopt the convention that $\bfN_{x\theta(m,n)y}$ and $\undH_{x\theta(m,n)y}$ are equal to zero whenever $m$ or $n$ is negative, for any $x,y\in W$. This will allow us to express some results more compactly and consistently.
\end{remark}

Using the same arguments as in the proof of Lemma \ref{lemma N mult by friend} we obtain the following four lemmas. 

\begin{lemma} \label{lemma first mult by N} 
	Let $m$ and $n$ be  positive integers. We have 
	\begin{align}
	  \bfN_{\theta (0,n)} \undH_{s_0} & = \bfN_{\theta(0,n)s_0} + v^2\bfN_{\theta(0,n-1)s_0};\\
	  \bfN_{\theta (m,0)} \undH_{t_m} &   = \bfN_{\theta(m,0)t_m}   +  v \bfN_{\theta(m-1,0)t_m'}.
	\end{align}
 Furthermore, $\bfN_{\theta(0,0)}\undH_{s_0} = \bfN_{\theta(0,0)s_0}$.
\end{lemma}

\begin{lemma}  \label{lemma mult by s0 on the left}
 For all $m\in \bbN$ we have
 \begin{equation}\label{eq lemma mult by s0 on the left} 
     \undH_{s_0}\bfN_{\theta(m,0)} = \bfN_{s_0\theta (m,0)} +v \bfN_{s_1\theta'(m-1,0)},
 \end{equation}
 where $\theta'(m-1,0)$ denotes the image of $\theta(m-1,0)$ under the automorphism $\vfi$. 
\end{lemma}

\begin{lemma}  \label{lemma mult N 0 theta }
Let $m$ be a positive integer. We have
\begin{equation}
    \bfN_{s_0\theta(m,0)}\undH_{t_m} = \bfN_{s_0\theta(m,0)t_m} + v \bfN_{s_0\theta(m-1,0)t_m'}.
\end{equation}
Furthermore, $\bfN_{s_0\theta(0,0) } \undH_{s_0} =\bfN_{s_0\theta(0,0)s_0} + v^2\bfN_{s_1s_0}$.
\end{lemma}

\begin{lemma}  \label{lemma mult 120 theta }
Let $m$ be a positive integer. We have
\begin{equation}
    \bfN_{s_1s_2s_0\theta(m,0)} \undH_{t_m} =  \bfN_{s_1s_2s_0\theta(m,0) t_m} + v  \bfN_{s_1s_2s_0\theta(m-1,0) t_m'}.
\end{equation}
\end{lemma}

The following lemma is proved using essentially the same argument as the one used in the proof of Lemma \ref{lemma N mult by friend}. This, however, is the first time so far in which we have a case where we must show $X \geqh Y$ with $Y$ consisting of two or more $\bfN$-terms, which leads some extra subtleties in the analysis.

\begin{lemma} \label{lemma almost N mult positive integers}
	Let $m$ and $n$ be positive integers. Then,
	\begin{equation}  \label{eq mult N A}
	\resizebox{0.9\hsize}{!}{$	
	\bfN_{\theta (m,n)}\undH_{t_m}\undH_{s_2}\undH_{t_m}= 2\bfN_{\theta (m,n)}\undH_{t_m} + \bfN_{\theta (m+1,n)} +\bfN_{\theta (m-1,n+1)} + \bfN_{\theta (m+1,n-1)} +\bfN_{\theta (m-1,n)}.
	$}
	\end{equation}
\end{lemma}

\begin{proof}
Let us denote by $L$ and $R$ the left-hand and right-hand side of \eqref{eq mult N A}, respectively. A repeated application of Corollary \ref{lem:abx} shows that
\begin{equation}
  c(L)=c(R) = 64(m^2+2m+4mn+2n^2+2n+1).  
\end{equation}
Therefore, to finish the proof we only need to show that $L\geqh R$.  Lemma \ref{lemma N mult by friend} implies
\begin{equation}
	R = 2\bfN_{\theta (m,n)t_m} +  2v\bfN_{\theta (m-1,n)t_m'}+\bfN_{\theta (m+1,n)} +\bfN_{\theta (m-1,n+1)} + \bfN_{\theta (m+1,n-1)} +\bfN_{\theta (m-1,n)}.
\end{equation}
Hence, by degree reasons, it is enough to check
\begin{align}
\label{eq one A}  L& \geqh    \bfN_{\theta (m+1,n)} \\
\label{eq one D}  L& \geqh \bfN_{\theta (m-1,n)}\\ 
\label{eq one B}  L & \geqh    \bfN_{\theta (m-1,n+1)} + 2\bfN_{\theta (m,n)t_m} \\
\label{eq one C}  L & \geqh  \bfN_{\theta (m+1,n-1)} +  2v\bfN_{\theta (m-1,n)t_m'} 
\end{align}
We now notice that $L$ is monotonic of height $\theta (m+1,n)$ by Lemma \ref{lemma preserving monotonicity}. It is clear that $G_{\theta (m+1,n)}(L)=1$. Then the monotonicity of $L$ shows \eqref{eq one A}. On the other hand, a direct computation reveals that   $	G_{\theta (m-1,n) }(v^3\bfH_{\theta (m-1,n)}\undH_{t_m}\undH_{s_2}\undH_{t_m} ) = v^2 +1$. Therefore,
\begin{equation}
G_{\theta (m-1,n) }(L) \geq 	G_{\theta (m-1,n) }(v^3\bfH_{\theta (m-1,n)}\undH_{t_m}\undH_{s_2}\undH_{t_m} ) = v^2+1\geq 1. 
\end{equation}
The above inequality together with the monotonicity of $L$ proves \eqref{eq one D}.

We now prove \eqref{eq one B}. Since we have two $\bfN$-terms that cannot be separated by degree reasons on the right-hand side, we need to proceed in a different way. Let us explain this more generally. Suppose we want to prove 
\begin{equation} \label{eq c1 and c2}
  L\geqh c_1v^i\bfN_x +c_2v^j\bfN_y,  
\end{equation} 
for some positive integers $c_1$ and $c_2$, and where $j-i=l(x)-l(y)$ (otherwise, degree reasons allow us to split the inequality in \eqref{eq c1 and c2} in two inequalities with only one $\bfN$-term on the right-hand side). Further suppose that $\leq x\ \cap  \leq y=\leq z\ \cup \leq w$ for some elements $z$ and $w$. Then, by the monotonicity of $L$,  \eqref{eq c1 and c2} will follow from
\begin{equation} \label{intersection reasons}
\begin{array}{rl}
    L&\geqh c_1v^i\bfN_x,    \\
     L&\geqh c_2v^j\bfN_y, \\
     L&\geqh (c_1+c_2)v^{i+l(x)-l(z)}\bfN_z, \text{  and}\\
      L&\geqh (c_1+c_2)v^{i+l(x)-l(w)}\bfN_w.
\end{array}
\end{equation}
Let us apply this general principle to prove \eqref{eq one B}. By the diagrammatic description of lower intervals given in \S\ref{section lower intervals} we have
\begin{equation}
	( \leq \theta (m-1,n+1)) \cap (\leq \theta (m,n)t_m ) = (\leq \theta(m-1,n)t_m's_2t_m).
\end{equation}
Therefore, we only need to prove 
\begin{equation}
G_{\theta (m-1,n+1)}	(L) \geq 1, \qquad    
G_{\theta (m,n)t_m} (L)  \geq 2 \qquad \mbox{and} \qquad 
G_{\theta(m-1,n)t_m's_2t_m}(L)    \geq 3v.
\end{equation}
These inequalities follow from the following identities which are obtained by a repeated application of  \eqref{mult H by Hunderline_s}.
\begin{equation} \label{eq1}
\begin{split}
G_{\theta(m-1,n+1)}( v\bfH_{\theta (m-1,n)t_m's_2} \undH_{t_m}\undH_{s_2}\undH_{t_m})  & =  1  \\
G_{\theta(m,n)t_m}( \bfH_{\theta (m,n)} \undH_{t_m}\undH_{s_2}\undH_{t_m}) & = 2 \\
G_{\theta(m-1,n)t_m's_2t_m } ( v^2\bfH_{\theta (m-1,n)t_m'} \undH_{t_m}\undH_{s_2}\undH_{t_m}  ) & =v\\
G_{\theta(m-1,n)t_m's_2t_m } ( v\bfH_{\theta (m-1,n)t_m's_2} \undH_{t_m}\undH_{s_2}\undH_{t_m}  ) & =2v.\\
\end{split}
\end{equation}
The proof of \eqref{eq one C} is similar, noting that \begin{equation}
   (\leq  \theta (m+1,n-1) ) \cap (  \leq \theta (m-1,n)t_m'  ) = (\leq \theta(m-1,n)) \cup (\leq \theta(m,n-1)t_m).
\end{equation} 
This description of the lower intervals allows one to apply the discussion leading up to (\ref{intersection reasons}); we leave the details to the reader.\end{proof}

\begin{proposition}  \label{proposition first identity N}
Let $m$ and $n$ be positive integers.  Then,
\begin{equation}
	\bfN_{\theta (m,n)} (\undH_{t_m}\undH_{s_2}\undH_{t_m}-2\undH_{t_m}) = \bfN_{\theta (m+1,n)} +\bfN_{\theta (m-1,n+1)} + \bfN_{\theta (m+1,n-1)} +\bfN_{\theta (m-1,n)}. 
\end{equation}
\end{proposition}

\begin{proof}
	The result is an immediate consequence of Lemma \ref{lemma almost N mult positive integers}.
\end{proof}
Similarly, we have the following proposition. 
\begin{proposition}   \label{proposition second identity N}
Let $m,n \in \bbN$. We have
\begin{align}
	\bfN_{\theta(m,0)}(\undH_{t_m}\undH_{s_2}\undH_{t_m}-2\undH_{t_m}) &  = \bfN_{\theta (m+1,0)} +\bfN_{\theta (m-1,1)}  +(1+v^2)\bfN_{\theta (m-1,0)}, \\	
\bfN_{\theta(0,n)}(\undH_{s_0}\undH_{s_2}\undH_{s_0}-2\undH_{s_0}) & = \bfN_{\theta (1,n)} +(1+v^2)\bfN_{\theta (1,n-1)}  +v^2\bfN_{\theta (1,n-2)}.
\end{align}
\end{proposition}

Propositions \ref{proposition first identity N} and  \ref{proposition second identity N} should be understood as a way to pass from  $\bfN_{\theta(m,n)}$ to $\bfN_{\theta(m+1,n)}$. We now focus on obtaining a similar statement but this time we want to pass from $\bfN_{\theta(0,n)}$ to $\bfN_{\theta(0,n+1)}$.

\begin{lemma}\label{lemma second mult by N} 
	Let $n$ be a positive integer. We have 
	\begin{equation} \label{help to prove mult N dos}
	\bfN_{\theta (0,n)s_0}\undH_{s_2} +v^2\bfN_{\theta (0,n-1)s_0s_2 }= \bfN_{\theta (0,n)s_0s_2} + \bfN_{\theta (0,n)} +v \bfN_{\theta (1,n-1)}.
	\end{equation}
\end{lemma}

\begin{proof}
	Once we notice that the element on the left-hand side of \eqref{help to prove mult N dos} is monotonic (of height $\theta(0,n)s_0s_2$) we can argue as in the proof of Lemma \ref{lemma N mult by friend}. 
\end{proof}

\begin{lemma}  \label{lemma mult N by four Hbars}
    Let $n$ be an integer greater than $2$. We have
    \begin{equation}  \label{eq to prove N mult by X2}
    \begin{array}{rc}
         &  \bfN_{\theta(0,n+1)} \\
         &  \\
         &  2\bfN_{\theta(0,n)s_0s_2}  + v^2\bfN_{\theta(0,n)} +\bfN_{\theta(2,n-1)} +  \\
        &  \\
     \bfN_{\theta(0,n)}\undH_{s_0}\undH_{s_2}\undH_{s_1}\undH_{s_2} +2v^4\bfN_{\theta(0,n-2)s_0s_2} =      & 3\bfN_{\theta(0,n)} +2v\bfN_{\theta(1,n-1)} +v^2\bfN_{\theta(2,n-2)} +  \\
         &  \\
         & v^{-2}\bfN_{\theta(0,n)} + 3v^2\bfN_{\theta(0,n-1)} +2v^3\bfN_{\theta(1,n-2)} + \\
         &  \\
         & \bfN_{\theta(0,n-1)} .
    \end{array}
    \end{equation}
\end{lemma}

\begin{proof}
Let $L$ and $R$ be the left and right sides of \eqref{eq to prove N mult by X2}, respectively. Corollary \ref{lem:abx} and Lemma \ref{lemma size closer friends} imply
\begin{equation}
   c(L) = c(R)= 8(36n^2+26n+18).
\end{equation}
Therefore, in order to prove $L=R$ we only need to show that $L\geqh R$.   By degree reasons, it is enough to prove the following inequalities
\begin{align}
\label{uno}    L \geqh &  \bfN_{\theta(0,n+1)},  \\  
 \label{dos}    L \geqh  & 2\bfN_{\theta(0,n)s_0s_2}  + v^2\bfN_{\theta(0,n)} +\bfN_{\theta(2,n-1)} ,        \\
\label{tres}    L \geqh   & 3\bfN_{\theta(0,n)} +2v\bfN_{\theta(1,n-1)} +v^2\bfN_{\theta(2,n-2)} ,        \\
\label{cuatro}   L \geqh   & v^{-2}\bfN_{\theta(0,n)} + 3v^2\bfN_{\theta(0,n-1)} +2v^3\bfN_{\theta(1,n-2)} ,         \\
\label{cinco}    L \geqh   &  \bfN_{\theta(0,n-1)}.        
\end{align}
Using Lemma \ref{lemma preserving monotonicity} we can see that $\bfN_{\theta(0,n)}\undH_{s_0}\undH_{s_2}\undH_{s_1}\undH_{s_2}$ is a monotonic element of height $\theta(0,n+1)$. Furthermore, 
$\bfN_{\theta(0,n-2)s_0s_2}$ is clearly monotonic. Since $\theta(0,n-2)s_0s_2 \leq \theta(0,n+1)  $ we conclude that $L$ is monotonic.

 An easy computation shows that
\begin{align}
\label{seis}    G_{\theta(0,n+1)}(\bfH_{\theta(0,n)}\undH_{s_0}\undH_{s_2}\undH_{s_1}\undH_{s_2}) & =   1,\\
\label{siete}    G_{\theta(0,n-1) } (v^3\bfH_{\theta(0,n-1)s_0}\undH_{s_0}\undH_{s_2}\undH_{s_1}\undH_{s_2} )  & =  1 +v^2. 
\end{align}
This allows us to conclude that $G_{\theta(0,n+1)}(L) \geq 1$ and $G_{\theta(0,n-1)}(L)\geq 1$. Then, the monotonicity of $L$ proves \eqref{uno} and \eqref{cinco}.

We now focus on \eqref{dos}. We notice that $\theta(0,n)$ is smaller than $\theta(0,n)s_0s_2$ and $\theta(2,n-1)$, but the latter elements are incomparable in Bruhat order. Furthermore, by the description of lower intervals given in \S\ref{section lower intervals} we have
\begin{equation}
    (\leq \theta(0,n)s_0s_2 )\cap (\leq \theta(2,n-1)) = (\leq \theta (1,n-1)s_1s_2).
\end{equation}
Finally, we notice that $\theta(0,n) \leq \theta (1,n-1)s_1s_2 $. Thus, by the monotonicity of $L$, in order to prove \eqref{dos} it is enough to show that
\begin{align}
\label{lala uno}    G_{\theta(0,n)s_0s_2}(L) & \geq 2; \\
\label{lala dos}    G_{\theta(2,n-1) }(L)    & \geq 1;\\
\label{lala faltante}   G_{\theta(1,n-1)s_1s_2}(L) & \geq 3v; \\
\label{lala tres}    G_{\theta(0,n)}  (L)        & \geq 4v^2.
\end{align} 
A direct computation shows that $G_{\theta(0,n)s_0s_2}\left(  \bfH_{\theta(0,n)} \undH_{s_0}\undH_{s_2}\undH_{s_1} \undH_{s_2}  \right) = 2$. This proves \eqref{lala uno}. Similarly, we have $G_{\theta(2,n-1)}\left( v \bfH_{\theta(0,n-1)s_0s_2s_1} \undH_{s_0}\undH_{s_2}\undH_{s_1} \undH_{s_2}  \right) =1$, which proves \eqref{lala dos}. On the other hand, we have
\begin{align}
    G_{\theta (1,n-1)s_1s_2}(v\bfH_{\theta(0,n-1)s_0s_2s_1} \undH_{s_0}\undH_{s_2}\undH_{s_1} \undH_{s_2}  ) &  = 2v ;\\
    G_{\theta (1,n-1)s_1s_2}(v^2 \bfH_{\theta(0,n-1)s_0s_2} \undH_{s_0}\undH_{s_2}\undH_{s_1} \undH_{s_2}  )   &  = v.
\end{align}
This proves \eqref{lala faltante}. Finally, we have
\begin{align}
    G_{\theta(0,n)}\left( v \bfH_{\theta(0,n-1)s_0s_2s_1} \undH_{s_0}\undH_{s_2}\undH_{s_1} \undH_{s_2}  \right)  & =v^2; \\
    G_{\theta(0,n)}\left( v^2 \bfH_{\theta(0,n-1)s_0s_2} \undH_{s_0}\undH_{s_2}\undH_{s_1} \undH_{s_2}  \right) & =v^2; \\
    G_{\theta(0,n)}\left( v^3 \bfH_{\theta(0,n-1)s_0} \undH_{s_0}\undH_{s_2}\undH_{s_1} \undH_{s_2}  \right) & =v^2; \\
    G_{\theta(0,n)}\left( v^2 \bfH_{d_{4n+2}} \undH_{s_0}\undH_{s_2}\undH_{s_1} \undH_{s_2}  \right) & =v^2.
\end{align}
These four equalities show \eqref{lala tres}. This completes the proof of inequality \eqref{dos}. The remaining two inequalities \eqref{tres} and \eqref{cuatro} are treated similarly. 
\end{proof}

We now are in a position to obtain a multiplication formula that relates $\bfN_{\theta(0,n)}$ and $\bfN_{\theta(0,n+1)}$. 

\begin{proposition}  \label{proposition third identity N}
Let $n$ be an integer greater than $2$. Then,
\begin{equation} \label{nueve}
\resizebox{1\hsize}{!}{$	
    	\bfN_{\theta(0,n)}(\undH_{s_0} \undH_{s_2}  \undH_{s_1} \undH_{ s_2} -2\undH_{s_0}\undH_{s_2}+1-(v+v^{-1})^2) = \bfN_{\theta (0,n+1)} +(1+v^2)\bfN_{\theta (0,n-1)}  +v^2\bfN_{\theta (2,n-2)}+\bfN_{\theta (2,n-1)}. 
    	$}
\end{equation}
\end{proposition}
\begin{proof}
By combining Lemma \ref{lemma first mult by N} and Lemma \ref{lemma second mult by N} we obtain 
\begin{equation}\label{ocho}
    \bfN_{\theta(0,n)} \undH_{s_0}\undH_{s_2} = \bfN_{\theta(0,n)s_0s_2}+\bfN_{\theta(0,n)} +v\bfN_{\theta(1,n-1)} +v^2\bfN_{\theta(0,n-1)} +v^3\bfN_{\theta(1,n-2)} -v^4\bfN_{\theta(0,n-2)s_0s_2}. 
\end{equation}
Hence, the result follows by combining Lemma \ref{lemma mult N by four Hbars} and \eqref{ocho}.
\end{proof}

So far we have proved several multiplicative identities involving $\bfN$-elements. These will be sufficient to obtain formulas for Kazhdan-Lusztig basis elements corresponding to elements in the big region in \S\ref{section big region}. We end this section by stating some extra multiplicative identities that will be useful to prove formulas for Kazhdan-Lusztig basis elements corresponding to elements in the thick region in \S\ref{section thick region}.  

 \begin{lemma}  \label{lemma some identities N de m0}
 For $m\geq 1$, the element $\left( \undH_{s_1}\undH_{s_2}\undH_{s_0}-\undH_{s_0}\right)\bfN_{\theta(m,0)}$ is equal to 
\begin{equation} \label{eq on the left for m positive thick}
    \bfN_{s_1s_2s_0\theta(m,0)}+(v^{-1}+2v)\bfN_{\theta(m,0)}+v(\bfN_{\theta(m-2,1)}+\bfN_{s_1\theta'(m-1,0)})+v^2(\bfN_{s_0\theta(m-2,0)}- \bfN_{s_1s_2s_0\theta(m-2,0)}).
\end{equation} 
Furthermore, we have
\begin{equation} \label{eq on the left for m=0 thick}
  \left( \undH_{s_1}\undH_{s_2}\undH_{s_0}-\undH_{s_0}\right)\bfN_{\theta(0,0)} =   \bfN_{s_1s_2s_0\theta (0,0)} + (v+v^{-1})\bfN_{\theta(0,0)}. 
\end{equation}
\end{lemma} 

\begin{proof}
Equation \eqref{eq on the left for m=0 thick} can be verified manually.  Similarly, we can check \eqref{eq on the left for m positive thick} for $m=1$.  We stress that in this case some terms vanish in accordance with Remark \ref{definition negative index is set equal to zero}. From now and on we fix  $m\geq 2$. We define
$ L:= \undH_{s_1}\undH_{s_2}\undH_{s_0} \bfN_{\theta(m,0)} + v^2 \bfN_{s_1s_2s_0\theta(m-2,0)}$ and 
\begin{equation}
    R:=\undH_{s_0}\bfN_{\theta(m,0)}+\bfN_{s_1s_2s_0\theta(m,0)}+(v^{-1}+2v)\bfN_{\theta(m,0)}+v(\bfN_{\theta(m-2,1)}+\bfN_{s_1\theta'(m-1,0)})+v^2\bfN_{s_0\theta(m-2,0)}.
\end{equation}
It is clear that the claim in the lemma is equivalent to $L=R$, which we will now prove. Using Lemma \ref{lemma mult by s0 on the left}, we can rewrite $R$ as
\begin{equation}
    R=\bfN_{s_0\theta(m,0)}+\bfN_{s_1s_2s_0\theta(m,0)}+(v^{-1}+2v)\bfN_{\theta(m,0)}+v(\bfN_{\theta(m-2,1)}+2\bfN_{s_1\theta'(m-1,0)})+v^2\bfN_{s_0\theta(m-2,0)}.
\end{equation}
A straightforward computation using Corollary \ref{lem:abx} and Lemma \ref{lemma cardinality big region not fundamental} shows
\begin{equation}
    c(L)=c(R)= 8(9m^2+17m+4).
\end{equation}
To complete the proof we only need to show that $L\geqh R$. By degree reasons,  this inequality will follow from the next four inequalities
\begin{align}
\label{eq unouno}    L &  \geqh  \bfN_{s_1s_2s_0\theta(m,0)} \\
\label{eq dosdos}      L &  \geqh  \bfN_{s_0\theta(m,0)} +2v\bfN_{\theta(m,0)} \\
\label{eq trestres}      L & \geqh  v^{-1}\bfN_{\theta(m,0)}  + 2v \bfN_{s_1 \theta'(m-1,0)} +v\bfN_{\theta(m-2,1)}   \\
\label{eq cuatrocuatro}      L & \geqh v^2\bfN_{s_0\theta(m-2,0)}.
\end{align}

We notice that the left version of Lemma \ref{lemma preserving monotonicity} shows that $L$ is monotonic.
Inequalities \eqref{eq unouno}, \eqref{eq dosdos} and \eqref{eq cuatrocuatro} are easily obtained, so we will focus on \eqref{eq trestres}. We notice that $\theta(m,0)$ is greater than $s_1 \theta'(m-1,0)$ and than $\theta(m-2,1)$ but the last two elements are incomparable in Bruhat order. By the diagrammatic description of the lower intervals given in \S\ref{section lower intervals} we obtain
\begin{equation}
    (\leq s_1 \theta'(m-1,0) ) \cap ( \leq \theta(m-2,1) )  = (\leq s_1s_2s_0\theta(m-2,0)).
\end{equation}
Therefore, by the monotonicity of $L$, \eqref{eq trestres} reduces to proving the following polynomial inequalities
\begin{align}
\label{pol ineq one}    G_{\theta(m,0)}(L)  & \geq   v^{-1}, \\
\label{pol ineq two}    G_{s_1\theta'(m-1,0)}(L) & \geq  3v,  \\  
\label{pol ineq three}       G_{\theta(m-2,1)} (L)  & \geq  2v, \\
\label{pol ineq four}       G_{s_1s_2s_0\theta(m-2,0)}(L) & \geq 4v^2. 
\end{align}
Set $X:=\undH_{s_1}\undH_{s_2}\undH_{s_0} $. A direct computation (using the left version of \eqref{mult H by Hunderline_s}) shows that $G_{\theta(m,0)}(X \bfH_{\theta(m,0)} ) = v^{-1}$, proving \eqref{pol ineq one}.  Similarly,  we have
\begin{equation}
 G_{s_1\theta'(m-1,0)} (X  (v\bfH_{s_2s_1\theta'(m-1,0)})  ) =  G_{s_1\theta'(m-1,0)} (X  (v^2\bfH_{s_1\theta'(m-1,0)})  ) =    G_{s_1\theta'(m-1,0)} (X  (v^3\bfH_{\theta'(m-1,0)})  )=v.
\end{equation}
These three equalities together prove \eqref{pol ineq two}. We now observe that
\begin{equation}
    G_{\theta(m-2,1)}(X(v^2\bfH_{s_1\theta'(m-1,0)})) =     G_{\theta(m-2,1)}(X(v^2\bfH_{\theta(m-2,1)})) = v.
\end{equation}
This proves \eqref{pol ineq three}. Finally, if we set $z:=s_1s_2s_0\theta(m-2,0)$  then we have
\begin{equation}
    G_{z}(X(v^2\bfH_{s_1\theta'(m-1,0)})) =   G_{z}(X(v^2\bfH_{\theta(m-2,1)})) = G_{z} (X(v^3\bfH_{\theta'(m-1,0)}))= v^2. 
\end{equation}
These three equalities show that $G_{z}(X\bfN_{\theta(m,0)}) \geq 3v^2$. Since $G_z(v^2\bfN_z)=v^2$, inequality \eqref{pol ineq four}  follows from the definition of $L$. The lemma is proved. 
\end{proof}

Using similar arguments we can prove the following result, recalling the definitions of $x_n$, $e_n$ and $u_n$ from the introduction. The details are left to the reader. 

\begin{lemma} \label{lema formulas N x e u}
Given an integer $k$ we set $f(k)= 3k+1$. Let $Y=\undH_{s_2}\undH_{s_1}\undH_{s_0}- \undH_{s_0}$. For $k\geq 2$, 
\begin{align}
    \bfN_{x_{f(k)}}Y&=\bfN_{x_{f(k+1)}}+\bfN_{x_{f(k-1)}} +  \bfN_{\theta(k-1,0)t_{k-1}}+\bfN_{s_1s_2s_0\theta(k-2,0)t_{k-2}}+\bfN_{s_1\theta'(k-2,0)t'_{k-2}}+v\bfN_{\theta(k-2,0)t_{k-2}}\\
    \bfN_{u_{f(k)}}Y&=\bfN_{u_{f(k+1)}}+ \bfN_{u_{f(k-1)}} + \bfN_{\theta(k-2,1)t_{k-2}}+(v+v^{-1})\bfN_{\theta(k-1,0)t_{k-1}}+v^2\bfN_{u_{f(k-1)}}\\ 
    \bfN_{e_{f(k)}}Y&=\bfN_{e_{f(k+1)}}+ \bfN_{e_{f(k-1)}} + \bfN_{s_0\theta(k-1,0)t_{k-1}}+\bfN_{s_1\theta'(k-1,0)t'_{k-1}}+v\bfN_{s_0\theta(k-2,0)t_{k-2}}+v\bfN_{s_1\theta'(k-2,0)t'_{k-2}}.
\end{align}
 In each case, the identity also holds if $Y$ is replaced by $Y'=\undH_{s_2}\undH_{s_0}\undH_{s_1}- \undH_{s_1}$.
\end{lemma}
\begin{remark}\rm
The reader may have noticed that in the second formula $\bfN_{u_{f(k-1)}}$ shows up twice with different coefficients. This may appear to be a slip, but it is not. The reason behind this choice will become clear in Section \ref{section thick region}.
\end{remark}

\section{Kazhdan-Lusztig basis in the big region}  \label{section big region}

In this section we provide an explicit formula for Kazhdan-Lusztig basis elements indexed by elements located in the big region. 
\subsection{Kazhdan-Lusztig basis in region \texorpdfstring{$\mathfrak{C}$}{ }  }  \label{subsection region C }
Let us recall a definition from the introduction.
\begin{definition}  \label{defin thetas hat}
Let $(m,n)\in \bbN^2$.  We define
\begin{equation}
    \Supp (m,n) = \left\{ \begin{array}{ll}
      \{  (m-2i,n-j) \in \bbN^2 \, |\,  i,j\in \bbN \},     & \mbox{if } m \mbox{ is odd;}  \\
       \{  (m-2i,n-j) \in \bbN^2 \, |\,  i,j\in \bbN \} -  \{ (0,b) \, \ \, | \,  b\not \equiv n \mod 2  \},  & \mbox{if } m \mbox{ is even. } 
    \end{array}   \right.
\end{equation}
We also define
\begin{equation}
\Hhat_{\theta (m,n)}: = \sum_{(a,b)\in \Supp (m,n)} v^{(m-a)+2(n-b)} \bfN_{\theta (a,b)}.
\end{equation}
\end{definition}

The first goal in this section is to prove that $\undH_{\theta(m,n)} = \Hhat_{\theta (m,n)}$ for all $(m,n)\in \bbN^2$. To do this we need several preliminary lemmas.

\begin{lemma}  \label{lemma nice decomposition}
Let $m$ and $n$ be positive integers. Then,
\begin{equation} \label{eq case B one}
    \Hhat_{\theta(m,n)} = \bfN_{\theta(m,n)} + v^2 \Hhat_{\theta(m-2,n)} +v^{2} \Hhat_{\theta (m,n-1)} -v^4\Hhat_{\theta(m-2,n-1)}.
\end{equation}
Furthermore, 
\begin{equation} \label{eq nice decom two}
    \Hhat_{\theta(0,n)} = \bfN_{\theta(0,n)} + v^4\Hhat_{\theta(0,n-2)}  \quad \mbox{ and } \quad  \Hhat_{\theta(m,0)} = \bfN_{\theta(m,0)} + v^2\Hhat_{\theta(m-2,0)}.
\end{equation}
\end{lemma}
\begin{proof}
The result follows directly from the definition of the elements $\Hhat_{\theta(m,n)}$ once we recall our convention about negative indices in Remark \ref{definition negative index is set equal to zero}. 
\end{proof}

\begin{lemma}  \label{lemma crecer en la primera}
Let $(m,n) \in \bbN^2$. Then
\begin{equation}  \label{eq six}
	\Hhat_{\theta(m,n)}(\undH_{t_m}\undH_{s_2}\undH_{t_m}-2\undH_{t_m}) = \Hhat_{\theta (m+1,n)} +\Hhat_{\theta (m-1,n+1)} + \Hhat_{\theta (m+1,n-1)} +\Hhat_{\theta (m-1,n)}.
\end{equation}
\end{lemma}

\begin{proof}
A direct computation shows that \eqref{eq six} holds for the pairs $(0,0)$, $(1,0)$,  $(0,1)$ , $(2,0)$ , $(1,1)$ and $(0,2)$. We now fix $u\geq 3$ and assume that \eqref{eq six} holds for all pairs $(m',n')$ with $m'+n'< u$. We fix a pair $(m,n)$ with $m+n=u$. We split the proof in four  cases. We recall that the set of $D_R(\theta(a,b))$ only depends on the parity of $a$. For the sake of brevity we write 
\begin{equation}
    X_{\uparrow}^1 := \undH_{t_m}\undH_{s_2}\undH_{t_m}-2\undH_{t_m}.
\end{equation}
\textbf{Case A.} Suppose that $m\geq 2$ and $n\geq 1$. \\
Using Proposition \ref{proposition first identity N}, Lemma \ref{lemma nice decomposition} and our inductive hypothesis we get 
\begin{equation} \label{eq case A}
\Hhat_{\theta(m,n)}X_{\uparrow}^1  =   
    \begin{array}{cccccccc}
           \bfN_{\theta(m+1,n)}    &+& \bfN_{\theta(m-1,n+1)}    &+& \bfN_{\theta(m+1,n-1)}    &+& \bfN_{\theta(m-1,n)} &+       \\
          v^2\Hhat_{\theta (m-1,n)}  &+& v^2\Hhat_{\theta (m-3,n+1)} &+& v^2\Hhat_{\theta (m-1,n-1)} &+ & v^2\Hhat_{\theta (m-3,n)} &+ \\
         v^2\Hhat_{\theta (m+1,n-1)} &+ & v^2\Hhat_{\theta (m-1,n)}   &+ & v^2 \Hhat_{\theta (m+1,n-2)} &+ & v^2\Hhat_{\theta (m-1,n-1)}&- \\
         v^4\Hhat_{\theta (m-1,n-1)} &- &v^4\Hhat_{\theta (m-3,n)}   &-& v^4 \Hhat_{\theta (m-1,n-2)} &-&v^4\Hhat_{\theta (m-3,n-1)} .&  
    \end{array}
\end{equation}
Then, by adding by columns in the right-hand side of \eqref{eq case A} and using Lemma \ref{lemma nice decomposition} we obtain \eqref{eq six}.\\\\ 
\textbf{Case B.} Suppose that $m=1$ and $n\geq 2$. \\  
The main difference with respect to the above \textbf{Case A} is that when we decompose $\Hhat_{\theta(1,n)}$ using Lemma \ref{lemma nice decomposition} there are two terms that do not appear. More concretely, we have
\begin{equation}
  \Hhat_{\theta(1,n)} = \bfN_{\theta(1,n)}  +v^{2} \Hhat_{\theta (1,n-1)} .  
\end{equation}
Using Proposition \ref{proposition first identity N}   and our inductive hypothesis we get 
\begin{equation} \label{eq case B}
    \Hhat_{\theta(1,n)}X_{\uparrow}^1  =   
    \begin{array}{cccccccc}
           \bfN_{\theta(2,n)}    &+&\bfN_{\theta(0,n+1)}    &+& \bfN_{\theta(2,n-1)}    &+& \bfN_{\theta(0,n)} &+       \\
         v^2\Hhat_{\theta (2,n-1)} &+& v^2\Hhat_{\theta (0,n)}   &+& v^2 \Hhat_{\theta (2,n-2)} &+& v^2\Hhat_{\theta (0,n-1)}. &
    \end{array}
\end{equation}
We now use \eqref{eq nice decom two} to rewrite $\bfN_{\theta(0,n+1)}$ and $\bfN_{\theta(0,n)}$ as
\begin{equation} \label{eq case B two}
    \bfN_{\theta(0,n+1)}= \Hhat_{\theta(0,n+1)}-v^4\Hhat_{\theta(0,n-1)}
    \quad \mbox{ and }  \quad 
    \bfN_{\theta(0,n)}= \Hhat_{\theta(0,n)}-v^4\Hhat_{\theta(0,n-2)}
\end{equation}
within \eqref{eq case B}. Another application of Lemma \ref{lemma nice decomposition} gives us \eqref{eq six}.\\\\
\textbf{Case C.} Suppose that $m=0$ and $n\geq 3$.\\
As in the previous cases we use Lemma \ref{lemma nice decomposition} in order to get
\begin{equation}
  \Hhat_{\theta(0,n)}  = \bfN_{\theta(0,n)}+v^4\Hhat_{\theta(0,n-2)}.
\end{equation}
Using Proposition \ref{proposition second identity N}  and our inductive hypothesis we get
\begin{equation}
   \Hhat_{\theta(0,n)}X_{\uparrow}^1  = ( \bfN_{\theta(1,n)}+v^2\bfN_{\theta(1,n-1)} +v^4\Hhat_{\theta(1,n-2)} ) +  ( \bfN_{\theta(1,n-1)}+v^2\bfN_{\theta(1,n-2)} +v^4\Hhat_{\theta(1,n-3)} ).  
\end{equation}
Finally, two applications of Lemma \ref{lemma nice decomposition} give \eqref{eq six}.\\\\
\textbf{Case D. } $m\geq 3$ and $n=0$. \\
By combining \eqref{eq nice decom two}, Proposition \ref{proposition second identity N} and our inductive hypothesis we get
\begin{equation}
\resizebox{\linewidth}{!}{$
 \begin{array}{rl}
    \Hhat_{\theta(m,0)}X_{\uparrow}^1  =   &  ( \bfN_{\theta(m+1,0)}+v^2\Hhat_{\theta(m-1,0)} ) + ( \bfN_{\theta(m-1,1)} +v^2\bfN_{\theta(m-1,0)} +v^2 \Hhat_{\theta(m-3,1)} ) +   ( \bfN_{\theta(m-1,0)}+v^2\Hhat_{\theta(m-3,0)}  )  \\
     =   & \Hhat_{\theta(m+1,0)} +\Hhat_{\theta(m-1,1)} + \Hhat_{\theta(m-1,0)}. 
 \end{array}
 $}
\end{equation}
This completes the proof of the lemma.
\end{proof}

\begin{lemma}  \label{lemma crecer en la segunda}
For all $n\in \bbN$ we have
\begin{equation}  \label{eq seven}
	\Hhat_{\theta(0,n)} (\undH_{s_0}\undH_{s_2}\undH_{s_1}\undH_{s_2} -2\undH_{s_0}\undH_{s_2}+1-(v+v^{-1})^2) = \Hhat_{\theta (0,n+1)} +  \Hhat_{\theta (2,n-1)}  +\Hhat_{\theta (0,n-1)}.
\end{equation}
\end{lemma}
\begin{proof}
We proceed by induction on $n$. The result can be checked directly for $n\leq 3$, so we assume that $n>3$ and that \eqref{eq seven} holds for all $n'<n$. We write
\begin{equation}
    X_{\uparrow}^2:=\undH_{s_0}\undH_{s_2}\undH_{s_1}\undH_{s_2} -2\undH_{s_0}\undH_{s_2}+1-(v+v^{-1})^2.
\end{equation}
By combining  Proposition \ref{proposition third identity N}, Lemma \ref{lemma nice decomposition} and our inductive hypothesis we obtain
\begin{equation}
  \begin{array}{rl}
     &\left(\bfN_{\theta(0,n+1)} +v^4\Hhat_{\theta(0,n-1)}\right) +       \\
    \Hhat_{\theta(0,n)}X_\uparrow^2    = &\left(\bfN_{\theta(2,n-1)} + v^2\bfN_{\theta(2,n-2)} + v^2\bfN_{\theta(0,n-1)} +v^4\Hhat_{\theta(2,n-3)}\right)       +    \\
     & \left( \bfN_{\theta(0,n-1)} +v^4\Hhat_{\theta(0,n-3)} \right)  \\
     =&  \Hhat_{\theta (0,n+1)} +  \Hhat_{\theta (2,n-1)}  +\Hhat_{\theta (0,n-1)},
    \end{array}
\end{equation}
where for the last equality we have used \eqref{eq nice decom two} and the identity     
\begin{equation}
\Hhat_{\theta (2,n-1)}=   \bfN_{\theta(2,n-1)} + v^2\bfN_{\theta(2,n-2)} + v^2\bfN_{\theta(0,n-1)} +v^4\Hhat_{\theta(2,n-3)}, 
\end{equation}
which follows directly from the definition of $\Hhat_{\theta (2,n-1)}$.
\end{proof}

\begin{theorem} \label{teo thetas}
For all $(m,n)\in \bbN^2$ we have $	\Hhat_{\theta (m,n)}=\undH_{\theta(m,n)}$. 
\end{theorem}

\begin{proof}
By definition of $\Hhat_{\theta(m,n)}$ it is clear that $G_{\theta(m,n)} (\Hhat_{\theta(m,n)})=1$ and that $G_x(\Hhat_{\theta(m,n)})\in v\mathbb{N}[v]$ for all $x<\theta(m,n)$. Therefore, in order to prove the theorem it is enough to show that $\Hhat_{\theta(m,n)}$ is self-dual. To do this we proceed by induction on $m+n$. The result is clear for the pair $(0,0)$ since $\undH_{\theta(0,0)}=\bfN_{\theta(0,0)} = \Hhat_{\theta(0,0)}$.  We fix $u\in \bbN$ and assume that $\Hhat_{\theta(m,n)}$ is self-dual for all pairs $(m,n)$ such that $m+n\leq u$. By Lemma \ref{lemma crecer en la primera} we know that
\begin{equation}  \label{eq six two}
	\Hhat_{\theta(m,n)}(\undH_{t_m}\undH_{s_2}\undH_{t_m}-2\undH_{t_m}) = \Hhat_{\theta (m+1,n)} +\Hhat_{\theta (m-1,n+1)} + \Hhat_{\theta (m+1,n-1)} +\Hhat_{\theta (m-1,n)}.
\end{equation}
We notice that $\undH_{t_m}\undH_{s_2}\undH_{t_m}-2\undH_{t_m}$ is self-dual. Also, the terms $\Hhat_{\theta(m,n)}$, $\Hhat_{\theta (m-1,n+1)}$, $\Hhat_{\theta (m+1,n-1)}$ and $\Hhat_{\theta (m-1,n)}$ are self-dual by our inductive hypothesis. We can thus conclude that  $\Hhat_{\theta (m+1,n)}$ is self-dual as well. This shows that $\Hhat_{\theta(a,b)}$ is self-dual for all the pairs $(a,b)$ with $a+b=u+1$ and $a\neq 0$.

It remains to show that  $\Hhat_{\theta(0,u+1)}$ is self-dual. By Lemma \ref{lemma crecer en la segunda} we know that
\begin{equation}
  	\Hhat_{\theta(0,u)} (\undH_{s_0}\undH_{s_2}\undH_{s_1}\undH_{s_2} -2\undH_{s_0}\undH_{s_2}+1-(v+v^{-1})^2) = \Hhat_{\theta (0,u+1)} +  \Hhat_{\theta (2,u-1)}  +\Hhat_{\theta (0,u-1)}.  
\end{equation}
We notice that $\undH_{s_0}\undH_{s_2}\undH_{s_1}\undH_{s_2} -2\undH_{s_0}\undH_{s_2}+1-(v+v^{-1})^2$ is self-dual and that $\Hhat_{\theta (0,u-1)}$ is also self-dual by our inductive hypothesis. Although the self-duality of $\Hhat_{\theta (2,u-1)}$ is not covered by our inductive hypothesis, we have already proved its self-duality since $(2,u-1)$ is a pair $(a,b)$ with $a\neq 0$.  As before, we conclude that $\Hhat_{\theta (0,u+1)}$ is self-dual as well, and the theorem is proved. 
\end{proof}


\begin{remark}\rm
There is an alternative  way to prove Theorem \ref{teo thetas} using the results in \cite{nicoleodavid}. We prefer to present the proof just given in order to keep this paper self-contained and cite only available sources.
\end{remark}

\begin{theorem}\label{teo amigos de theta}
Let $(m,n)\in \bbN^2$.  Then,
\begin{align}
   \undH_{\theta(m,n)}\undH_{t_m}         & =  \undH_{\theta(m,n)t_m}  \label{eq amigos uno} \\
   \undH_{\theta(m,n)t_m}\undH_{s_2}      & =  \undH_{\theta(m,n)t_ms_2} +\undH_{\theta(m,n)} \label{eq amigos dos} \\
   \undH_{\theta(m,n)t_ms_2}\undH_{t_m'}  & =  \undH_{\theta(m,n)t_ms_2t_m'} + \undH_{\theta(m,n)t_m}.  \label{eq amigos cuatro}
\end{align}
\end{theorem}

\begin{proof}
 By Lemma \ref{lem: less theta mejores amigos}, we have   $xt_m>x $ for all $x\lessdot \theta(m,n)$. On the other hand, Theorem \ref{teo thetas} implies that $\mu(x,\theta(m,n))\neq 0$ if and only if  $x\lessdot \theta(m,n)$. Hence, \eqref{eq mult recurrence} implies \eqref{eq amigos uno}. 
 
In order to prove \eqref{eq amigos dos} we begin by noticing that Lemma \ref{lem: less theta mejores amigos} implies that the only element in $\lessdot \theta(m,n)t_m$ satisfying $xs_2<x$
 is $\theta(m,n)$. Therefore, a combination of \eqref{eq mult recurrence} and Lemma \ref{useful lemma} yields
\begin{equation} \label{amigos theta eq 0}
    \undH_{\theta(m,n)t_m}\undH_{s_2}  = \undH_{\theta(m,n)t_ms_2} +\undH_{\theta(m,n)} + \sum_{ \substack{y\in W \\  D_R(y)=\{t_m,t_m',s_2 \}  }} {m_y} \undH_y,
\end{equation}
where $m_y\in \bbN$. We stress that $D_R(\theta(m,n)t_m)=\{t_m,t_m'\}$. Since there is no  $y\in W$ with $D_R(y)=\{t_m,t_m',s_2 \} $ we conclude that the sum in \eqref{amigos theta eq 0} is empty. This gives  \eqref{eq amigos dos}. 

 We now prove \eqref{eq amigos cuatro}. Once again, Lemma \ref{lem: less theta mejores amigos} shows that the only element in $\lessdot \theta (m,n) t_ms_2$ satisfying $xt_m'<x$ is $\theta(m,n)t_m$. Hence \eqref{eq mult recurrence} and Lemma \ref{useful lemma}  give
\begin{equation}  \label{amigos theta 2 y medio}
  \undH_{\theta(m,n)t_ms_2}\undH_{t_m'}    =  \undH_{\theta(m,n)t_ms_2t_m'} + \undH_{\theta(m,n)t_m} +  \sum_{y\in Y } {n_y} \undH_y,
\end{equation}
where $n_y\in \mathbb{N}$ and $Y=\{ y\in W \, |\, D_R(y)=\{t_m',s_2 \}, \, D_L(y)=\{s_1,s_2 \} \}$ . 

By the discussion at the beginning of \S\ref{section lower intervals}  we conclude that  $y\in Y$ if and only if there exist   $(a,b) \in \bbN^2$ with $a  \equiv m \bmod 2  $ such that $y= \theta (a,b)$. We now multiply \eqref{amigos theta 2 y medio} on the right by $\undH_{t_m}$ and using \eqref{eq v+v^-1}  and  \eqref{eq amigos uno} we obtain
\begin{equation} \label{amigos theta eq 3}
  \undH_{\theta(m,n)t_ms_2}\undH_{t_m'}\undH_{t_m}    =  \undH_{\theta(m,n)t_ms_2t_m'}\undH_{t_m} + (v+v^{-1})\undH_{\theta(m,n)t_m} +  \sum_{ y\in Y} {n_y} \undH_{yt_m}. 
 \end{equation} 
 On the other hand, we combine  \eqref{eq v+v^-1}, Lemma \ref{lemma crecer en la primera}, Theorem \ref{teo thetas}, \eqref{eq amigos uno} and \eqref{eq amigos dos} to obtain   
\begin{equation}\label{amigos theta eq 4}
\resizebox{.9\linewidth}{!}{$
  \undH_{\theta(m,n)t_ms_2}\undH_{t_m}\undH_{t_m'}    =  (v+v^{-1})\undH_{\theta(m,n)t_m} + \undH_{\theta (m+1,n)t_m'} +\undH_{\theta (m-1,n+1)t_m'} + \undH_{\theta (m+1,n-1)t_m'} +\undH_{\theta (m-1,n)t_m'}. 
  $}
\end{equation}
We now notice that $\undH_{t_m}\undH_{t_m'} =\undH_{t_m'}\undH_{t_m}$ and conclude that the right-hand side of \eqref{amigos theta eq 3} and \eqref{amigos theta eq 4} coincide. Therefore, after cancelling out the term $(v+v^{-1})\undH_{\theta(m,n)t_m}$ we get
\begin{equation}\label{amigos theta eq 5}
\undH_{\theta(m,n)t_ms_2t_m'}\undH_{t_m}  +\sum_{ y\in Y} {n_y} \undH_{yt_m} =   \undH_{\theta (m+1,n)t_m'} +\undH_{\theta (m-1,n+1)t_m'} + \undH_{\theta (m+1,n-1)t_m'} +\undH_{\theta (m-1,n)t_m'}.
\end{equation}
Suppose that $n_y\neq 0$ for some $y=\theta(a,b)\in Y$. Since $a\equiv m \bmod 2$ we know that $a\neq m\pm 1$. In particular, $\theta (a,b)t_m$ is not equal to any of the elements that index the Kazhdan-Lusztig basis elements on the right-hand side of  \eqref{amigos theta eq 5}. We reach a contradiction and conclude that $n_y=0$ for all $y\in Y$. If we look back at \eqref{amigos theta 2 y medio} then we see that it reduces to \eqref{eq amigos cuatro}.
\end{proof}

\subsection{Kazhdan-Lusztig basis for the whole big region} \label{subsection whole big region}

Let us remark that Theorem \ref{teo thetas} and Theorem \ref{teo amigos de theta} give us all the Kazhdan-Lusztig basis elements indexed by elements in the fundamental region $\mf{C}$ and also in the region $\vfi(\mf{C})$. In this section we extend this result to all the elements in the big region. We stress that we only need specify the description of the Kazhdan-Lusztig basis elements for the regions $s_0\mf{C}$, $s_2s_0\mf{C}$ and $s_1s_2s_0\mf{C}$ since the description for the elements in the regions $s_1\vfi(\mf{C})$, $s_2s_1\vfi(\mf{C})$ and $s_0s_2s_1\vfi(\mf{C})$ will follow from the above by applying the automorphism $\varphi$.

\begin{theorem} \label{teo big region all}
Let $(m,n)\in \bbN^2$.  Then, 
\begin{align}
\label{eq amigos leganos A} \undH_{s_0} \undH_{\theta(m,n)}    &  = \undH_{s_0\theta(m,n)} \\ 
\retainlabel{eq amigos leganos B}\undH_{s_0} \undH_{\theta(m,n)t_m}    & = \undH_{s_0\theta(m,n)t_m}  \\ 
\retainlabel{eq amigos leganos C} \undH_{s_0} \undH_{\theta(m,n)t_ms_2}    & =\undH_{s_0\theta(m,n)t_ms_2}    \\ 
\label{eq amigos leganos D} \undH_{s_0} \undH_{\theta(m,n)t_ms_2t_m'}      & =    \undH_{s_0\theta(m,n)t_ms_2t_m'}\\
\label{eq amigos leganos E} \undH_{s_2} \undH_{s_0\theta(m,n)}     & =   \undH_{s_2s_0\theta(m,n)} +\undH_{\theta(m,n)}  \\ 
\retainlabel{eq amigos leganos F}\undH_{s_2} \undH_{s_0\theta(m,n)t_m}     & =  \undH_{s_2s_0\theta(m,n)t_m} +\undH_{\theta(m,n)t_m}   \\ 
\retainlabel{eq amigos leganos G}\undH_{s_2} \undH_{s_0\theta(m,n)t_ms_2}    & =  \undH_{s_2s_0\theta(m,n)t_ms_2} +\undH_{\theta(m,n)t_ms_2}   \\ 
\label{eq amigos leganos H}\undH_{s_2} \undH_{s_0\theta(m,n)t_ms_2t_m'}     & =    \undH_{s_2s_0\theta(m,n)t_ms_2t_m'} +\undH_{\theta(m,n)t_ms_2t_m'} \\
\label{eq amigos leganos I}\undH_{s_1} \undH_{s_2s_0\theta(m,n)}     & = \undH_{s_1s_2s_0\theta(m,n)}  + \undH_{s_0\theta(m,n)}     \\ 
\retainlabel{lalalalalalalal}   \undH_{s_1} \undH_{s_2s_0\theta(m,n)t_m} &  = \undH_{s_1s_2s_0\theta(m,n)t_m}  + \undH_{s_0\theta(m,n)t_m}    \\ 
\retainlabel{lalalalalalalal A}  \undH_{s_1} \undH_{s_2s_0\theta(m,n)t_ms_2}     & = \undH_{s_1s_2s_0\theta(m,n)t_ms_2}  + \undH_{s_0\theta(m,n)t_ms_2}   \\ 
\label{eq amigos leganos L} \undH_{s_1} \undH_{s_2s_0\theta(m,n)t_ms_2t_m'}     &  =   \undH_{s_1s_2s_0\theta(m,n)t_ms_2t_m'}  + \undH_{s_0\theta(m,n)t_ms_2t_m'} . 
\end{align}
\end{theorem}

\begin{proof}
We fix $(m,n)\in \bbN^2$ and let $\mathcal{R}:=\{e,t_m,t_ms_2,t_ms_2t_m'\}$. By our description of the set of coatoms for lower intervals  in Lemma \ref{lem: less theta mejores amigos} we have $s_0x>x$  for all $y\in \mathcal{R}$ and $x \lessdot \theta(m,n)y$. Therefore, since $D_{{L}}(\theta(m,n)y)=\{s_1,s_2\}$ for all $y\in \mathcal{R}$ and there is no element $z\in W$ with $D_{L}(z)=\{s_0,s_1,s_2\}$, we obtain via the left version of \eqref{eq mult recurrence} and Lemma \ref{useful lemma} all the identities \eqref{eq amigos leganos A}--\eqref{eq amigos leganos D}. We pause our proof for a moment to make an observation that will be useful for the rest of the proof.

\begin{remark}\rm\label{rem adentro de proof}
We remark that \eqref{eq amigos leganos A}--\eqref{eq amigos leganos D}  together are equivalent to the following: 
For any $z\in \mf{C}$ we have $\undH_{s_0}\undH_{z} = \undH_{s_0z}$. Of course, by applying $\vfi $ to the above equality we get $\undH_{s_1}\undH_{u} = \undH_{s_1u}$ for all $u\in \vfi (\mf{C})$.
\end{remark}

We continue with the proof of the next four identities. Now, Lemma \ref{lemma amigos de theta no cercanos} allows us to conclude that for each $y \in \mathcal{R}$ the only element in $\lessdot s_0\theta(m,n)y$ satisfying $s_2x<x $  is $\theta(m,n)y$. Since $D_{{L}}(s_0\theta(m,n)y)=\{s_0,s_1\}$ for all $y\in \mathcal{R}$, another application of the left version of \eqref{eq mult recurrence}  and Lemma \ref{useful lemma} give us \eqref{eq amigos leganos E}--\eqref{eq amigos leganos H}. 

For the last four identities we need to work a little harder. In this case we have  $D_{L}(s_2s_0\theta(m,n)y) =\{s_2\}$ for all $y \in \mathcal{R}$.  Therefore, the left version of \eqref{eq mult recurrence} and Lemma \ref{useful lemma} only allow us to conclude that for each $y\in \mathcal{R} $ we have
\begin{equation} \label{casi final amigos A}
    \undH_{s_1} \undH_{s_2s_0\theta(m,n)y} = \undH_{s_1s_2s_0\theta(m,n)y}  + \undH_{s_0\theta(m,n)y} + \sum_{z\in \mf{C}} m_{z}^y \undH_{z},  
\end{equation}
where $m_{z}^y\in \mathbb{N}$. 

We now multiply \eqref{casi final amigos A} by $\undH_{s_0}$ on the left, and using \eqref{eq v+v^-1} together with Remark \ref{rem adentro de proof} we get
\begin{equation} \label{casi final amigos B}
   \undH_{s_0} \undH_{s_1} \undH_{s_2s_0\theta(m,n)y} = \undH_{s_0}\undH_{s_1s_2s_0\theta(m,n)y}  +(v+v^{-1}) \undH_{s_0\theta(m,n)y} + \sum_{z\in \mf{C}} m_{z}^y \undH_{s_0z}.
\end{equation}
We now compute $\undH_{s_0} \undH_{s_1} \undH_{s_2s_0\theta(m,n)y}$ in a different way using the fact that $\undH_{s_0}\undH_{s_1}=\undH_{s_1}\undH_{s_0} $. Lemma \ref{lemma amigos de theta no cercanos} shows that for all $y\in \mathcal{R}$ the only element in $\lessdot s_2s_0\theta(m,n)y$ satisfying $s_0x<x$ is $s_0\theta(m,n)y$. On the other hand, we notice that $s_0s_2s_0\theta(m,n)y$  belongs to $\vfi(\mf{C})$ since $D_L(s_0s_2s_0\theta(m,n)y)=\{s_0, s_2\}$. Using these two facts together with the left version of \eqref{eq mult recurrence} and  Lemma \ref{useful lemma} we obtain
\begin{equation} \label{casi final amigos C}
    \begin{array}{rl}
    \undH_{s_1} \undH_{s_0} \undH_{s_2s_0\theta(m,n)y}  =   &  \undH_{s_1}\left(  \undH_{s_0s_2s_0\theta(m,n)y} + \undH_{s_0\theta(m,n)y}+\displaystyle \sum_{u\in \vfi (\mf{C}) } n_{u}^y  \undH_{u}    \right) \\
    &  \\
        = &   \undH_{s_1s_0s_2s_0\theta(m,n)y} +(v+v^{-1}) \undH_{s_0\theta(m,n)y}+\displaystyle \sum_{u\in \vfi (\mf{C}) } n_{u}^y  \undH_{s_1u}    ,
    \end{array}
\end{equation}
for some $n_u^y\in \mathbb{N}$.

We now fix $y\in \mathcal{R}$ and suppose that $m_{z}^y\neq 0$ for some $z\in \mf{C}$. By comparing the right-hand sides of \eqref{casi final amigos B} and \eqref{casi final amigos C}, we conclude that $s_0z \in s_1\vfi(\mf{C})$, which is absurd since $s_{0}\mf{C} \cap s_1\vfi(\mf{C}) = \emptyset$. Therefore, $m_z^y=0$ for all $z\in \mf{C}$ and \eqref{casi final amigos A} reduces to 
\begin{equation} \label{casi final amigos D}
    \undH_{s_1} \undH_{s_2s_0\theta(m,n)y} = \undH_{s_1s_2s_0\theta(m,n)y}  + \undH_{s_0\theta(m,n)y} .
\end{equation}
Since this identity holds for all $y\in \mathcal{R}$ we obtain \eqref{eq amigos leganos I}--\eqref{eq amigos leganos L}.
\end{proof}

\begin{remark}\rm
The formulas in Theorem \ref{teo amigos de theta} and Theorem \ref{teo big region all} are apparently different from the ones presented in Theorem \ref{theorem intro big}. It is an easy exercise (which is left to the reader) to show that both are equivalent.
\end{remark}


\subsection{Extra explicit formulas needed in the sequel}
Formulas given in  \S\ref{subsection region C } and \S\ref{subsection whole big region} enable the efficient computation of any Kazhdan-Lusztig polynomial $h_{x,w}(v)$ for all $x\in W$ and all $w$ in the big region. In order to obtain formulas for the thick region in \S\ref{section thick region}, it will be convenient to have more explicit expressions on hand for certain Kazhdan-Lusztig basis elements in the big region. This is the main goal of this section.

\begin{lemma} \label{lemma explicit description for theta s}
Let $m\geq 0$. We have $ \undH_{\theta(m,0)t_m} = \bfN_{\theta(m,0)t_m} + v\undH_{\theta(m-1,0)t_{m-1}}$.
\end{lemma}
\begin{proof}
We proceed by induction on $m$. The cases $m=0$ and $m=1$ are easily checked. We assume $m>1$ and that the lemma holds for $m-1$. Using Lemma \ref{lemma first mult by N}, the explicit formula for $ \undH_{\theta(m,0)}$, \eqref{eq amigos uno}, our inductive hypothesis and the fact that $t_m=t_{m-2}$ we obtain 
\begin{equation}
    \begin{array}{rl}
         \undH_{\theta(m,0)t_m}  = &   \undH_{\theta(m,0)}\undH_{t_{m}}  \\
                                 = &    (\bfN_{\theta(m,0)} +v^2 \undH_{\theta(m-2,0)}  )  \undH_{t_{m}} \\
                                 = &   \bfN_{\theta(m,0)t_m} + v\bfN_{\theta(m-1,0)t_{m-1}} +v^2\undH_{\theta(m-2,0)t_m} \\
                                 = &  \bfN_{\theta(m,0)t_m} + v( \bfN_{\theta(m-1,0)t_{m-1}}  + v \undH_{\theta(m-2,0)t_{m-2}}  )\\
                                 = & \bfN_{\theta(m,0)t_m} + v \undH_{\theta(m-1,0)t_{m-1}}.
    \end{array}
\end{equation}
\end{proof}

The proof of the following lemma is very similar. 
\begin{lemma} \label{lem:s0theta}
    For all $m\geq 0$, we have $\undH_{s_0\theta(m,0)}=\bfN_{s_0\theta(m,0)}+v\undH_{s_1\theta'(m-1,0)}$.

\end{lemma}
    
\begin{lemma} \label{lemma explicit for 120 theta }
For all $m\geq 0$ the element $ \undH_{s_1s_2s_0 \theta(m,0)}$ is equal to 
 \begin{equation} \label{eq decomposition N120theta m0}
 \resizebox{0.92\hsize}{!}{$
 \bfN_{s_1s_2s_0 \theta(m,0)} + 
      \displaystyle    v\sum_{i=0}^{\floor{\frac{m-1}{2}}} v^{2i}(\bfN_{\theta(m-2i,0)}+
\bfN_{s_1\theta'(m-1-2i,0)})+
   \sum_{i=1}^{\floor{\frac{m}{2}}}v^{2i}(v^{-1}\bfN_{\theta(m-2i,1)}+\bfN_{s_0\theta(m-2i,0)}).
   $} 
\end{equation}
\end{lemma} 

\begin{proof}
For $m=0$ and $m=1$ the result can be checked by a direct computation. We then assume $m\geq 2$. Definition \ref{defin thetas hat} and Theorem \ref{teo thetas} give us the following identity 
\begin{equation}
    \undH_{\theta(m,0)} = \displaystyle\sum_{i=0}^{ \floor{\frac{m}{2}} } v^{2i}\bfN_{\theta(m-2i,0)}.
\end{equation}
Then, combining \eqref{eq amigos leganos A}, \eqref{eq amigos leganos E} and \eqref{eq amigos leganos I} we get
\begin{equation}
\begin{array}{rl}
\undH_{s_1s_2s_0 \theta(m,0)} & = \left( \undH_{s_1} \undH_{s_2} \undH_{s_0} - \undH_{s_0} -\undH_{s_1} \right)  \undH_{\theta (m,0)} \\
  &  \\
     & = \left( \undH_{s_1} \undH_{s_2} \undH_{s_0} - \undH_{s_0}  \right)\undH_{\theta(m,0)}  -(v+v^{-1}) \undH_{\theta (m,0)} \\
       &  \\
     & =  \left( \undH_{s_1} \undH_{s_2} \undH_{s_0} - \undH_{s_0}  \right) \displaystyle\sum_{i=0}^{ \floor{\frac{m}{2}} } v^{2i}\bfN_{\theta(m-2i,0)} -(v+v^{-1})  \displaystyle\sum_{i=0}^{ \floor{\frac{m}{2}} }v^{2i} \bfN_{\theta(m-2i,0)} \\
     
     & =   \displaystyle\sum_{i=0}^{ \floor{\frac{m}{2}} }\left( \undH_{s_1} \undH_{s_2} \undH_{s_0} - \undH_{s_0}  \right) v^{2i}\bfN_{\theta(m-2i,0)} -(v+v^{-1})  \displaystyle\sum_{i=0}^{ \floor{\frac{m}{2}} }v^{2i} \bfN_{\theta(m-2i,0)}. 
\end{array}
\end{equation}
We assume $m$ is odd, the case $m$ even being similar. We then use Lemma \ref{lemma some identities N de m0} to conclude that 
\begin{equation}  \label{eq two sums}
\resizebox{0.92\hsize}{!}{$	
\begin{array}{rl}
  \undH_{s_1s_2s_0 \theta(m,0)}  = &   \displaystyle\sum_{i=0}^{ \floor{\frac{m}{2}} } v^{2i} \left[  \bfN_{s_1s_2s_0\theta(m-2i,0)} -v^2\bfN_{s_1s_2s_0\theta(m-2(i+1),0)} \right]  +  \\
  &  \\
     &     \displaystyle\sum_{i=0}^{ \floor{\frac{m}{2}} } v^{2i} \left[   v (\bfN_{\theta(m-2i,0)} + \bfN_{s_1\theta'(m-1-2i,0)} ) +   v^2( v^{-1} \bfN_{\theta(m-2(i+1),1)}   +\bfN_{s_0\theta(m-2(i+1),0)} ) \right].
\end{array}
$}
\end{equation}
We notice that the first sum in \eqref{eq two sums} telescopes to $\bfN_{s_1s_2s_0\theta(m,0)}$ and therefore the right-hand side of \eqref{eq two sums} reduces to \eqref{eq decomposition N120theta m0}. We stress that the apparent discrepancy between both expressions is solved by the fact that in \eqref{eq two sums} the term $ v^{-1} \bfN_{\theta(m-2(i+1),1)}   +\bfN_{s_0\theta(m-2(i+1),0)}$ becomes zero when $i=(m-1)/2$. 
\end{proof}

\begin{corollary} \label{corollary recursive 120 theta t in terms of the previous}
	For all $m\geq 1$ we have 
\begin{equation} \label{eq recursive 120 theta t in terms of the previous  AAAA}
\resizebox{0.92\linewidth}{!}{$
\begin{array}{ll}
  \undH_{s_1s_2s_0\theta(m,0)t_m}= 	   & \bfN_{s_1s_2s_0\theta(m,0)t_m}+v\bfN_{s_1s_2s_0\theta(m-1,0)t_{m-1}}+  v^{k+2}\bfN_{s_1s_0}+\displaystyle \sum_{i=0}^{m}v^{i+1}\bfN_{\theta(m-i,0)t_{m-i}}+\\
    &\displaystyle \sum_{i=1}^{m}v^{i}\bfN_{s_1\theta'(m-i,0)t'_{m-i}} +\sum_{i=2}^{m}\lt(v^{i-1}\bfN_{\theta(m-i,1)t_{m-i}}+v^{i}\bfN_{s_0\theta(m-i,0)t_{m-i}}\rt).
\end{array}	
$}
\end{equation}	
Consequently, 
\begin{equation} \label{eq recursive 120 theta t in terms of the previous  BBBB}
\begin{array}{ll}
	\undH_{s_1s_2s_0\theta(m+1,0)t_{m+1}} = &   \bfN_{s_1s_2s_0\theta(m+1,0)t_{m+1}}  + v\undH_{s_1s_2s_0\theta(m,0)t_m}   -v^2\bfN_{s_1s_2s_0\theta(m-1,0)t_{m-1}} + \\
	& \\
	&  v\bfN_{\theta(m+1,0)t_{m+1}} +v\bfN_{s_1\theta'(m,0)t_m' } +v\bfN_{\theta(m-1,1)t_{m-1}} +v^2\bfN_{s_0\theta(m-1,0)t_{m-1}}.
	\end{array}	
\end{equation}
\end{corollary}

\begin{proof}
We first notice that $(s_1s_2s_0\theta(m,0))^{-1}$ belongs to $\mathfrak{C} $	(resp. $\varphi(\mathfrak{C})$) if $m$ is even (resp. odd). Thus, using either \eqref{eq amigos leganos D} or its image under $\varphi$ we conclude that
\begin{equation}
	\undH_{t_m} \undH_{(s_1s_2s_0\theta(m,0))^{-1}} = \undH_{t_m (s_1s_2s_0\theta(m,0))^{-1}}. 
\end{equation}
It follows that $\undH_{s_1s_2s_0\theta(m,0)t_m} = \undH_{s_1s_2s_0\theta(m,0)} \undH_{t_m} $. Therefore we can use the expression for $\undH_{s_1s_2s_0\theta(m,0)}$ given in Lemma \ref{lemma explicit for 120 theta } in order to compute $\undH_{s_1s_2s_0\theta(m,0)t_m}$. Then \eqref{eq recursive 120 theta t in terms of the previous  AAAA} follows by a combination of Lemma \ref{lemma N mult by friend} with $n=1$, Lemma \ref{lemma first mult by N}, Lemma \ref{lemma mult N 0 theta } and its $\varphi$-image and Lemma \ref{lemma mult 120 theta }. Finally, \eqref{eq recursive 120 theta t in terms of the previous  BBBB} is a direct consequence of \eqref{eq recursive 120 theta t in terms of the previous  AAAA}.
\end{proof}

The following result can be proved with similar arguments. For the sake of brevity we omit the proof. 
\begin{lemma} \label{lemma explicit decomposition amigo de theta:  s theta s } 
Let $m\in \bbN$. Then, we have
\begin{align}
    \undH_{s_0\theta(m+1,0)t_{m+1}}   & =\bfN_{s_0\theta(m+1,0)t_{m+1}} + v \bfN_{s_1\theta'(m,0)t_m'} + v\undH_{s_0\theta(m,0)t_m},\\
    \undH_{s_1\theta'(m+1,0)t_{m+1}'}  & =\bfN_{s_1\theta'(m+1,0)t_{m+1}'} + v \bfN_{s_0\theta(m,0)t_m} + v\undH_{s_1\theta'(m,0)t_m'}.
\end{align}
\end{lemma}

\section{Kazhdan-Lusztig polynomials for the thick region}  \label{section thick region}
In this section we compute Kazhdan-Lusztig basis elements corresponding to  elements in the thick region. We recall from Figure \ref{fig: geometric realization of W} that this region is made of four sub-regions: north ($\mcN$), south ($\mcS$), east ($\mcE$) and west ($\mcW$). Recall from       introduction that $ \mcN = \{ x_n \mid n\geq 1 \} \cup  \{ \bar{x}_{3n} \mid n\geq 1 \}$, where $x_n=a_1\cdots a_n$ and $\bar{x}_n=s_1s_2s_0x_{n-3}$, and the $a_i$ are defined by the sequence $
 \{ a_n\}_{n=1}^\infty =(s_1,s_2,s_1,s_0,s_2,s_0,s_1,s_2,s_1,s_0,s_2,s_0,\ldots).   
$
 We also recall that $\mcS=\varphi(\mcN) $, $\mc{E}=\{e_n=s_1x'_n \mid n\geq 1\} \cup \{e'_{3k}\mid k\geq 1\}$, and $\mcW = s_2(\mcE) \cup \{ s_2\}$.
\subsection{North and South}

\begin{definition} \label{definition candidate to be  Hbar x}
We recall from the introduction the notation $f(k)= 3k+1$ and $u_n=s_2x_n$.  For all $k\geq 2$ we define
  \begin{equation}  \label{closed formula for x 3k +1}
     \Hhat_{x_{f(k)}}= \bfN_{x_{f(k)}}+v\bfN_{x_{f(k-1)}} + \lt(\sum_{j=2}^{k-1} v^{j-1}(\bfN_{e_{f(k-j)}}+\bfN_{u_{f(k-j)}})\rt)
 +v^{k-1}\bfN_{s_1s_0}.
 \end{equation}
\end{definition}

\begin{lemma}  \label{lemma from x f(k) to x f(k+1)}
Let $Y=\undH_{s_2}\undH_{s_1}\undH_{s_0}- \undH_{s_0}$. For all $k\geq 2$ we have
\begin{equation}
    \Hhat_{x_{f(k)}}Y = \Hhat_{x_{f(k+1)}} + \Hhat_{x_{f(k-1)}} + \undH_{s_1\theta'(k-2,0)t'_{k-2}} + \undH_{\theta(k-1,0)t_{k-1}} + \undH_{\theta(k-3,0)t_{k-3}} +\undH_{s_1s_2s_0\theta(k-2,0)t_{k-2}}.
\end{equation}
\end{lemma}
\begin{proof}
We proceed by induction on $k$. The case $k=2$ follows by a direct computation. We assume the lemma holds for some  $k \geq 2$.  It follows directly from Definition \ref{definition candidate to be  Hbar x} that 
\begin{equation} \label{decomposition x}
    \Hhat_{x_{f(k+1)}} = \bfN_{x_{f(k+1)}} +v\bfN_{e_{{f(k-1)}}} + v \bfN_{u_{f(k-1)}} - v^2\bfN_{x_{f(k-1)}} +v\Hhat_{x_{f(k)}}.
\end{equation}
Multiplying this equality on the right by $Y$ and using  our inductive hypothesis and  Lemma \ref{lema formulas N x e u}  we obtain that $  \Hhat_{x_{f(k+1)}} Y    $  is equal to the sum of  all the entries of the following matrix.
\begin{equation}
\resizebox{\linewidth}{!}{$
\begin{array}{|c|c|c|c|c|c|}  \hline
   \bfN_{x_{f(k+2)}} &\bfN_{x_{f(k)}} &  \bfN_{\theta(k,0)t_k} & \bfN_{s_1s_2s_0\theta(k-1,0)t_{k-1}} & \bfN_{s_1\theta'(k-1,0)t'_{k-1} }  &  v\bfN_{\theta(k-1,0)t_{k-1} }   \\  \hline
   v\bfN_{u_{f(k)}} & v \bfN_{u_{f(k-2)}} & v\bfN_{\theta(k-3,1) t_{k-3} }& v^2\bfN_{\theta(k-2,0)t_{k-2} } &   \bfN_{\theta(k-2,0) t_{k-2}} &  v^3\bfN_{u_{f(k-2)}}  \\   \hline
v\bfN_{e_{f(k)}}&  v \bfN_{e_{f(k-2)}} & v \bfN_{s_0\theta(k-2,0)t_{k-2} } &   v\bfN_{s_1\theta'(k-2,0) t'_{k-2} } & v^2 \bfN_{s_0\theta(k-3,0)t_{k-3}} &v^2\bfN_{s_1\theta'(k-3,0)t'_{k-3}}\\   \hline
-v^2\bfN_{x_{f(k)}} &  - v^2\bfN_{x_{f(k-2)}} &  - v^2 \bfN_{\theta(k-2,0)t_{k-2} } &-v^2\bfN_{s_1s_2s_0\theta(k-3,0)t_{k-3} } & -v^2\bfN_{s_1\theta'(k-3,0)t'_{k-3}}  & -v^3\bfN_{\theta(k-3,0)t_{k-3} }\\   \hline
v\Hhat_{x_{f(k+1)}} & v \Hhat_{x_{f(k-1)}} &   v \undH_{s_1\theta'(k-2,0)t'_{k-2} } & v \undH_{\theta(k-1,0)t_{k-1}} & v \undH_{\theta(k-3,0)t_{k-3}} &v \undH_{s_1s_2s_0\theta(k-2,0) t_{k-2}}\\  \hline
\end{array}
$}
\end{equation}
We denote this matrix by $A=(A_{ij})$. Using \eqref{decomposition x} we see that
\begin{equation}
	\Hhat_{x_{f(k+2)}} = \sum_{i=1}^5  A_{i1}  \qquad \mbox{ and }  \qquad   \Hhat_{x_{f(k)}}  =  \sum_{i=1}^5  A_{i2}. 
\end{equation}
By Lemma \ref{lemma explicit description for theta s} we have
	$\undH_{\theta(k,0)t_k} = A_{13} + A_{54}$ and $\undH_{\theta(k-2,0)t_{k-2}} = A_{25} + A_{55}$. On the other hand, Lemma \ref{lemma explicit decomposition amigo de theta:  s theta s } implies that $
	\undH_{s_1\theta'(k-1,0)t_{k-1}'} = A_{15} + A_{33} + A_{53}$. We also observe that $
	A_{24}+A_{43} = A_{36}+A_{45} = A_{26} +A_{46} =0$, this last equality being a consequence of the fact that $\theta(k-3,0)t_{k-3}=u_{f({k-2})}$. Finally, we notice that Corollary \ref{corollary recursive 120 theta t in terms of the previous} yields
$
	\undH_{s_1s_2s_0 \theta (k-1,0) t_{k-1}} = A_{14}+A_{56} + A_{44} + A_{16} + A_{34}  + A_{23}
	+A_{35}$. 
Putting all this together we obtain
\begin{equation}
\Hhat_{x_{f(k+1)}}Y = \Hhat_{x_{f(k+2)}} + \Hhat_{x_{f(k)}} + \undH_{s_1\theta'(k-1,0)t'_{k-1}} + \undH_{\theta(k,0)t_{k}} + \undH_{\theta(k-2,0)t_{k-2}} +\undH_{s_1s_2s_0\theta(k-1,0)t_{k-1}}, 	
\end{equation}
as we wanted to show.
\end{proof}

\begin{theorem} \label{corollary hat equal without hat thick wall north}
	For all $k \geq 2$ we have $\undH_{x_{f(k)}} = \Hhat_{x_{f(k)}} $.
\end{theorem}
\begin{proof}
For $k=2$ the result follows by a direct computation. We now fix some $k\geq 2$ and  assume the theorem holds for all $2\leq k' \leq k$. We notice that $G_{x_{f(k+1)}} (\Hhat_{x_{f(k+1)}} ) =1$ and $G_{z} (\Hhat_{x_{f(k+1)}} ) \in v\mathbb{N}[v]$, for all $z< x_{f(k+1)}$. Thus we only need to show that $\Hhat_{x_{f(k+1)}}$ is self-dual. This is an immediate consequence of Lemma \ref{lemma from x f(k) to x f(k+1)} and our inductive hypothesis. 
\end{proof}

\begin{theorem}  \label{corollary recurences north thick region}
 Let $k\geq 2$. We have
 \begin{equation} \label{recurrence thick wall north A}
       \undH_{x_{3k+1}}\undH_{s_2}   = \undH_{x_{3k+2}}.
 \end{equation}
 \begin{equation}\label{recurrence thick wall north B}
     \undH_{x_{3k+2}} \undH_{s_0} = \left\{  \begin{array}{ll}
     \undH_{\bar{x}_{3k+3} }+ \undH_{x_{3k+1}}+\undH_{\theta(k-1,0)} +\undH_{s_1\theta'(k-2,0)}+\undH_{\theta(k-3,0)}    , &  \mbox{if } k \mbox{ is even;}  \\
     \undH_{x_{3k+3}} + \undH_{x_{3k+1}} +\undH_{s_1s_2s_0\theta(k-2,0) }    , &  \mbox{if } k \mbox{ is odd.}
     \end{array}  \right.
 \end{equation}
 \begin{equation}\label{recurrence thick wall north C}
     \undH_{x_{3k+2}} \undH_{s_1} = \left\{  \begin{array}{ll}
 \undH_{x_{3k+3}} + \undH_{x_{3k+1}} + \undH_{s_1s_2s_0\theta(k-2,0)} ,  &  \mbox{if } k \mbox{ is even;}  \\
      \undH_{\bar{x}_{3k+3} }+\undH_{x_{3k+1}}+\undH_{\theta(k-1,0)} +\undH_{s_1\theta'(k-2,0)}+\undH_{\theta(k-3,0)}        , &  \mbox{if } k \mbox{ is odd.}
     \end{array}  \right.
 \end{equation}
\end{theorem}
\begin{proof}
  The claim can be checked directly for $k=2$. From now on we assume $k\geq 3$. An inspection of the explicit formula for $\undH_{x_{3k+1}}$ provided in Definition \ref{definition candidate to be  Hbar x} allows us to conclude that $xs_2>x$ for all  $x \in W $ such that $\mu(x, x_{3k+1}) \neq 0$. Thus the sum in \eqref{eq mult recurrence} is empty and \eqref{recurrence thick wall north A} follows.
  
 We now prove \eqref{recurrence thick wall north B}. We first treat the case $k$ even. Equation \eqref{eq v+v^-1}, Lemma \ref{lemma from x f(k) to x f(k+1)} and Theorem \ref{corollary hat equal without hat thick wall north} imply
   \begin{equation} \label{x por 3 uno}
     \begin{array}{rl}
    \undH_{x_{3k+1}}\undH_{s_2}\undH_{s_0}\undH_{s_1} =  &  (v+v^{-1}) \undH_{x_{3k+1}}+  \undH_{x_{3k+4}} + \undH_{x_{3k-2}} +    \\
      & \\
      &   \undH_{s_1\theta'(k-2,0)s_1} + \undH_{\theta(k-1,0)s_1} + \undH_{\theta(k-3,0)s_1} +\undH_{s_1s_2s_0\theta(k-2,0)s_0}.  
    \end{array}
\end{equation}
On the other hand, we can combine  \eqref{equation y greater than ys}, \eqref{closed formula for x 3k +1} and \eqref{recurrence thick wall north A} to obtain 
\begin{align}
    h_{s_1\theta'(k-2,0),x_{3k+2}}(v) &= v^{-1} h_{s_1\theta'(k-2,0),x_{3k+1}}(v) + h_{e_{3(k-1)},x_{3k+1}} (v)\\
                                               &= v^{-1}(v^2)+ v^3\\
                                               & =v+v^3
\end{align}
\begin{align}
    h_{\theta(k-3,0),x_{3k+2}}(v) &= v^{-1} h_{\theta(k-3,0),x_{3k+1}}(v) + h_{x_{3(k-2)},x_{3k+1}}(v) \\
                                               &= v^{-1}( v^6 + v^4+v^2) + (v^7+v^5+2v^{3} ) \\
                                               & =v +3v^3+2v^5+v^7.
\end{align}
From this we conclude that
\begin{equation} \label{mus A}
    \mu(s_1\theta'(k-2,0), x_{3k+2} )= \mu(\theta(k-3,0), x_{3k+2} )=1.
\end{equation}

We have the following equalities
\begin{align} 
 \undH_{x_{3k+1}}\undH_{s_2}\undH_{s_0}\undH_{s_1} =  &  \undH_{x_{3k+2}}\undH_{s_0}\undH_{s_1}  \\
 =  &  \left( \undH_{\bar{x}_{3k+3} }+ \undH_{x_{3k+1}}+\undH_{\theta(k-1,0)}+  \undH_{s_1\theta'(k-2,0)}+\undH_{\theta(k-3,0)}  \displaystyle + \sum_{y\in Y} m_y \undH_{y} \right) \undH_{s_1} \\
 =   & \left(  \undH_{{x}_{3k+4} } + \undH_{s_1s_2s_0\theta(k-2,0)s_0} + \displaystyle \sum_{z\in Z} n_z\undH_z  \right) +   (v+v^{-1}) \undH_{x_{3k+1}}  + \\
   &  \undH_{\theta(k-1,0)s_1}+\undH_{s_1\theta'(k-2,0)s_1}+\undH_{\theta(k-3,0)s_1}  \displaystyle + \sum_{y\in Y} m_y \undH_{ys_1}. \label{x por 3 dos}
\end{align}
where $m_y,n_z\in \mathbb{N}$,  $  Z=\{ z\in W \mid D_R(y) = \{s_0,s_1 \} \} $ and 
\begin{equation}
    Y = \{ y \in W \mid D_R(y) = \{s_0,s_2 \}  \} \setminus \{ s_1\theta'(k-2,0) ,\theta(k-3,0)   \}.
\end{equation}
Let us explain how to obtain the above equalities. The first equality is a direct consequence of \eqref{recurrence thick wall north A}. For the second one we use \eqref{eq mult recurrence}; first consider the elements in $\lessdot x_{3k+2} $ satisfying $ws_0<w$. Lemma \ref{lem:lessx} together with an easy case analysis show that these elements are $x_{3k+1} $ and $\theta(k-1,0)$. This justifies the occurrence of the terms $\undH_{x_{3k+1}}$ and $\undH_{\theta(k-1,0)} $. On the other hand, \eqref{mus A} explains the occurrence of the terms $\undH_{s_1\theta'(k-2,0)}$ and $\undH_{\theta(k-3,0)}$. Finally, the equality $D_{R}(x_{3k+2})=\{ s_2 \}$ allows us to conclude, via Lemma \ref{useful lemma}, that any other term occurring must have right descent set equal to $\{s_0,s_2 \}$. This explains the appearance of the sum.  We remark that $D_{R}( s_1\theta'(k-2,0)) =D_R(\theta(k-3,0))=\{s_0,s_2\}$ as well (for $k$ even). However the elements $s_1\theta'(k-2,0)$ and $\theta(k-3,0)$ were already considered, which explains the somewhat strange definition of $Y$.

For the last equality, the terms inside the parentheses are justified applying the same arguments used for the second equality to the multiplication $\undH_{\bar{x}_{3k+3}}\undH_{s_1}$. We notice that $D_{R}(\bar x_{3k+3})=\{ s_0\}$ and this explains via Lemma \ref{useful lemma} the definition of $Z$.  The term $(v+v^{-1}) \undH_{x_{3k+1}}$ follows by $\eqref{eq v+v^-1}$. Finally, the occurrence of all the other terms follows by applying  Remark \ref{rem adentro de proof}, which (in its $\varphi$-version) is equivalent to the statement that if  $D_L(w)= \{s_0,s_2\}$ then $\undH_{s_1}\undH_{w}=\undH_{s_1w}$. Using inverses to move from left to right, the latter is equivalent to saying that if  $D_R(w)= \{s_0,s_2\}$ then  $\undH_{w}\undH_{s_1}=\undH_{ws_1}$. This is the version we need to conclude.

A comparison of the expressions for $ \undH_{x_{3k+1}}\undH_{s_2}\undH_{s_0}\undH_{s_1}$ given in \eqref{x por 3 uno} and \eqref{x por 3 dos} yields
\begin{equation}
   \undH_{x_{3k-2}} =   \displaystyle \sum_{z\in Z} n_z\undH_z +   \sum_{y\in Y} m_y \undH_{ys_1} .
\end{equation}
Suppose that $x_{3k-2}= ys_1$ for some $y\in Y$. Multiplying by $s_1$ on the right we obtain that $\bar x_{3k-3}\in Y$ since $k$ is even. However, $D_R(\bar x_{3k-3})=\{s_0\}$ and we reach a contradiction. We conclude that $m_y=0$ for all $y\in Y$. Therefore,
\begin{equation}
   \undH_{x_{3k+2}} \undH_{s_0}   =   \undH_{\bar{x}_{3k+3} }+ \undH_{x_{3k+1}}+\undH_{\theta(k-1,0)} +\undH_{s_1\theta'(k-2,0)}+\undH_{\theta(k-3,0)} ,
\end{equation}
as we wanted to show. This completes the proof of \eqref{recurrence thick wall north B} for $k $ even.

We now assume $k$ is odd.  As before, Equation \eqref{eq v+v^-1}, Lemma \ref{lemma from x f(k) to x f(k+1)} and Theorem \ref{corollary hat equal without hat thick wall north} imply
   \begin{equation} \label{X por tres A}
     \begin{array}{rl}
    \undH_{x_{3k+1}}\undH_{s_2}\undH_{s_0}\undH_{s_1} =  &  (v+v^{-1}) \undH_{x_{3k+1}}+  \undH_{x_{3k+4}} + \undH_{x_{3k-2}} +    \\
      & \\
      &   \undH_{s_1\theta'(k-2,0)s_0} + \undH_{\theta(k-1,0)s_0} + \undH_{\theta(k-3,0)s_0} +\undH_{s_1s_2s_0\theta(k-2,0)s_1}.  
    \end{array}
    \end{equation}
On the other hand, arguing as in \eqref{x por 3 dos} we obtain
\begin{equation}\label{X por tres B}
\begin{array}{rl}
 \undH_{x_{3k+1}}\undH_{s_2}\undH_{s_0}\undH_{s_1} =  &  \undH_{x_{3k+2}}\undH_{s_0}\undH_{s_1}  \\
 &  \\
                                        =  &  \left(  \undH_{x_{3k+3}} + \undH_{x_{3k+1}} +\undH_{s_1s_2s_0\theta(k-2,0) }  \displaystyle + \sum_{y\in Y} m_y \undH_{y} \right) \undH_{s_1} \\
                                        & \\
                                        =   &   \undH_{{x}_{3k+4} } +  \undH_{\theta(k-1,0)s_0}  + \displaystyle \sum_{z\in Z} n_z\undH_z +   (v+v^{-1}) \undH_{x_{3k+1}}  + \\
                                        &  \\
                                    & \undH_{s_1s_2s_0\theta(k-2,0)s_1}  \displaystyle + \sum_{y\in Y} m_y \undH_{ys_1}  .
                                        
\end{array}
\end{equation}
where $m_y,n_z\in \mathbb{N}$,  $  Z=\{ z\in W \mid D_R(y) = \{s_0,s_1 \} \}$ and $ Y = \{ y \in W \mid D_R(y) = \{s_0,s_2 \}  \}$. 

A comparison of the expressions for $ \undH_{x_{3k+1}}\undH_{s_2}\undH_{s_0}\undH_{s_1}$ given in \eqref{X por tres A}  and \eqref{X por tres B} yields
\begin{equation}
   \undH_{x_{3k-2}} +   \undH_{s_1\theta'(k-2,0)s_0} + \undH_{\theta(k-3,0)s_0} =   \displaystyle \sum_{z\in Z} n_z\undH_z +   \sum_{y\in Y} m_y \undH_{ys_1} .
\end{equation}
As before we want to show that the sum over $Y$ is zero. Suppose that $x_{3k-2}=ys_1$ for some $y\in Y$. Multiplying by $s_1$ on the right we have that $x_{3k-3}\in Y$ since $k$ is odd. Since $D_R(x_{3k-3})=\{ s_0\}$ we obtain a contradiction.  We conclude that term $\undH_{x_{3k-2}}$ does not come from the sum over $Y$. We now suppose that $s_1\theta'(k-2,0)s_0=ys_1$ for some $y\in Y$. Multiplying by $s_1$ on the right we obtain
\begin{equation}
y= s_1\theta'(k-2,0)s_0s_1 = s_1\theta'(k-3,0)s_1s_2s_1s_0s_1 = s_1\theta'(k-3,0)s_1s_2s_0.
\end{equation}
It follows that $s_1\theta'(k-3,0)s_1s_2s_0\in Y$. However, $D_R(s_1\theta'(k-3,0)s_1s_2s_0)=\{s_0\}$ and we reach a contradiction. We conclude that term $\undH_{s_1\theta'(k-2,0)s_0}$ does not come from the sum over $Y$. Similarly, we can prove that $ \undH_{\theta(k-3,0)s_0} $ does not come from the sum over $Y$. It follows that the sum over $Y$ is zero. Therefore, 
\begin{equation}
\undH_{x_{3k+2}}\undH_{s_0}	=  \undH_{x_{3k+3}} + \undH_{x_{3k+1}} +\undH_{s_1s_2s_0\theta(k-2,0) },
\end{equation}
as we wanted to show. 

The proof of \eqref{recurrence thick wall north C} is dealt with similarity. The details are left to the reader. 
\end{proof}

Theorem \ref{corollary hat equal without hat thick wall north} and Theorem \ref{corollary recurences north thick region} give formulas for all Kazhdan-Lusztig basis  elements indexed by elements of $\mcN$ of length greater than six; the remainder can be obtained by direct calculation. As we already pointed out at the beginning of this section, $\mcS =\varphi(\mcN)$, therefore acting by $\varphi$ in each one of these formulas we obtain the corresponding formulas for all the Kazhdan-Lusztig basis elements indexed by elements of $\mcS$.

\subsection{East and  West.}  \label{section east and west}

\begin{definition}  \label{defin explicitit east}
	For $k \geq 1 $ we define the following elements
	\begin{align}
    \Hhat_{e_{3k+1}}&=\sum_{i=0}^k v^i\bfN_{e_{3(k-i)+1}}, \label{eq:e3k1} \\
    \Hhat_{e_{3k+2}}&=\bfN_{e_{3k+2}}+\sum_{i=1}^k 2v^i\bfN_{e_{3(k-i)+2}},\label{eq:e3k2} \\
    \Hhat_{e_{3k+3}}&=\bfN_{e_{3k+3}}+v\bfN_{e'_{3k}}+\sum_{i=1}^k v^i\bfN_{e_{3(k-i)+1}}+\sum_{i=1}^k v^{i}\bfN_{\vfi^{i}(s_0\theta(k-i-1,0))}. 
\end{align}
\end{definition}

\begin{lemma} \label{lemma multiplying tres on the left east}
	Let $k\geq 1$. We have
	\begin{equation}\label{eq east to west 3k +1}
		\Hhat_{e_{f(k)}}\undH_{s_2}\undH_{t_k}\undH_{t_{k}'} = \Hhat_{e_{f(k+1)}} + (v+v^{-1})\Hhat_{e_{f(k)}} + \Hhat_{e_{f(k-1)}}  + \undH_{s_{0}\theta(k-1,0)t_{k}' } + \undH_{s_1\theta'(k-1,0)t_k}.
	\end{equation}
\end{lemma}

\begin{proof}
We proceed by induction on $k$. A direct computation gives us the result for $k=1$. We fix $k\geq 1$ and assume that \eqref{eq east to west 3k +1} holds for all $k'\leq k$. By Definition \ref{defin explicitit east} we have 
\begin{equation} \label{east 3k +1}
\Hhat_{e_{f(k+1)}} = \bfN_{e_{f(k+1)}} +v \Hhat_{e_{f(k)}}  .	
\end{equation}
 Then, multiplying this last equation on the right by $\undH_{s_2}\undH_{t_k}\undH_{t_{k}'}$ and using Lemma \ref{lema formulas N x e u}, Lemma \ref{lemma N por Hs cuando baja} and our inductive hypothesis we conclude that $\Hhat_{e_{f(k+1)}}\undH_{s_2}\undH_{t_k}\undH_{t_{k}'} $ is equal to the sum of the entries of the following matrix 
\begin{equation}
	\begin{array}{|c|c|c|c|c|} \hline
	\bfN_{e_{f(k+2)}}	&  (v+v^{-1}) \bfN_{e_{f(k+1)}}   & \bfN_{e_{f(k)}}        & \bfN_{s_0\theta(k,0)t_k}   &  \bfN_{s_1\theta'(k,0)t_k'}   \\ \hline
	v\Hhat_{e_{f(k+1)}} &  v(v+v^{-1}) \Hhat_{e_{f(k)}}   & v\Hhat_{e_{f(k-1)}}    & v\bfN_{s_1\theta'(k-1,0)t_{k-1}'}    &  v\bfN_{s_0\theta(k-1,0)t_{k-1}}     \\ \hline
0	& 0 & 0 & v\undH_{s_0\theta(k-1,0)t_k'}   & v\undH_{s_1\theta' (k-1,0)t_k}  \\ \hline  
	\end{array}
\end{equation}
Bearing in mind that $t'_{k}=t_{k-1}$, Lemma \ref{lemma explicit decomposition amigo de theta:  s theta s }  together with \eqref{east 3k +1} show that if we add the entries of the above matrix by columns then we obtain
\begin{equation}
			\Hhat_{e_{f(k+1)}}\undH_{s_2}\undH_{t_k}\undH_{t_{k}'} = \Hhat_{e_{f(k+2)}}+ (v+v^{-1})\Hhat_{e_{f(k-1)}} + \Hhat_{e_{f(k)}}  + \undH_{s_{0}\theta(k,0)t_{k+1}' } + \undH_{s_1\theta'(k,0)t_{k+1}}.
\end{equation}
This completes our induction.
\end{proof}

\begin{theorem}  \label{theorem east todo junto}
	For all $k \geq 1$ and $j\in \{1,2,3\}$ we have 
	\begin{equation}\label{hat equal no-hat east}
	    	\Hhat_{e_{3k+j}} = \undH_{e_{3k+j}}. 
	\end{equation}
\end{theorem}

\begin{proof}
We split the proof of \eqref{hat equal no-hat east} in three cases in accordance with the value of $j$.\\\\
\textbf{Case A.} $(j=1)$ For $k=1$ the result can be checked by hand. We assume $k>1$.
It follows directly from the definition of $\Hhat_{e_{3k+1}}$ that $G_{e_{3k+1}}(\Hhat_{e_{3k+1}})=1$ and $G_z(\Hhat_{e_{3k+1}})\in v\mathbb{N}[v]$, for all $z <e_{3k+1}$. Therefore, we only need to check that $\Hhat_{e_{3k+1}}$ is self-dual. The latter follows by an inductive argument using Lemma \ref{lemma multiplying tres on the left east}.\\\\
\textbf{Case B.} $(j=2)$ We will prove this identity by showing that $c(\Hhat_{e_{3k+2}}) = c(\undH_{e_{3k+2}})$ and that $\undH_{e_{3k+2}} \geqh \Hhat_{e_{3k+2}}$. Inspection of the formula (\ref{eq:e3k1}) and the set $\lessdot e_{3k+1}$ from Lemma~\ref{lemma coatomos e} allows us to conclude the identity $\undH_{e_{3k+2}}=\undH_{e_{3k+1}}\undH_{s_2}$, applying (\ref{eq mult recurrence}) and \textbf{Case A}. Therefore, the first step reduces to showing that $c(\hat\undH_{e_{3k+2}}) = 2c(\undH_{e_{3k+1}})$ which follows from a straightforward computation using Lemma \ref{lem:abeu}.

In order to prove that $\undH_{e_{3k+2}} \geqh \Hhat_{e_{3k+2}}$, by degree reasons and the monotonicity of $\undH_{e_{3k+2}}$, it is enough to check that  $G_{e_{3(k-i)+2}}(\undH_{e_{3k+2}})=G_{e_{3(k-i)+2}}(\undH_{e_{3k+1}}\undH_{s_2}) \geq 2v^i$ for $1 \leq i \leq k$. In view of (\ref{eq:e3k1}), we see that
\[G_{e_{3(k-i)+2}}(\undH_{e_{3k+1}}\undH_{s_2})\geq G_{e_{3(k-i)+2}}((v^i\bfH_{e_{3(k-i)+1}}+v^{i+1}\bfH_{e_{3(k-i)+2}})\undH_{s_2})=2v^i\]
as desired.\\\\
\textbf{Case C.} $(j=3)$ Note that $e_{3k+3}=e_{3k+2}t_k$. Proceeding as in \textbf{Case B}, inspection of formula~(\ref{eq:e3k2}) and Lemma~\ref{lemma coatomos e} allows us to conclude
\begin{equation} \label{eq:e3k3 recurrence}
    \undH_{e_{3k+3}}=\undH_{e_{3k+2}}\undH_{t_k}-\undH_{e_{3k+1}}-\undH_{s_0\theta(k-1,0)}.
\end{equation}
Combining the facts that $c(\undH_{e_{3k+2}}\undH_{t_k})=c(\undH_{e_{3k+1}}\undH_{s_2}\undH_{t_k})=4c(\undH_{e_{3k+1}})$ and $c(\undH_{s_0\theta(k-1,0)})=2c(\undH_{\theta(k-1,0)})$, a straightforward calculation using Lemma~\ref{lem:abeu}, Lemma~\ref{lemma cardinality big region not fundamental}, Theorem~\ref{teo thetas} and Corollary~\ref{lem:abx} shows that $c(\undH_{e_{3k+3}})=c(\hat{\undH}_{e_{3k+3}})$. 

Next we show that $\undH_{e_{3k+3}}\geqh \hat{\undH}_{e_{3k+3}}$. By degree reasons and monotonicity of $\undH_{e_{3k+3}}$, it suffices to prove that $G_{e'_{3k}}(\undH_{e_{3k+3}})\geq v$, $G_{e_{1}}(\undH_{e_{3k+3}})\geq v^k$, and for $1\leq i \leq k-1$, 
\begin{equation}\label{eq:e3k3pf}
        \undH_{e_{3k+3}}\geqh v^i\bfN_{e_{3(k-i)+1}}+v^i\bfN_{\vfi^{i}(s_0\theta(k-1-i,0))}.
    \end{equation}
We will only show (\ref{eq:e3k3pf}), as the other items follow by similar (and easier) arguments.

Observing that 
\begin{equation}
    (\leq e_{3(k-i)+1})\cap (\leq \vfi^i(s_0\theta(k-1-i,0)))=\leq\vfi^i(e'_{3(k-i)}),
\end{equation} the proof of (\ref{eq:e3k3pf}) reduces to showing that 
\begin{align}
G_{e_{3(k-i)+1}}(\undH_{e_{3k+3}})& \geq v^i, \\
G_{\vfi^i(s_0\theta(k-1-i,0))}(\undH_{e_{3k+3}})&\geq v^i,\qquad \text{and}\\
G_{\vfi^i(e'_{3(k-i)})}(\undH_{e_{3k+3}})&\geq 2v^{i+1}. 
\end{align}
The term $2v^i\bfN_{e_{3(k-i)+2}}$ in (\ref{eq:e3k2}) has three summands of interest: $2v^{i+1}\bfH_{e_{3(k-i)+1}}$, $2v^{i+1}\bfH_{\vfi^i(s_0\theta(k-1-i,0))}$, and $2v^{i+2}\bfH_{\vfi^i(e'_{3(k-i)})}$. Multiplying these terms by $\undH_{t_k}$ contributes $2v^i\bfH_{e_{3(k-i)+1}}$, $2v^i\bfH_{\vfi^i(s_0\theta(k-1-i,0))}$,  and $4v^{i+1}\bfH_{\vfi^i(e'_{3(k-i)})}$ to $\undH_{e_{3k+3}}$. Then using (\ref{eq:e3k1}) and Lemma~\ref{lem:s0theta} with (\ref{eq:e3k3 recurrence}), the subtraction of terms $v^i\bfN_{e_{3(k-i)+1}}$ and $v^i\bfN_{\vfi^i(s_0\theta(k-1-i,0))}$ has the only effect in the degrees of interest, leaving us with the desired coefficients. 
\end{proof} 

Theorem \ref{theorem east todo junto}  provides formulas for all the Kazhdan-Lusztig basis elements indexed by elements located in $\mcE$ of length at least five. Smaller elements can be computed directly. The following theorem provides formulas for all the Kazhdan-Lusztig basis elements indexed by elements located in $\mcW$, thus completing our description of Kazhdan-Lusztig basis for the thick region.

\begin{theorem}
    Let $n\geq 1$. Then $ \undH_{w_{n}}= \undH_{s_2}\undH_{e_n}$. 
\end{theorem}
\begin{proof}
We notice that $D_{L}(e_n)=\{s_0,s_1\}$. Then,  Lemma \ref{useful lemma} implies
\begin{equation}\label{eq east to west}
   \undH_{s_2}\undH_{e_n}=  \undH_{w_{n}} + \sum_{ \substack{ z \lessdot e_n \\  s_2z <z } } \mu (z, e_{n}) \undH_{z}. 
\end{equation}
However, a case-by-case analysis reveals that if $z\lessdot e_n $ then $ s_2z >z$. It follows that the sum in \eqref{eq east to west} is empty and the theorem follows. 
\end{proof}


\section{Kazhdan-Lusztig basis in the thin region}  \label{section thin region}
In this section we present conjectural formulas for $\undH_w$ for $w$ located in the thin region. We begin by describing the elements in this region. We denote by $\mathcal{NW}$ (resp. $\mathcal {NE}$, $\mathcal {SW}$ and $\mathcal {SE}$) the sub-region of the thin region formed by triangles located to the northwest (resp. northeast, southwest and southeast) of the identity triangle. We define $d_n$ as the product of the first $n$ symbols of the infinite string $s_2s_1s_2s_0s_2s_1s_2s_0s_2s_1s_2s_0 \dots$. The elements $d_n$ comprise the thin wall that goes towards the northwest in our convention, and the elements $d'_n=\varphi(\mathcal{NW})$ are the southwestward thin wall (see Figure \ref{fig: geometric realization of W}). In other words, $\mathcal{NW}=\{ d_n \mid  n \geq 3 \}$ and $\mathcal{SW}=\{ d'_n \mid n \geq 3 \}$. Finally, we define $\bar{d}_n=s_0 d_n$ and with these describe the remaining two thin walls as $\mathcal{SE}=\{ \bar{d}_n \mid  n \geq 3 \}$ and $\mathcal{NE}=\{ \bar{d}'_n \mid  n \geq 3 \}$.

The thin region coincides very nearly with the two-sided cell of $W$ consisting of the elements with unique reduced expression. Wang  has shown \cite{wang2011kazhdan} that there exist $u\in W$ for which $\mu(u,w)\neq 0$ for infinitely many $w$ in the thin region, meaning that the $W$-graph of the group is not locally finite. According to our conjecture, there are in fact infinitely many such $u$. This indicates an additional level of complexity for the Kazhdan-Lusztig basis in the thin region, reflected in the formulas below.

We first define some notation: For $x,z \in W$, let 
$$\mb{D}_{x}^{z}:=\sum_{\substack{w\leq x \\ w \not\leq z}} v^{l(x)-l(w)} \mb{H}_w $$
Notice that $\mb{D}_{x}^{z}$ is a truncation of $\bfN_x$. We further define $\bfU_{x}:=\bfN_{x}+\bfD_{x'}^x=\bfU_{x'}$. 

\begin{conjecture}
For all $k\geq 1$,
\begin{enumerate}
    \item \begin{align}
\undH_{d_{4k+3}}&=\bfN_{d_{4k+3}}+v\bfD_{\theta(0,k-1)}^{s_2s_1\theta'(0,k-2)s_1s_2}+(v+v^3)\undH_{\theta(0,k-2)}+v\undH_{s_2e_2}+v\sum_{i=3}^{k}\lt(\undH_{\theta(0,k-i)}+\undH_{\theta'(0,k-i)}\rt)
\\
&\quad + v\sum_{i=2}^{k}\lt(\bfU_{s_2s_1\theta'(0,k-i)s_1s_2}+\bfU_{\theta(0,k-i)s_0s_2}+\undH_{s_2s_0\theta(0,k-i)}+\undH_{s_2s_1\theta'(0,k-i)}\rt),
\end{align}

\item \begin{align}
    \undH_{d_{4k}}\undH_{s_2}&=\undH_{d_{4k+1}}+\undH_{d_{4k-1}}+\sum_{i=0}^{k-2} \undH_{\theta(1,i)}+\sum_{i=0}^{k-3} \undH_{s_2s_0\theta(1,i)}.\\
    \undH_{d_{4k+1}}\undH_{s_1}&=\undH_{d_{4k+2}}+\sum_{i=0}^{k-2}\undH_{\theta(0,i)}+\sum_{i=0}^{k-3} \undH_{s_2s_0\theta(0,i)}.\\
    \undH_{d_{4k+2}}\undH_{s_2}   &  = \undH_{d_{4k+3}}+\undH_{d_{4k+1}}+ \sum_{i=0}^{k-2} \undH_{s_2s_1\theta'(1,i)}+\sum_{i=0}^{k-3} \undH_{\theta'(1,i)}    \\
    \undH_{d_{4k+3}}\undH_{s_0}&=\undH_{d_{4k+4}}+\sum_{i=0}^{k-2} \undH_{s_2s_1\theta'(0,i)}+\sum_{i=0}^{k-3} \undH_{\theta'(0,i)}.\\
\end{align}

\item \begin{align}
    \undH_{s_0}\undH_{d_{4k}}&=\undH_{\bar{d}_{4k}}+\sum_{i=0}^{k-3}\lt(\undH_{\theta'(0,i)s_1s_2s_0}+\undH_{\theta'(0,i)s_1}\rt) \\
      \undH_{s_0}\undH_{d_{4k+1}}&=\undH_{\bar{d}_{4k+1}}+\sum_{i=0}^{k-2}\undH_{\theta'(0,i)}+\sum_{i=0}^{k-3}\undH_{\theta'(0,i)s_1s_2} \\
       \undH_{s_0}\undH_{d_{4k+2}}&=\undH_{\bar{d}_{4k+2}}+\undH_{\theta'(0,k-2)s_1}+\sum_{i=0}^{k-3}\lt(2\undH_{\theta'(0,i)s_1}+\undH_{\theta'(1,i)}\rt) \\
           \undH_{s_0}\undH_{d_{4k+3}}&=\undH_{\bar{d}_{4k+3}}+\sum_{i=0}^{k-2}\undH_{\theta'(0,i)s_1s_2}+\sum_{i=0}^{k-3}\undH_{\theta'(0,i)} 
\end{align}
\end{enumerate}
\end{conjecture}
These formulas above have been checked up to $k=5$. 
Note again that these formulas cover both the northwest and southeast walls of the thin region, while basis elements from the other walls of this region can be obtained by applying the automorphism $\vfi$.

\bibliography{mybibfile}
\bibliographystyle{alpha}

\end{document}

%% file: dibujos.tex
\def\dibujoetreskmasdos{
\psscalebox{1.0 1.0} 
{
\begin{pspicture}[shift=-10](5,6)(15,16)

\pspolygon*[linearc=0,linecolor=mygray2](10.5,6.5)(6.5,10.5)(10.5,14.5)(14.5,10.5)
\pspolygon*[linearc=0,linecolor=mygray4](10.5,7.5)(7.5,10.5)(10.5,13.5)(13.5,10.5)
\pspolygon*[linearc=0,linecolor=mygray7](10.5,8.5)(8.5,10.5)(10.5,12.5)(12.5,10.5)
\pspolygon*[linearc=0,linecolor=mygray9](10,10)(9.5,10.5)(10,11)(10.5,11.5)(11,11)(11.5,10.5)(11,10)(10.5,9.5)

\multips(0,0)(2,0){5}{\psline[linewidth=1pt,linestyle=solid, linecolor=cb-burgundy](5,16)(6,15)(7,16)}
\multips(0,0)(0,-2){5}{\psline[linewidth=1pt,linestyle=solid, linecolor=cb-burgundy](5,16)(6,15)(5,14)}
\multips(0,0)(2,0){5}{\psline[linewidth=1pt,linestyle=solid, linecolor=cb-burgundy](5,6)(6,7)(7,6)}
\multips(0,0)(0,-2){5}{\psline[linewidth=1pt,linestyle=solid, linecolor=cb-burgundy](15,16)(14,15)(15,14)}

\psgrid[unit=1,gridcolor=cb-blue, subgridcolor=cb-blue,subgriddiv=1, subgridwidth=0.8pt,subgriddots=15, gridlabels=0](5,6)(15,16)
\rput{45}(10,6){\psgrid[unit=1.414213562373,subgridcolor=cb-burgundy, gridcolor=cb-green-lime,subgriddiv=2, subgridwidth=0.8pt,subgriddots=0,gridlabels=0](0,0)(5,5) }
\rput{45}(12,6){\psgrid[unit=1.414213562373,subgridcolor=cb-burgundy, gridcolor=cb-green-lime,subgriddiv=2, subgridwidth=0.8pt,subgriddots=0,gridlabels=0](0,0)(3,7) }
\rput{45}(14,6){\psgrid[unit=1.414213562373,subgridcolor=cb-burgundy, gridcolor=cb-green-lime,subgriddiv=2, subgridwidth=0.8pt,subgriddots=0,gridlabels=0](0,0)(1,9) }
\rput{45}(8,6){\psgrid[unit=1.414213562373,subgridcolor=cb-burgundy,  gridcolor=cb-green-lime,subgriddiv=2, subgridwidth=0.8pt,subgriddots=0,gridlabels=0](0,0)(7,3) }
\rput{45}(6,6){\psgrid[unit=1.414213562373,subgridcolor=cb-burgundy,  gridcolor=cb-green-lime,subgriddiv=2, subgridwidth=0.8pt,subgriddots=0,gridlabels=0](0,0)(9,1) }
\rput(10.25,10.5){\scalebox{1.5}{$\bullet$}}
\end{pspicture}
}
}

\def\dibujoetreskmasuno{
\psscalebox{1.0 1.0} 
{
\begin{pspicture}[shift=-10](5,6)(15,16)

\pspolygon*[linearc=0,linecolor=mygray2](10,6)(10,7)(9,7)(9,8)(8,8)(8,9)(7,9)(7,10)(6,10)(6,11)(7,11)(7,12)(8,12)(8,13)(9,13)(9,14)(10,14)(10,15)(11,15)(11,14)(12,14)(12,13)(13,13)(13,12)(14,12)(14,11)(15,11)(15,10)(14,10)(14,9)(13,9)(13,8)(12,8)(12,7)(11,7)(11,6) 
\pspolygon*[linearc=0,linecolor=mygray4](10,7)(10,8)(9,8)(9,9)(8,9)(8,10)(7,10)(7,11)(8,11)(8,12)(9,12)(9,13)(10,13)(10,14)(11,14)(11,13)(12,13)(12,12)(13,12)(13,11)(14,11)(14,10)(13,10)(13,9)(12,9)(12,8)(11,8)(11,7) 
\pspolygon*[linearc=0,linecolor=mygray7](10,8)(10,9)(9,9)(9,10)(8,10)(8,11)(9,11)(9,12)(10,12)(10,13)(11,13)(11,12)(12,12)(12,11)(13,11)(13,10)(12,10)(12,9)(11,9)(11,8) 
\pspolygon*[linearc=0,linecolor=mygray9](9,10)(9,11)(10,11)(10,12)(11,12)(11,11)(12,11)(12,10)(11,10)(11,9)(10,9)(10,10) 

\multips(0,0)(2,0){5}{\psline[linewidth=1pt,linestyle=solid, linecolor=cb-burgundy](5,16)(6,15)(7,16)}
\multips(0,0)(0,-2){5}{\psline[linewidth=1pt,linestyle=solid, linecolor=cb-burgundy](5,16)(6,15)(5,14)}
\multips(0,0)(2,0){5}{\psline[linewidth=1pt,linestyle=solid, linecolor=cb-burgundy](5,6)(6,7)(7,6)}
\multips(0,0)(0,-2){5}{\psline[linewidth=1pt,linestyle=solid, linecolor=cb-burgundy](15,16)(14,15)(15,14)}

\psgrid[unit=1,gridcolor=cb-blue, subgridcolor=cb-blue,subgriddiv=1, subgridwidth=0.8pt,subgriddots=15, gridlabels=0](5,6)(15,16)
\rput{45}(10,6){\psgrid[unit=1.414213562373,subgridcolor=cb-burgundy, gridcolor=cb-green-lime,subgriddiv=2, subgridwidth=0.8pt,subgriddots=0,gridlabels=0](0,0)(5,5) }
\rput{45}(12,6){\psgrid[unit=1.414213562373,subgridcolor=cb-burgundy, gridcolor=cb-green-lime,subgriddiv=2, subgridwidth=0.8pt,subgriddots=0,gridlabels=0](0,0)(3,7) }
\rput{45}(14,6){\psgrid[unit=1.414213562373,subgridcolor=cb-burgundy, gridcolor=cb-green-lime,subgriddiv=2, subgridwidth=0.8pt,subgriddots=0,gridlabels=0](0,0)(1,9) }
\rput{45}(8,6){\psgrid[unit=1.414213562373,subgridcolor=cb-burgundy,  gridcolor=cb-green-lime,subgriddiv=2, subgridwidth=0.8pt,subgriddots=0,gridlabels=0](0,0)(7,3) }
\rput{45}(6,6){\psgrid[unit=1.414213562373,subgridcolor=cb-burgundy,  gridcolor=cb-green-lime,subgriddiv=2, subgridwidth=0.8pt,subgriddots=0,gridlabels=0](0,0)(9,1) }
\rput(10.25,10.5){\scalebox{1.5}{$\bullet$}}
\end{pspicture}
}
}

\def\dibujoetresk{
\psscalebox{1.0 1.0} 
{
\begin{pspicture}[shift=-10](5,6)(15,16)

\pspolygon*[linearc=0,linecolor=mygray2](10,6)(10,7)(9,7)(9,8)(8,8)(8,9)(7,9)(7,10)(6,10)(11,15)(11,14)(12,14)(12,13)(13,13)(13,12)(14,12)(14,11)(15,11) 
\pspolygon*[linearc=0,linecolor=mygray4](10,7)(10,8)(9,8)(9,9)(8,9)(8,10)(7,10)(11,14)(11,13)(12,13)(12,12)(13,12)(13,11)(14,11) 
\pspolygon*[linearc=0,linecolor=mygray7](10,8)(10,9)(9,9)(9,10)(8,10)(11,13)(11,12)(12,12)(12,11)(13,11) 
\pspolygon*[linearc=0,linecolor=mygray9](9,10)(11,12)(11,11)(12,11)(10,9)(10,10) 

\multips(0,0)(2,0){5}{\psline[linewidth=1pt,linestyle=solid, linecolor=cb-burgundy](5,16)(6,15)(7,16)}
\multips(0,0)(0,-2){5}{\psline[linewidth=1pt,linestyle=solid, linecolor=cb-burgundy](5,16)(6,15)(5,14)}
\multips(0,0)(2,0){5}{\psline[linewidth=1pt,linestyle=solid, linecolor=cb-burgundy](5,6)(6,7)(7,6)}
\multips(0,0)(0,-2){5}{\psline[linewidth=1pt,linestyle=solid, linecolor=cb-burgundy](15,16)(14,15)(15,14)}

\psgrid[unit=1,gridcolor=cb-blue, subgridcolor=cb-blue,subgriddiv=1, subgridwidth=0.8pt,subgriddots=15, gridlabels=0](5,6)(15,16)
\rput{45}(10,6){\psgrid[unit=1.414213562373,subgridcolor=cb-burgundy, gridcolor=cb-green-lime,subgriddiv=2, subgridwidth=0.8pt,subgriddots=0,gridlabels=0](0,0)(5,5) }
\rput{45}(12,6){\psgrid[unit=1.414213562373,subgridcolor=cb-burgundy, gridcolor=cb-green-lime,subgriddiv=2, subgridwidth=0.8pt,subgriddots=0,gridlabels=0](0,0)(3,7) }
\rput{45}(14,6){\psgrid[unit=1.414213562373,subgridcolor=cb-burgundy, gridcolor=cb-green-lime,subgriddiv=2, subgridwidth=0.8pt,subgriddots=0,gridlabels=0](0,0)(1,9) }
\rput{45}(8,6){\psgrid[unit=1.414213562373,subgridcolor=cb-burgundy,  gridcolor=cb-green-lime,subgriddiv=2, subgridwidth=0.8pt,subgriddots=0,gridlabels=0](0,0)(7,3) }
\rput{45}(6,6){\psgrid[unit=1.414213562373,subgridcolor=cb-burgundy,  gridcolor=cb-green-lime,subgriddiv=2, subgridwidth=0.8pt,subgriddots=0,gridlabels=0](0,0)(9,1) }
\rput(10.25,10.5){\scalebox{1.5}{$\bullet$}}
\end{pspicture}
}
}

\def\dibujoxtresk{
\psscalebox{1.0 1.0} 
{
\begin{pspicture}[shift=-10](5,6)(15,16)

\pspolygon*[linearc=0,linecolor=mygray2](10,7)(6,11)(7,11)(7,12)(8,12)(8,13)(9,13)(9,14)(10,14)(10,15)(14,11)
\pspolygon*[linearc=0,linecolor=mygray4](10,8)(7,11)(8,11)(8,12)(9,12)(9,13)(10,13)(10,14)(13,11)
\pspolygon*[linearc=0,linecolor=mygray7](10,9)(8,11)(9,11)(9,12)(10,12)(10,13)(12,11)
\pspolygon*[linearc=0,linecolor=mygray9](10,10)(9,11)(10,11)(10,12)(11,11)

\multips(0,0)(2,0){5}{\psline[linewidth=1pt,linestyle=solid, linecolor=cb-burgundy](5,16)(6,15)(7,16)}
\multips(0,0)(0,-2){5}{\psline[linewidth=1pt,linestyle=solid, linecolor=cb-burgundy](5,16)(6,15)(5,14)}
\multips(0,0)(2,0){5}{\psline[linewidth=1pt,linestyle=solid, linecolor=cb-burgundy](5,6)(6,7)(7,6)}
\multips(0,0)(0,-2){5}{\psline[linewidth=1pt,linestyle=solid, linecolor=cb-burgundy](15,16)(14,15)(15,14)}

\psgrid[unit=1,gridcolor=cb-blue, subgridcolor=cb-blue,subgriddiv=1, subgridwidth=0.8pt,subgriddots=15, gridlabels=0](5,6)(15,16)
\rput{45}(10,6){\psgrid[unit=1.414213562373,subgridcolor=cb-burgundy, gridcolor=cb-green-lime,subgriddiv=2, subgridwidth=0.8pt,subgriddots=0,gridlabels=0](0,0)(5,5) }
\rput{45}(12,6){\psgrid[unit=1.414213562373,subgridcolor=cb-burgundy, gridcolor=cb-green-lime,subgriddiv=2, subgridwidth=0.8pt,subgriddots=0,gridlabels=0](0,0)(3,7) }
\rput{45}(14,6){\psgrid[unit=1.414213562373,subgridcolor=cb-burgundy, gridcolor=cb-green-lime,subgriddiv=2, subgridwidth=0.8pt,subgriddots=0,gridlabels=0](0,0)(1,9) }
\rput{45}(8,6){\psgrid[unit=1.414213562373,subgridcolor=cb-burgundy,  gridcolor=cb-green-lime,subgriddiv=2, subgridwidth=0.8pt,subgriddots=0,gridlabels=0](0,0)(7,3) }
\rput{45}(6,6){\psgrid[unit=1.414213562373,subgridcolor=cb-burgundy,  gridcolor=cb-green-lime,subgriddiv=2, subgridwidth=0.8pt,subgriddots=0,gridlabels=0](0,0)(9,1) }
\rput(10.25,10.5){\scalebox{1.5}{$\bullet$}}
\end{pspicture}
}
}

\def\dibujoxtreskmasdos{
\psscalebox{1.0 1.0} 
{
\begin{pspicture}[shift=-10](5,6)(15,16)

\pspolygon*[linearc=0,linecolor=mygray2](10.5,15.5)(14.5,11.5)(14,11)(14.5,10.5)(10.5,6.5)(10,7)(9.5,6.5)(5.5,10.5) 
\pspolygon*[linearc=0,linecolor=mygray4](10.5,14.5)(13.5,11.5)(13,11)(13.5,10.5)(10.5,7.5)(10,8)(9.5,7.5)(6.5,10.5) 
\pspolygon*[linearc=0,linecolor=mygray7](10.5,13.5)(12.5,11.5)(12,11)(12.5,10.5)(10.5,8.5)(10,9)(9.5,8.5)(7.5,10.5) 
\pspolygon*[linearc=0,linecolor=mygray9](11,10)(11.5,10.5)(11,11)(11.5,11.5)(11,12)(10.5,12.5)(10,12)(9,11)(8.5,10.5)(9,10)(9.5,9.5)(10,10)(10.5,9.5) 

\multips(0,0)(2,0){5}{\psline[linewidth=1pt,linestyle=solid, linecolor=cb-burgundy](5,16)(6,15)(7,16)}
\multips(0,0)(0,-2){5}{\psline[linewidth=1pt,linestyle=solid, linecolor=cb-burgundy](5,16)(6,15)(5,14)}
\multips(0,0)(2,0){5}{\psline[linewidth=1pt,linestyle=solid, linecolor=cb-burgundy](5,6)(6,7)(7,6)}
\multips(0,0)(0,-2){5}{\psline[linewidth=1pt,linestyle=solid, linecolor=cb-burgundy](15,16)(14,15)(15,14)}

\psgrid[unit=1,gridcolor=cb-blue, subgridcolor=cb-blue,subgriddiv=1, subgridwidth=0.8pt,subgriddots=15, gridlabels=0](5,6)(15,16)
\rput{45}(10,6){\psgrid[unit=1.414213562373,subgridcolor=cb-burgundy, gridcolor=cb-green-lime,subgriddiv=2, subgridwidth=0.8pt,subgriddots=0,gridlabels=0](0,0)(5,5) }
\rput{45}(12,6){\psgrid[unit=1.414213562373,subgridcolor=cb-burgundy, gridcolor=cb-green-lime,subgriddiv=2, subgridwidth=0.8pt,subgriddots=0,gridlabels=0](0,0)(3,7) }
\rput{45}(14,6){\psgrid[unit=1.414213562373,subgridcolor=cb-burgundy, gridcolor=cb-green-lime,subgriddiv=2, subgridwidth=0.8pt,subgriddots=0,gridlabels=0](0,0)(1,9) }
\rput{45}(8,6){\psgrid[unit=1.414213562373,subgridcolor=cb-burgundy,  gridcolor=cb-green-lime,subgriddiv=2, subgridwidth=0.8pt,subgriddots=0,gridlabels=0](0,0)(7,3) }
\rput{45}(6,6){\psgrid[unit=1.414213562373,subgridcolor=cb-burgundy,  gridcolor=cb-green-lime,subgriddiv=2, subgridwidth=0.8pt,subgriddots=0,gridlabels=0](0,0)(9,1) }
\rput(10.25,10.5){\scalebox{1.5}{$\bullet$}}
\end{pspicture}
}
}

\def\dibujoxtreskmasuno{
\psscalebox{1.0 1.0} 
{
\begin{pspicture}[shift=-10](5,6)(15,16)

\pspolygon*[linearc=0,linecolor=mygray2](10,14)(10,15)(11,15)(11,14)(12,14)(12,13)(13,13)(13,12)(14,12)(14,10)(13,10)(13,9)(12,9)(12,8)(11,8)(11,7)(9,7)(9,8)(8,8)(8,9)(7,9)(7,10)(6,10)(6,11)(7,11)(7,12)(8,12)(8,13)(9,13)(9,14) 
\pspolygon*[linearc=0,linecolor=mygray4](10,13)(10,14)(11,14)(11,13)(12,13)(12,12)(13,12)(13,10)(12,10)(12,9)(11,9)(11,8)(9,8)(9,9)(8,9)(8,10)(7,10)(7,11)(8,11)(8,12)(9,12)(9,13) 
\pspolygon*[linearc=0,linecolor=mygray7](10,13)(11,13)(11,12)(12,12)(12,10)(11,10)(11,9)(9,9)(9,10)(8,10)(8,11)(9,11)(9,12)(10,12) 
\pspolygon*[linearc=0,linecolor=mygray9](11,10)(11,12)(10,12)(10,11)(9,11)(9,10) 

\multips(0,0)(2,0){5}{\psline[linewidth=1pt,linestyle=solid, linecolor=cb-burgundy](5,16)(6,15)(7,16)}
\multips(0,0)(0,-2){5}{\psline[linewidth=1pt,linestyle=solid, linecolor=cb-burgundy](5,16)(6,15)(5,14)}
\multips(0,0)(2,0){5}{\psline[linewidth=1pt,linestyle=solid, linecolor=cb-burgundy](5,6)(6,7)(7,6)}
\multips(0,0)(0,-2){5}{\psline[linewidth=1pt,linestyle=solid, linecolor=cb-burgundy](15,16)(14,15)(15,14)}

\psgrid[unit=1,gridcolor=cb-blue, subgridcolor=cb-blue,subgriddiv=1, subgridwidth=0.8pt,subgriddots=15, gridlabels=0](5,6)(15,16)
\rput{45}(10,6){\psgrid[unit=1.414213562373,subgridcolor=cb-burgundy, gridcolor=cb-green-lime,subgriddiv=2, subgridwidth=0.8pt,subgriddots=0,gridlabels=0](0,0)(5,5) }
\rput{45}(12,6){\psgrid[unit=1.414213562373,subgridcolor=cb-burgundy, gridcolor=cb-green-lime,subgriddiv=2, subgridwidth=0.8pt,subgriddots=0,gridlabels=0](0,0)(3,7) }
\rput{45}(14,6){\psgrid[unit=1.414213562373,subgridcolor=cb-burgundy, gridcolor=cb-green-lime,subgriddiv=2, subgridwidth=0.8pt,subgriddots=0,gridlabels=0](0,0)(1,9) }
\rput{45}(8,6){\psgrid[unit=1.414213562373,subgridcolor=cb-burgundy,  gridcolor=cb-green-lime,subgriddiv=2, subgridwidth=0.8pt,subgriddots=0,gridlabels=0](0,0)(7,3) }
\rput{45}(6,6){\psgrid[unit=1.414213562373,subgridcolor=cb-burgundy,  gridcolor=cb-green-lime,subgriddiv=2, subgridwidth=0.8pt,subgriddots=0,gridlabels=0](0,0)(9,1) }
\rput(10.25,10.5){\scalebox{1.5}{$\bullet$}}
\end{pspicture}
}
}

\def\dibujothetafour{
\psscalebox{1.0 1.0} 
{
\begin{pspicture}[shift=-10](4,5)(16,17)
\pspolygon*[linearc=0,linecolor=mygray2](8,16)(9,16)(9,17)(10,16)(11,17)(12,16)(13,17)(13,16)(14,16)(14,15)(15,15)(15,14)(16,14)(15,13)(16,12)(15,11)(16,10)(15,9)(16,8)(15,8)(15,7)(14,7)(14,6)(13,6)(13,5)(12,6)(11,5)(10,6)(9,5)(8,6)(7,5)(7,6)(6,6)(6,7)(5,7)(5,8)(4,8)(5,9)(4,10)(5,11)(4,12)(5,12)(5,13)(6,13)(6,14)(7,14)(7,15)(8,15)
\pspolygon*[linearc=0,linecolor=mygray4](12,16)(11.5,16.5)(11,16)(10.5,16.5)(10,16)(9.5,16.5)(9,16)(8,15)(4.5,11.5)(5,11)(4.5,10.5)(5,10)(4.5,9.5)(5,9)(4.5,8.5)(7.5,5.5)(8,6)(8.5,5.5)(9,6)(9.5,5.5)(10,6)(10.5,5.5)(11,6)(11.5,5.5)(12,6)(12.5,5.5)(15.5,8.5)(15,9)(15.5,9.5)(15,10)(15.5,10.5)(15,11)(15.5,11.5)(15,12)(15.5,12.5)(15,13)(15.5,13.5)(12.5,16.5)
\pspolygon*[linearc=0,linecolor=mygray7](8,14)(8,15)(9,15)(9,16)(13,16)(13,15)(14,15)(14,14)(15,14)(15,13)(15,12)(15,9)(15,8)(14,8)(14,7)(13,7)(13,6)(12,6)(8,6)(7,6)(7,7)(6,7)(6,8)(5,8)(5,12)(6,12)(6,13)(7,13)(7,14)
\pspolygon*[linearc=0,linecolor=mygray9](12,6)(15,9)(14,10)(15,11)(14,12)(15,13)(12,16)(11,15)(10,16)(5,11)(6,10)(5,9)(8,6)(9,7)(10,6)(11,7)

\multips(0,0)(2,0){6}{\psline[linewidth=1pt,linestyle=solid, linecolor=cb-green-lime](4,16)(5,17)(6,16)}
\multips(0,0)(0,-2){6}{\psline[linewidth=1pt,linestyle=solid, linecolor=cb-green-lime](5,17)(4,16)(5,15)}
\multips(0,0)(2,0){6}{\psline[linewidth=1pt,linestyle=solid, linecolor=cb-green-lime](4,6)(5,5)(6,6)}
\multips(0,0)(0,-2){6}{\psline[linewidth=1pt,linestyle=solid, linecolor=cb-green-lime](15,17)(16,16)(15,15)}
\rput{45}(10,5){\psgrid[unit=1.414213562373,gridcolor=cb-burgundy,gridwidth=1pt, subgridcolor=cb-green-lime,subgriddiv=2, subgridwidth=1pt,subgriddots=0,gridlabels=0](0,0)(6,6) }
\rput{45}(12,5){\psgrid[unit=1.414213562373,gridcolor=cb-burgundy,gridwidth=1pt, subgridcolor=cb-green-lime,subgriddiv=2, subgridwidth=1pt,subgriddots=0,gridlabels=0](0,0)(4,8) }
\rput{45}(14,5){\psgrid[unit=1.414213562373,gridcolor=cb-burgundy,gridwidth=1pt, subgridcolor=cb-green-lime,subgriddiv=2, subgridwidth=1pt,subgriddots=0,gridlabels=0](0,0)(2,10) }
\rput{45}(16,5){\psgrid[unit=1.414213562373,gridcolor=cb-burgundy,gridwidth=1pt, subgridcolor=cb-green-lime,subgriddiv=2, subgridwidth=1pt,subgriddots=0,gridlabels=0](0,0)(0,12) }
\rput{45}(8,5){\psgrid[unit=1.414213562373,gridcolor=cb-burgundy, gridwidth=1pt, subgridcolor=cb-green-lime,subgriddiv=2, subgridwidth=1pt,subgriddots=0,gridlabels=0](0,0)(8,4) }
\rput{45}(6,5){\psgrid[unit=1.414213562373,gridcolor=cb-burgundy, gridwidth=1pt, subgridcolor=cb-green-lime,subgriddiv=2, subgridwidth=1pt,subgriddots=0,gridlabels=0](0,0)(10,2) }
\rput{45}(4,5){\psgrid[unit=1.414213562373,gridcolor=cb-burgundy, gridwidth=1pt, subgridcolor=cb-green-lime,subgriddiv=2, subgridwidth=1pt,subgriddots=0,gridlabels=0](0,0)(12,0) }
\psgrid[unit=1,gridcolor=cb-blue, gridwidth=1pt, gridcolor=cb-blue,subgriddiv=1, subgridwidth=1pt,subgriddots=15, gridlabels=0](4,5)(16,17)
\rput(10.25,10.5){\scalebox{1.5}{$\bullet$}}
\end{pspicture}
}
}

\def\dibujothetathree{
\psscalebox{1.0 1.0} 
{
\begin{pspicture}[shift=-10](4,5)(16,17)
\pspolygon*[linearc=0,linecolor=mygray2](8,16)(9,15)(10,16)(11,15)(12,16)(12,15)(13,15)(13,14)(14,14)(14,13)(15,13)(15,12)(16,12)(15,11)(16,10)(15,9)(16,8)(15,8)(15,7)(14,7)(14,6)(13,6)(13,5)(12,6)(11,5)(10,6)(9,5)(8,6)(7,5)(7,6)(6,6)(6,7)(5,7)(5,8)(4,8)(5,9)(4,10)(5,11)(4,12)(5,12)(5,13)(6,13)(6,14)(7,14)(7,15)(8,15)
\pspolygon*[linearc=0,linecolor=mygray4](12,15)(11.5,15.5)(11,15)(10.5,15.5)(10,15)(9.5,15.5)(9,15)(8.5,15.5)(8,15)(4.5,11.5)(5,11)(4.5,10.5)(5,10)(4.5,9.5)(5,9)(4.5,8.5)(7.5,5.5)(8,6)(8.5,5.5)(9,6)(9.5,5.5)(10,6)(10.5,5.5)(11,6)(11.5,5.5)(12,6)(12.5,5.5)(15.5,8.5)(15,9)(15.5,9.5)(15,10)(15.5,10.5)(15,11)(15.5,11.5)(15,12)(14,13)
\pspolygon*[linearc=0,linecolor=mygray7](8,14)(8,15)(12,15)(12,14)(13,14)(13,13)(14,13)(14,12)(15,12)(15,9)(15,8)(14,8)(14,7)(13,7)(13,6)(12,6)(8,6)(7,6)(7,7)(6,7)(6,8)(5,8)(5,12)(6,12)(6,13)(7,13)(7,14)
\pspolygon*[linearc=0,linecolor=mygray9](12,6)(15,9)(14,10)(15,11)(11,15)(10,14)(9,15)(5,11)(6,10)(5,9)(8,6)(9,7)(10,6)(11,7)
\multips(0,0)(2,0){6}{\psline[linewidth=1pt,linestyle=solid, linecolor=cb-green-lime](4,16)(5,17)(6,16)}
\multips(0,0)(0,-2){6}{\psline[linewidth=1pt,linestyle=solid, linecolor=cb-green-lime](5,17)(4,16)(5,15)}
\multips(0,0)(2,0){6}{\psline[linewidth=1pt,linestyle=solid, linecolor=cb-green-lime](4,6)(5,5)(6,6)}
\multips(0,0)(0,-2){6}{\psline[linewidth=1pt,linestyle=solid, linecolor=cb-green-lime](15,17)(16,16)(15,15)}
\rput{45}(10,5){\psgrid[unit=1.414213562373,gridcolor=cb-burgundy,gridwidth=1pt, subgridcolor=cb-green-lime,subgriddiv=2, subgridwidth=1pt,subgriddots=0,gridlabels=0](0,0)(6,6) }
\rput{45}(12,5){\psgrid[unit=1.414213562373,gridcolor=cb-burgundy,gridwidth=1pt, subgridcolor=cb-green-lime,subgriddiv=2, subgridwidth=1pt,subgriddots=0,gridlabels=0](0,0)(4,8) }
\rput{45}(14,5){\psgrid[unit=1.414213562373,gridcolor=cb-burgundy,gridwidth=1pt, subgridcolor=cb-green-lime,subgriddiv=2, subgridwidth=1pt,subgriddots=0,gridlabels=0](0,0)(2,10) }
\rput{45}(16,5){\psgrid[unit=1.414213562373,gridcolor=cb-burgundy,gridwidth=1pt, subgridcolor=cb-green-lime,subgriddiv=2, subgridwidth=1pt,subgriddots=0,gridlabels=0](0,0)(0,12) }
\rput{45}(8,5){\psgrid[unit=1.414213562373,gridcolor=cb-burgundy, gridwidth=1pt, subgridcolor=cb-green-lime,subgriddiv=2, subgridwidth=1pt,subgriddots=0,gridlabels=0](0,0)(8,4) }
\rput{45}(6,5){\psgrid[unit=1.414213562373,gridcolor=cb-burgundy, gridwidth=1pt, subgridcolor=cb-green-lime,subgriddiv=2, subgridwidth=1pt,subgriddots=0,gridlabels=0](0,0)(10,2) }
\rput{45}(4,5){\psgrid[unit=1.414213562373,gridcolor=cb-burgundy, gridwidth=1pt, subgridcolor=cb-green-lime,subgriddiv=2, subgridwidth=1pt,subgriddots=0,gridlabels=0](0,0)(12,0) }
\psgrid[unit=1,gridcolor=cb-blue, gridwidth=1pt, gridcolor=cb-blue,subgriddiv=1, subgridwidth=1pt,subgriddots=15, gridlabels=0](4,5)(16,17)
\rput(10.25,10.5){\scalebox{1.5}{$\bullet$}}
\end{pspicture}
}
}

\def\dibujothetatwo{
\psscalebox{1.0 1.0} 
{
\begin{pspicture}[shift=-10](4,5)(16,17)
\pspolygon*[linearc=0,linecolor=mygray2](8,16)(9,15)(10,16)(11,15)(12,16)(12,15)(13,15)(13,14)(14,14)(14,13)(15,13)(15,12)(16,12)(15,11)(16,10)(15,9)(16,8)(15,8)(15,7)(14,7)(14,6)(13,6)(13,5)(12,6)(11,5)(10,6)(9,5)(9,6)(8,6)(8,7)(7,7)(7,8)(6,8)(6,9)(5,9)(6,10)(5,11)(6,12)(5,13)(6,13)(6,14)(7,14)(7,15)(8,15)
\pspolygon*[linearc=0,linecolor=mygray4](12,15)(11.5,15.5)(11,15)(10.5,15.5)(10,15)(9.5,15.5)(9,15)(8.5,15.5)(8,15)(6,13)(5.5,12.5)(6,12)(5.5,11.5)(6,11)(5.5,10.5)(6,10)(5.5,9.5)(6,9)(8,7)(9,6)(9.5,5.5)(10,6)(10.5,5.5)(11,6)(11.5,5.5)(12,6)(12.5,5.5)(15.5,8.5)(15,9)(15.5,9.5)(15,10)(15.5,10.5)(15,11)(15.5,11.5)(15,12)(14,13)
\pspolygon*[linearc=0,linecolor=mygray7](8,14)(8,15)(12,15)(12,14)(13,14)(13,13)(14,13)(14,12)(15,12)(15,9)(15,8)(14,8)(14,7)(13,7)(13,6)(9,6)(9,7)(8,7)(8,8)(7,8)(7,9)(6,9)(6,13)(7,13)(7,14)
\pspolygon*[linearc=0,linecolor=mygray9](11,7)(12,6)(15,9)(14,10)(15,11)(11,15)(10,14)(9,15)(6,12)(7,11)(6,10)(10,6)
\multips(0,0)(2,0){6}{\psline[linewidth=1pt,linestyle=solid, linecolor=cb-green-lime](4,16)(5,17)(6,16)}
\multips(0,0)(0,-2){6}{\psline[linewidth=1pt,linestyle=solid, linecolor=cb-green-lime](5,17)(4,16)(5,15)}
\multips(0,0)(2,0){6}{\psline[linewidth=1pt,linestyle=solid, linecolor=cb-green-lime](4,6)(5,5)(6,6)}
\multips(0,0)(0,-2){6}{\psline[linewidth=1pt,linestyle=solid, linecolor=cb-green-lime](15,17)(16,16)(15,15)}
\rput{45}(10,5){\psgrid[unit=1.414213562373,gridcolor=cb-burgundy,gridwidth=1pt, subgridcolor=cb-green-lime,subgriddiv=2, subgridwidth=1pt,subgriddots=0,gridlabels=0](0,0)(6,6) }
\rput{45}(12,5){\psgrid[unit=1.414213562373,gridcolor=cb-burgundy,gridwidth=1pt, subgridcolor=cb-green-lime,subgriddiv=2, subgridwidth=1pt,subgriddots=0,gridlabels=0](0,0)(4,8) }
\rput{45}(14,5){\psgrid[unit=1.414213562373,gridcolor=cb-burgundy,gridwidth=1pt, subgridcolor=cb-green-lime,subgriddiv=2, subgridwidth=1pt,subgriddots=0,gridlabels=0](0,0)(2,10) }
\rput{45}(16,5){\psgrid[unit=1.414213562373,gridcolor=cb-burgundy,gridwidth=1pt, subgridcolor=cb-green-lime,subgriddiv=2, subgridwidth=1pt,subgriddots=0,gridlabels=0](0,0)(0,12) }
\rput{45}(8,5){\psgrid[unit=1.414213562373,gridcolor=cb-burgundy, gridwidth=1pt, subgridcolor=cb-green-lime,subgriddiv=2, subgridwidth=1pt,subgriddots=0,gridlabels=0](0,0)(8,4) }
\rput{45}(6,5){\psgrid[unit=1.414213562373,gridcolor=cb-burgundy, gridwidth=1pt, subgridcolor=cb-green-lime,subgriddiv=2, subgridwidth=1pt,subgriddots=0,gridlabels=0](0,0)(10,2) }
\rput{45}(4,5){\psgrid[unit=1.414213562373,gridcolor=cb-burgundy, gridwidth=1pt, subgridcolor=cb-green-lime,subgriddiv=2, subgridwidth=1pt,subgriddots=0,gridlabels=0](0,0)(12,0) }
\psgrid[unit=1,gridcolor=cb-blue, gridwidth=1pt, gridcolor=cb-blue,subgriddiv=1, subgridwidth=1pt,subgriddots=15, gridlabels=0](4,5)(16,17)
\rput(10.25,10.5){\scalebox{1.5}{$\bullet$}}
\end{pspicture}
}
}

\def\dibujothetaone{
\psscalebox{1.0 1.0} 
{
\begin{pspicture}[shift=-10](4,5)(16,17)
\pspolygon*[linearc=0,linecolor=mygray2](8,16)(9,15)(10,16)(11,15)(12,16)(12,15)(13,15)(13,14)(14,14)(14,13)(15,13)(14,12)(15,11)(14,10)(15,9)(14,9)(14,8)(13,8)(13,7)(12,7)(12,6)(11,7)(10,6)(9,7)(8,6)(8,7)(7,7)(7,8)(6,8)(6,9)(5,9)(6,10)(5,11)(6,12)(5,13)(6,13)(6,14)(7,14)(7,15)(8,15)
\pspolygon*[linearc=0,linecolor=mygray4](12,15)(11.5,15.5)(11,15)(10.5,15.5)(10,15)(9.5,15.5)(9,15)(8.5,15.5)(8,15)(6,13)(5.5,12.5)(6,12)(5.5,11.5)(6,11)(5.5,10.5)(6,10)(5.5,9.5)(6,9)(8,7)(8.5,6.5)(9,7)(9.5,6.5)(10,7)(10.5,6.5)(11,7)(11.5,6.5)(12,7)(14,9)(14.5,9.5)(14,10)(14.5,10.5)(14,11)(14.5,11.5)(14,12)(14.5,12.5)(14,13)
\pspolygon*[linearc=0,linecolor=mygray7](8,14)(8,15)(12,15)(12,14)(13,14)(13,13)(14,13)(14,9)(13,9)(13,8)(12,8)(12,7)(8,7)(8,8)(7,8)(7,9)(6,9)(6,13)(7,13)(7,14)
\pspolygon*[linearc=0,linecolor=mygray9](9,7)(10,8)(11,7)(14,10)(13,11)(14,12)(11,15)(10,14)(9,15)(6,12)(7,11)(6,10)

\multips(0,0)(2,0){6}{\psline[linewidth=1pt,linestyle=solid, linecolor=cb-green-lime](4,16)(5,17)(6,16)}
\multips(0,0)(0,-2){6}{\psline[linewidth=1pt,linestyle=solid, linecolor=cb-green-lime](5,17)(4,16)(5,15)}
\multips(0,0)(2,0){6}{\psline[linewidth=1pt,linestyle=solid, linecolor=cb-green-lime](4,6)(5,5)(6,6)}
\multips(0,0)(0,-2){6}{\psline[linewidth=1pt,linestyle=solid, linecolor=cb-green-lime](15,17)(16,16)(15,15)}
\rput{45}(10,5){\psgrid[unit=1.414213562373,gridcolor=cb-burgundy,gridwidth=1pt, subgridcolor=cb-green-lime,subgriddiv=2, subgridwidth=1pt,subgriddots=0,gridlabels=0](0,0)(6,6) }
\rput{45}(12,5){\psgrid[unit=1.414213562373,gridcolor=cb-burgundy,gridwidth=1pt, subgridcolor=cb-green-lime,subgriddiv=2, subgridwidth=1pt,subgriddots=0,gridlabels=0](0,0)(4,8) }
\rput{45}(14,5){\psgrid[unit=1.414213562373,gridcolor=cb-burgundy,gridwidth=1pt, subgridcolor=cb-green-lime,subgriddiv=2, subgridwidth=1pt,subgriddots=0,gridlabels=0](0,0)(2,10) }
\rput{45}(16,5){\psgrid[unit=1.414213562373,gridcolor=cb-burgundy,gridwidth=1pt, subgridcolor=cb-green-lime,subgriddiv=2, subgridwidth=1pt,subgriddots=0,gridlabels=0](0,0)(0,12) }
\rput{45}(8,5){\psgrid[unit=1.414213562373,gridcolor=cb-burgundy, gridwidth=1pt, subgridcolor=cb-green-lime,subgriddiv=2, subgridwidth=1pt,subgriddots=0,gridlabels=0](0,0)(8,4) }
\rput{45}(6,5){\psgrid[unit=1.414213562373,gridcolor=cb-burgundy, gridwidth=1pt, subgridcolor=cb-green-lime,subgriddiv=2, subgridwidth=1pt,subgriddots=0,gridlabels=0](0,0)(10,2) }
\rput{45}(4,5){\psgrid[unit=1.414213562373,gridcolor=cb-burgundy, gridwidth=1pt, subgridcolor=cb-green-lime,subgriddiv=2, subgridwidth=1pt,subgriddots=0,gridlabels=0](0,0)(12,0) }
\psgrid[unit=1,gridcolor=cb-blue, gridwidth=1pt, gridcolor=cb-blue,subgriddiv=1, subgridwidth=1pt,subgriddots=15, gridlabels=0](4,5)(16,17)
\rput(10.25,10.5){\scalebox{1.5}{$\bullet$}}
\end{pspicture}
}
}

\def\dibujothetas{
\psscalebox{1.0 1.0} 
{
\begin{pspicture}[shift=-10](2,3)(18,19)
\pspolygon*[linearc=0,linecolor=cb-yellow](8,18)(3,13)(4,12)(3,11)(4,10)(3,9)(8,4)(9,5)(10,4)(11,5)(12,4)(17,9)(16,10)(17,11)(16,12)(17,13)(12,18)(11,17)(10,18)(9,17)
\pspolygon*[linearc=0,linecolor=cb-clay](9,17)(4,12)(5,11)(4,10)(9,5)(10,6)(11,5)(16,10)(15,11)(16,12)(11,17)(10,16)
\pspolygon*[linearc=0,linecolor=mygray1](10,6)(15,11)(10,16)(5,11)
\pspolygon*[linearc=0,linecolor=mygray3](10,7)(14,11)(10,15)(6,11)
\pspolygon*[linearc=0,linecolor=mygray5](10,8)(13,11)(10,14)(7,11)
\pspolygon*[linearc=0,linecolor=mygray7](10,9)(12,11)(10,13)(8,11)
\pspolygon*[linearc=0,linecolor=mygray9](10,10)(11,11)(10,12)(9,11)
\multips(0,0)(2,0){8}{\psline[linewidth=1pt,linestyle=solid, linecolor=cb-green-lime](2,18)(3,19)(4,18)}
\multips(0,0)(0,-2){8}{\psline[linewidth=1pt,linestyle=solid, linecolor=cb-green-lime](3,19)(2,18)(3,17)}
\multips(0,0)(2,0){8}{\psline[linewidth=1pt,linestyle=solid, linecolor=cb-green-lime](2,4)(3,3)(4,4)}
\multips(0,0)(0,-2){8}{\psline[linewidth=1pt,linestyle=solid, linecolor=cb-green-lime](17,19)(18,18)(17,17)}
\psgrid[unit=1,gridcolor=cb-blue, subgridcolor=cb-blue,subgriddiv=1, subgridwidth=0.8pt,subgriddots=15, gridlabels=0](2,3)(18,19)
\rput{45}(10,3){\psgrid[unit=1.414213562373,gridcolor=cb-burgundy, subgridcolor=cb-green-lime,subgriddiv=2, subgridwidth=0.8pt,subgriddots=0,gridlabels=0](0,0)(8,8) }
\rput{45}(12,3){\psgrid[unit=1.414213562373,gridcolor=cb-burgundy, subgridcolor=cb-green-lime,subgriddiv=2, subgridwidth=0.8pt,subgriddots=0,gridlabels=0](0,0)(6,10) }
\rput{45}(14,3){\psgrid[unit=1.414213562373,gridcolor=cb-burgundy, subgridcolor=cb-green-lime,subgriddiv=2, subgridwidth=0.8pt,subgriddots=0,gridlabels=0](0,0)(4,12) }
\rput{45}(16,3){\psgrid[unit=1.414213562373,gridcolor=cb-burgundy, subgridcolor=cb-green-lime,subgriddiv=2, subgridwidth=0.8pt,subgriddots=0,gridlabels=0](0,0)(2,14) }
\rput{45}(18,3){\psgrid[unit=1.414213562373, gridcolor=cb-burgundy, subgridcolor=cb-green-lime,subgriddiv=2, subgridwidth=0.8pt,subgriddots=0,gridlabels=0](0,0)(0,16) }
\rput{45}(8,3){\psgrid[unit=1.414213562373,gridcolor=cb-burgundy, subgridcolor=cb-green-lime,subgriddiv=2, subgridwidth=0.8pt,subgriddots=0,gridlabels=0](0,0)(10,6) }
\rput{45}(6,3){\psgrid[unit=1.414213562373,gridcolor=cb-burgundy, subgridcolor=cb-green-lime,subgriddiv=2, subgridwidth=0.8pt,subgriddots=0,gridlabels=0](0,0)(12,4) }
\rput{45}(4,3){\psgrid[unit=1.414213562373,gridcolor=cb-burgundy, subgridcolor=cb-green-lime,subgriddiv=2, subgridwidth=0.8pt,subgriddots=0,gridlabels=0](0,0)(14,2) }
\rput{45}(2,3){\psgrid[unit=1.414213562373,gridcolor=cb-burgundy, subgridcolor=cb-green-lime,subgriddiv=2, subgridwidth=0.8pt,subgriddots=0,gridlabels=0](0,0)(16,0) }
\rput(10.25,10.5){\scalebox{1.5}{$\bullet$}}
\end{pspicture}
}
}

\def\IntTess{
\psscalebox{1.0 1.0} 
{
\begin{pspicture}[shift=-6](4,5)(16,17)
\pspolygon*[linearc=0,linecolor=mygray9](4,5)(16,5)(16,17)(4,17)
\pspolygon*[linearc=0,linecolor=mygray7](4,17)(4,16)(15,5)(16,5)
\pspolygon*[linearc=0,linecolor=mygray7](4,5)(5,5)(16,16)(16,17)
\pspolygon*[linearc=0,linecolor=mygray5](10,5)(11,5)(11,17)(10,17)
\pspolygon*[linearc=0,linecolor=mygray5](4,10)(4,11)(16,11)(16,10)
\pspolygon*[linearc=0,linecolor=cb-yellow](10,10)(10.5,10.5)(10,11)
\multips(0,0)(2,0){6}{\psline[linewidth=1pt,linestyle=solid, linecolor=cb-green-lime](4,16)(5,17)(6,16)}
\multips(0,0)(0,-2){6}{\psline[linewidth=1pt,linestyle=solid, linecolor=cb-green-lime](5,17)(4,16)(5,15)}
\multips(0,0)(2,0){6}{\psline[linewidth=1pt,linestyle=solid, linecolor=cb-green-lime](4,6)(5,5)(6,6)}
\multips(0,0)(0,-2){6}{\psline[linewidth=1pt,linestyle=solid, linecolor=cb-green-lime](15,17)(16,16)(15,15)}
\rput{45}(10,5){\psgrid[unit=1.414213562373,gridcolor=cb-burgundy,gridwidth=1pt, subgridcolor=cb-green-lime,subgriddiv=2, subgridwidth=1pt,subgriddots=0,gridlabels=0](0,0)(6,6) }
\rput{45}(12,5){\psgrid[unit=1.414213562373,gridcolor=cb-burgundy,gridwidth=1pt, subgridcolor=cb-green-lime,subgriddiv=2, subgridwidth=1pt,subgriddots=0,gridlabels=0](0,0)(4,8) }
\rput{45}(14,5){\psgrid[unit=1.414213562373,gridcolor=cb-burgundy,gridwidth=1pt, subgridcolor=cb-green-lime,subgriddiv=2, subgridwidth=1pt,subgriddots=0,gridlabels=0](0,0)(2,10) }
\rput{45}(16,5){\psgrid[unit=1.414213562373,gridcolor=cb-burgundy,gridwidth=1pt, subgridcolor=cb-green-lime,subgriddiv=2, subgridwidth=1pt,subgriddots=0,gridlabels=0](0,0)(0,12) }
\rput{45}(8,5){\psgrid[unit=1.414213562373,gridcolor=cb-burgundy, gridwidth=1pt, subgridcolor=cb-green-lime,subgriddiv=2, subgridwidth=1pt,subgriddots=0,gridlabels=0](0,0)(8,4) }
\rput{45}(6,5){\psgrid[unit=1.414213562373,gridcolor=cb-burgundy, gridwidth=1pt, subgridcolor=cb-green-lime,subgriddiv=2, subgridwidth=1pt,subgriddots=0,gridlabels=0](0,0)(10,2) }
\rput{45}(4,5){\psgrid[unit=1.414213562373,gridcolor=cb-burgundy, gridwidth=1pt, subgridcolor=cb-green-lime,subgriddiv=2, subgridwidth=1pt,subgriddots=0,gridlabels=0](0,0)(12,0) }
\psgrid[unit=1,gridcolor=cb-blue, gridwidth=1pt, gridcolor=cb-blue,subgriddiv=1, subgridwidth=1pt,subgriddots=15, gridlabels=0](4,5)(16,17)
\rput(10.25,10.5){\scalebox{1.5}{$\bullet$}}
\end{pspicture}
}
}